\documentclass[10pt]{article}

\usepackage[utf8]{inputenc}
\usepackage[english]{babel}
\usepackage{lmodern}
\usepackage{microtype}
\usepackage{latexsym}
\usepackage{amsthm}
\usepackage{amsmath}
\usepackage{amssymb}
\usepackage{enumerate}
\usepackage{xcolor}
\usepackage{mathrsfs}
\usepackage{enumitem}

\theoremstyle{plain}
\newtheorem{theorem}{Theorem}[section]
\newtheorem{proposition}[theorem]{Proposition}
\newtheorem{lemma}[theorem]{Lemma}
\newtheorem{corollary}[theorem]{Corollary}

\theoremstyle{definition}
\newtheorem{definition}[theorem]{Definition}
\newtheorem{remark}[theorem]{Remark}

\newtheorem{step}{Step}

\makeatletter

\@addtoreset{equation}{section}
\makeatother

\newcommand{\R}{\ensuremath{\mathbb{R}}}
\newcommand{\N}{\ensuremath{\mathbb{N}}}
\newcommand{\Prob}{\ensuremath{\mathcal{P}}}

\newcommand{\Dom}{\mathcal{D}}
\newcommand{\comp}{{\mathrm{c}}}
\newcommand{\ac}{{\mathrm{ac}}}
\newcommand{\Id}{\mathrm{Id}}
\newcommand{\tran}{{\mathrm{t}}}
\renewcommand{\det}{{\mathrm{det}}}
\newcommand{\Ent}{{\mathrm{Ent}}}
\newcommand{\CD}{\mathrm{CD}}
\newcommand{\TCD}{\mathrm{TCD}}
\newcommand{\RCD}{\mathrm{RCD}}
\newcommand{\meas}{\mathfrak{m}}
\newcommand{\Ric}{\mathrm{Ric}}
\newcommand{\vol}{\mathrm{vol}}
\renewcommand{\d}{\ensuremath{\mathrm{d}}}
\newcommand{\e}{\ensuremath{\mathrm{e}}}
\newcommand{\FF}{\mathcal{F}}
\newcommand{\GG}{\mathcal{G}}

\newcommand{\BB}{\mathsf{B}}
\newcommand{\II}{\mathsf{I}}
\newcommand{\JJ}{\mathsf{J}}
\newcommand{\RR}{\mathsf{R}}
\newcommand{\TT}{\mathsf{T}}

\DeclareMathOperator{\proj}{proj}
\DeclareMathOperator{\supp}{spt}
\DeclareMathOperator{\sing}{sing}
\DeclareMathOperator{\trace}{trace}
\DeclareMathOperator*{\esssup}{ess\,sup}

\newcommand{\MB}[1]{}

\newcommand{\SO}[1]{\textcolor{orange}{#1}}

\title{Optimal transport and timelike lower Ricci curvature bounds on Finsler spacetimes}
\author{Mathias Braun\thanks{Fields Institute for Research in Mathematical Sciences, 222 College Street, Toronto, Ontario M5T 3J1, Canada. \texttt{braun@math.toronto.edu}}{ }, Shin-ichi Ohta\thanks{Department of Mathematics, Osaka University, Osaka 560-0043, Japan. \texttt{s.ohta@math.sci.osaka-u.ac.jp}}}
\date{May 2023}

\usepackage[runin]{abstract}

\begin{document}
\setlength{\abstitleskip}{-\absparindent}
\abslabeldelim{\ \ }

\maketitle

\begin{abstract}
We prove that a Finsler spacetime endowed with a smooth reference measure whose induced weighted Ricci curvature $\Ric_N$ is bounded from below by a real number $K$ in every timelike direction satisfies the timelike curvature-dimension condition $\smash{\mathrm{TCD}_q(K,N)}$ for all $q\in (0,1)$.
A nonpositive-dimensional version ($N \le 0$) of this result is also shown.
Our discussion is based on the solvability of the Monge problem with respect to the $q$-Lorentz--Wasserstein distance as well as the characterization of $q$-geodesics of probability measures.
One consequence of our work is the sharp timelike Brunn--Minkowski inequality in the Lorentz--Finsler case.
\end{abstract}



\MB{\textbf{Metaremark:} I've gone through the introduction, have made sure all numbered equations and references are referred to at least once, and finished the proof of Thm \ref{Th:TCD-Ric}. Please check the latter, but otherwise I'm finished.}

\MB{I'm finished with the proof now. It is quite condensed as of now, so please check and correct if necessary. Especially, I'm not sure if the limit argument for $N=0$ is correct, I might have overlooked something.}

\section{Introduction}

The investigation of lower Ricci curvature bounds on generalizations of Lorentzian manifolds is making breathtaking progress in recent time.
Most importantly, McCann \cite{Mc} pioneered  the equivalence between the lower weighted timelike Ricci curvature bound and the convexity of an entropy functional in terms of optimal transport theory (see \cite{mondino-suhr} as well),
and then in \cite{cavalletti2020} Cavalletti--Mondino developed the synthetic theory of non-smooth measured Lorentzian spaces with timelike Ricci curvature bounded below on top of the framework of Lorentzian length spaces introduced by Kunzinger--S\"amann \cite{kunzinger2018} (see also \cite{braun2022,minguzzi-suhr} for recent alternatives to either approaches). 

This development is reminiscent of (and motivated by) the highly successful theory of \emph{curvature-dimension condition} $\CD(K,N)$ for metric measure spaces developed by Lott, Sturm, Villani, and many others (we refer to \cite{villani2009} for a good overview).
This condition is defined by a convexity property of an entropy functional along geodesics in the Wasserstein space of probability measures (in other words, along optimal transports with respect to a specific cost function). For a weighted Riemannian manifold $\smash{(M,g,\e^{-\psi}\,\vol_g)}$, where $\psi$ is a weight function, $\CD(K,N)$ is equivalent to the lower bound $\smash{\Ric_N \ge K}$ for the weighted Bakry--Émery Ricci tensor $\smash{\Ric_N}$ \cite{bakry}.
Thus, a metric measure space satisfying $\CD(K,N)$ can be regarded as a non-smooth space with $\Ric_N \ge K$.

In the Lorentzian setting, the characterizations \cite{Mc,mondino-suhr} yield in particular a non-smooth analogue of the Hawking--Penrose strong energy condition \cite{hawking, hawking-penrose}. This property is regarded as the reason for which gravity is attractive, and not repulsive \cite{carroll}. A second physically relevant consequence is the prediction of singularities in the presence of e.g.~trapped surfaces. In fact, through their approach Cavalletti--Mondino obtained a synthetic Hawking singularity theorem which generalizes classical \cite{hawking, hawking1973} and recent \cite{Case,GW2014,graf,kunzinger} results for high and low regularity spacetimes, cf.~\cite{braun2022b}. This shows the strength of these tools and makes them interesting both in smooth and non-smooth settings. 

In the Lorentzian framework of \emph{weighted Finsler spacetimes}, some singularity theorems, comparison theorems, and a timelike splitting theorem have been established in \cite{LMO1,LMO2,LMO3}. As seen before, such results are implied by time\-like curvature bounds. It is thus natural to expect that the known characterization of the lower weighted timelike Ricci curvature bound $\Ric_N \ge K$ by the \emph{timelike curvature-dimension condition} $\TCD(K,N)$ (here we suppress the dependence on a parameter $q \in (0,1)$ for simplicity) can be generalized to weighted Finsler spacetimes (equipped with an appropriate reference measure).

The aim of this paper is to establish this equivalence (Theorems~\ref{th:Ric-TCD}, \ref{Th:TCD-Ric}), followed by the timelike Brunn--Minkowski inequality as a  geometric application (Theorem~\ref{Th:Brunn Minkowski}). Our characterization not only covers $N\in [n,+\infty]$, but also the case of non-positive ``dimension'' $N \in (-\infty,0]$ (going back to \cite{ohta-negative,ohta-needle} in the Riemannian or Finsler framework), which is new even in the Lorentzian case.
We refer to \cite{DDMT,WoWy2016,WoWy2018} for some works on this range of $N$.

Our characterization of $\Ric_N \ge K$ by $\TCD(K,N)$ has two major contributions. On the one hand, it enlarges the availability of $\TCD(K,N)$ to Lorentz--Finsler manifolds; the ``sharp'' $\TCD$ condition from Definition \ref{df:TCD}, in contrast to its entropic or reduced counterparts, is new even in the genuine Lorentzian setting.
Lorentz--Finsler manifolds form a much wider class than Lo\-rentz\-ian manifolds and allow for more room in physical interpretations.
In fact, they play a key role in several approaches to gravity, e.g.~fields in media, or models for dark matter or dark energy,  and have resolved some important difficulties; see \cite{Beem,hohmann, matsumoto, Min-cone, szilasi} and the references therein.
Then the validity of $\TCD(K,N)$ is obviously meaningful progress.
On the other hand, this validity strongly motivates the study of some appropriate ``Lorentzian'' condition in addition to $\TCD(K,N)$ in order to discuss genuinely Lorentzian properties, namely an analogue of the \emph{Riemannian curvature-dimension condition} $\RCD(K,N)$ for metric measure spaces in the Lorentzian context. One motivation behind the introduction of $\RCD(K,N)$ was that Finsler manifolds can satisfy $\CD(K,N)$. Precisely, the equivalence between $\CD(K,N)$ and $\Ric_N \ge K$ holds true also in the Finsler context \cite{ohta2009,Obook}. To discuss genuinely Riemannian properties such as isometric splitting theorems, as obtained in \cite{gigli2013}, the linearity of the heat flow was added to $\CD(K,N)$ as a further hypothesis \cite{AGS,AGS2, EKS,gigli2015}. On Finsler manifolds, this linearity is equivalent to the manifold being Riemannian.

Besides the above characterization of $\TCD(K,N)$, we show that synthetic timelike sectional curvature bounds (from below or above) in terms of triangle comparison \cite{alexander2008,harris1982,kunzinger2018, minguzzi-suhr} rule out non-Lorentzian Lorentz--Finsler manifolds (Proposition~\ref{pr:triangle}). Moreover, we provide a simpler proof of the absolute continuity of intermediate measures along optimal transports following the idea of Figalli--Juillet \cite{figalli2008} relying on the measure contraction property.

This article is organized as follows.
In Section~\ref{sc:pre} we review the basics of Lorentz--Finsler geometry (including variation formulae of arclength and the invalidity of triangle comparison).
Section~\ref{Sec:Behavior} is concerned with the Lorentz--Finsler distance function; we especially show the semi-convexity and the failure of semi-concavity on timelike cut loci.
In Section~\ref{sc:geod} we introduce the $q$-Lorentz--Wasserstein distance for $q \in (0,1]$ and study the associated optimal transports ($q$-geodesics).
The useful notion of $q$-separation will be introduced following \cite{Mc}.
Then Section~\ref{sc:Ric-TCD} is devoted to the first half of our main result, namely we establish that $\Ric_N \ge K$ in the timelike directions implies $\TCD(K,N)$.
We also show the timelike Brunn--Minkowski inequality as an application.
Finally, in Section~\ref{sc:TCD-Ric} we show the converse implication from $\TCD(K,N)$ to $\Ric_N \ge K$. 

\paragraph{Acknowledgments}
MB was supported by the Fields Institute for Research in Mathematical Sciences. 
He thanks Osaka University for its kind hospitality during the workshop ``Geometry and Probability 2022''. 
SO was supported in part by JSPS Grant-in-Aid for Scientific Research (KAKENHI) 19H01786, 22H04942.
He thanks the Fields Institute for its kind hospitality during the workshop ``Aspects of Ricci Curvature Bounds''. 
The authors are sincerely grateful to Robert McCann and Ettore Minguzzi for several helpful discussions about this project, and to the Erwin Schr\"odinger International Institute
for Mathematics and Physics for its kind hospitality during the workshop ``Non-regular Spacetime Geometry''.

\section{Finsler spacetimes}\label{sc:pre}

We review the theory of Finsler spacetimes required to develop our program. We refer the reader to \cite{beem1996,hawking1973,oneill1983} for the basics of Lorentzian geometry, 
and to \cite{Min-spray,minguzzi2015c,Min-causality,Min-Rev} for some generalizations including Lorentz--Finsler manifolds.

Throughout the paper,
let $M$ be a connected  $C^{\infty}$-manifold without boundary of dimension $\dim M = n\ge 2$. Given a local coordinate system $\smash{(x^\alpha)_{\alpha=1}^n}$ on an open set $U\subset M$,
we will represent a tangent vector $v \in T_xM$ with $x \in U$ as
$\smash{v =\sum_{\alpha=1}^n v^{\alpha}\, (\partial/\partial x^{\alpha})\vert_x}$.

\subsection{Preliminaries for Lorentz--Finsler manifolds}\label{Sec:Preliminaries}

The following notations follow \cite{LMO1, LMO2, LMO3} except that $\dim M$ has been set to $n+1$ therein.

\begin{definition}[Lorentz--Finsler structure]
\label{df:LFstr}
A \emph{Lorentz--Finsler structure} of $M$ is a function
$L:TM \longrightarrow \R$ satisfying the following conditions.
\begin{enumerate}[label=\textnormal{(\arabic*)}]
\item $L \in C^{\infty}(TM \setminus 0)$, where $0$ denotes the zero section.

\item $L(c\,v)=c^2\, L(v)$ for all $v \in TM$ and all $c>0$.

\item For any $v \in TM \setminus 0$, the symmetric matrix
\begin{equation}\label{eq:g_ij}
\big( g_{\alpha \beta}(v) \big)_{\alpha,\beta=1}^n
 :=\bigg( \frac{\partial^2 L}{\partial v^\alpha \partial v^\beta}(v) \bigg)_{\alpha,\beta=1}^n
\end{equation}
is non-degenerate with signature $(-,+,\ldots,+)$.
\end{enumerate}
Then we call $(M,L)$ a \emph{Lorentz--Finsler manifold}.
\end{definition}

We say that $(M,L)$ is \emph{reversible} if $L(-v)=L(v)$ for all $v\in TM$.
Given any $v \in T_xM \setminus \{0\}$, we define a Lorentzian metric $g_v$ on $T_xM$
by using \eqref{eq:g_ij} as
\begin{equation}\label{eq:g_v}
g_v \Bigg( \sum_{\alpha=1}^n a^\alpha\, \frac{\partial}{\partial x^\alpha}\bigg|_x,
 \sum_{\beta=1}^n b^\beta\, \frac{\partial}{\partial x^\beta}\bigg|_x \Bigg)
 :=\sum_{\alpha,\beta=1}^n g_{\alpha \beta}(v)\, a^\alpha\, b^\beta.
\end{equation}
Then we have $g_v(v,v)=2\,L(v)$ by Euler's homogeneous function theorem. 

A tangent vector $v \in TM $ is said to be \emph{timelike} (or \emph{null}, respectively)
if $L(v)<0$ (or $L(v)=0$, respectively).
We say that $v$ is \emph{lightlike} if it is null and nonzero,
and \emph{causal} (or \emph{non-spacelike}) if it is timelike or lightlike
(i.e.\ $L(v) \le 0$ and $v \neq 0$).
\emph{Spacelike} vectors are those for which $L(v)>0$ or $v=0$.
We denote by $\smash{\Omega'_x \subset T_xM}$ the set of timelike vectors, and we set
$\smash{\Omega' :=\bigcup_{x \in M} \Omega'_x}$.
For later use, we define, for causal vectors $v$,
\begin{equation}\label{eq:LtoF}
F(v) :=\sqrt{-2\,L(v)} =\sqrt{-g_v(v,v)}.
\end{equation}

Next we introduce the covariant derivative and the Ricci curvature. Given any $v\in TM\setminus 0$ and any $\alpha,\beta,\delta =1,2,\ldots,n$, define
\[ \overline{\Gamma}{}^{\alpha}_{\beta \delta} (v)
 :=\frac{1}{2} \sum_{\lambda=1}^n g^{\alpha \lambda}(v)
\, \bigg( \frac{\partial g_{\lambda \delta}}{\partial x^{\beta}}(v) +\frac{\partial g_{\beta \lambda}}{\partial x^{\delta}}(v)
 -\frac{\partial g_{\beta \delta}}{\partial x^{\lambda}}(v) \bigg), \]
where $\smash{(g^{\alpha \beta}(v))_{\alpha,\beta=1}^n}$ is the inverse matrix of $\smash{(g_{\alpha\beta}(v))_{\alpha,\beta=1}^n}$, 
\[ G^{\alpha}(v) :=\frac{1}{2}\sum_{\beta,\delta=1}^n \overline{\Gamma}{}^{\alpha}_{\beta \delta}(v)\, v^{\beta} \,v^{\delta},
 \qquad N^{\alpha}_{\beta}(v) :=\frac{\partial G^{\alpha}}{\partial v^{\beta}}(v) \]
with the convention $G^{\alpha}(0) := N^{\alpha}_{\beta}(0):=0$, and
\begin{align}\label{eq:Gamma}
\begin{split}
\Gamma^{\alpha}_{\beta \delta}(v) &:=\overline{\Gamma}{}^{\alpha}_{\beta \delta}(v)
 -\frac{1}{2}\sum_{\lambda,\mu=1}^n g^{\alpha \lambda}(v)\,
 \bigg( \frac{\partial g_{\lambda \delta}}{\partial v^{\mu}}(v)\, N^{\mu}_{\beta}(v)\\
 &\qquad\qquad  +\frac{\partial g_{\beta \lambda}}{\partial v^{\mu}}(v)\, N^{\mu}_{\delta}(v)
 -\frac{\partial g_{\beta \delta}}{\partial v^{\mu}}(v)\, N^{\mu}_{\lambda}(v) \bigg).
 \end{split}
\end{align}
Then the \emph{covariant derivative} of a vector field
$\smash{X=\sum_{\alpha=1}^n X^{\alpha}\, (\partial/\partial x^{\alpha})}$ in the direction $v \in T_xM$ with reference vector $w \in T_xM \setminus \{0\}$ is defined as
\[ D_v^w X :=\sum_{\alpha,\beta=1}^n
 \bigg( v^{\beta}\, \frac{\partial X^{\alpha}}{\partial x^{\beta}}(x)
 +\sum_{\delta=1}^n \Gamma^{\alpha}_{\beta \delta}(w)\, v^{\beta}\, X^{\delta}(x) \bigg)\,
 \frac{\partial}{\partial x^{\alpha}} \bigg|_x. \]
We remark that the functions $\smash{\Gamma^{\alpha}_{\beta \delta}}$ in \eqref{eq:Gamma}
are the coefficients of the \emph{Chern connection} (or \emph{Chern--Rund connection}). In the Lorentzian case, $\smash{g_{\alpha\beta}}$ is constant in each tangent space
(thus $\smash{\Gamma^{\alpha}_{\beta\delta}=\overline{\Gamma}{}^{\alpha}_{\beta\delta}}$)
and the covariant derivative does not depend on the choice of a reference vector.

The \emph{geodesic equation} for $L$ is written as $\smash{D^{\dot{\gamma}}_{\dot{\gamma}} \dot{\gamma} = 0}$.
A $C^{\infty}$-vector field $J$ along a geodesic $\gamma$ is called a \emph{Jacobi field}
if it satisfies $\smash{D^{\dot{\gamma}}_{\dot{\gamma}} D^{\dot{\gamma}}_{\dot{\gamma}} J +R_{\dot{\gamma}}(J) =0}$. Here  the \emph{curvature tensor} is defined by setting, for  $v,w \in T_xM$,
\[ R_v(w):=\sum_{\alpha,\beta=1}^n R^{\alpha}_{\beta}(v)\, w^{\beta}\, \frac{\partial}{\partial x^{\alpha}}, \]
where
\begin{align*}
R^{\alpha}_{\beta}(v) &:=2\,\frac{\partial G^{\alpha}}{\partial x^{\beta}}(v)
 -\sum_{\delta=1}^n \bigg( \frac{\partial N^{\alpha}_{\beta}}{\partial x^{\delta}}(v)\, v^{\delta}
 -2\,\frac{\partial N^{\alpha}_{\beta}}{\partial v^{\delta}}(v)\, G^{\delta}(v) \bigg)\\
&\qquad\qquad -\sum_{\delta=1}^n N^{\alpha}_{\delta}(v)\, N^{\delta}_{\beta}(v). 
\end{align*}
A Jacobi field is also characterized as the variational vector field of a geodesic variation.
Note that $R_v(w)$ is linear in $w$, thereby $R_v\colon T_xM \longrightarrow T_xM$ is an endomorphism
for each $v \in T_xM$.

\begin{definition}[Ricci curvature]\label{df:curv}
For $v \in T_xM$, we define the \emph{Ricci curvature}
\textnormal{(}or \emph{Ricci scalar}\textnormal{)} of $v$ as the trace of $R_v$, i.e. $\mathrm{Ric}(v):=\mathrm{trace}(R_v)$.
\end{definition}

We remark that $\mathrm{Ric}(c\,v)=c^2\, \mathrm{Ric}(v)$ for every  $c>0$.
For later use, we also recall that
\begin{equation}\label{eq:R_v}
R_v(v)=0, \qquad
g_v\big( w_1,R_v(w_2) \big) =g_v \big( R_v(w_1),w_2 \big)
\end{equation}
for any $v \in \Omega'_x$ and $w_1,w_2 \in T_xM$ (see \cite[Proposition~2.4]{minguzzi2015c}, \cite[Proposition 3.4]{LMO1}).

\subsection{Preliminaries for Finsler spacetimes}\label{ssc:Fspacetime}

\begin{definition}[Finsler spacetime]\label{df:spacetime}
If a Lorentz--Finsler manifold $(M,L)$ admits a smooth timelike vector field $X$,
then $(M,L)$ is said to be \emph{time oriented} \textnormal{(}by $X$\textnormal{)}.
A time oriented Lorentz--Finsler manifold is called a \emph{Finsler spacetime}.
\end{definition}

In such a Finsler spacetime,
a causal vector $v \in T_xM$ is said to be \emph{future-directed}
if it lies in the same connected component of $\smash{\overline{\Omega}'_x \setminus \{0\}}$ as $X(x)$.
We denote by $\Omega_x \subset \Omega'_x$ the set of future-directed timelike vectors,
and define
\[ \Omega :=\bigcup_{x \in M} \Omega_x, \qquad
 \overline{\Omega} :=\bigcup_{x \in M} \overline{\Omega}_x, \qquad
 \overline{\Omega} \setminus 0 :=\bigcup_{x \in M} (\overline{\Omega}_x \setminus \{0\}). \]
A $C^1$-curve in $(M,L)$ is said to be \emph{timelike} (or \emph{causal}, respectively)
if its tangent vector is always timelike (or causal, respectively). All causal curves will be future-directed unless otherwise indicated.

We recall some fundamental concepts in causality theory.
Given $x,y \in M$, we write $x \ll y$ (or $x<y$, respectively)
if there is a future-directed timelike (or causal, respectively) curve from $x$ to $y$, and $x \le y$ means that $x=y$ or $x<y$.
Then the \emph{chronological past}
and the \emph{chronological future} of $x$ are defined by
\[ I^-(x):=\{ y \in M \mid y \ll x\}, \qquad I^+(x):=\{ y \in M \mid x \ll y\}, \]
and its \emph{causal past} and its \emph{causal future} are defined by
\[ J^-(x):=\{ y \in M \mid y \le x\}, \qquad J^+(x):=\{ y \in M \mid x \le y\}. \]
For a general set $S\subset M$, we define $I^-(S)$, $I^+(S)$, $J^-(S)$, and $J^+(S)$ analogously. For a Borel probability measure $\mu$ on $M$, we abbreviate $\smash{I^-(\mu) := I^-(\supp\mu)}$ in terms of its support $\supp\mu\subset M$, and analogously for $I^+(\mu)$, $J^-(\mu)$, and $J^+(\mu)$.

\begin{definition}[Causality conditions]\label{df:causal}
Let $(M,L)$ be a Finsler spacetime.
\begin{enumerate}[label=\textnormal{(\arabic*)}]
\item $(M,L)$ is said to be \emph{chronological} if $x \notin I^+(x)$ for all $x\in M$.
\item We say that $(M,L)$ is \emph{causal} if there is no closed causal curve.
\item $(M,L)$ is said to be \emph{strongly causal} if, for all $x \in M$,
every neighborhood $U$ of $x$ contains another neighborhood $V$ of $x$
such that no causal curve intersects $V$ more than once.
\item We say that $(M,L)$ is \emph{globally hyperbolic}
if it is strongly causal and, for any $x,y \in M$, $J^+(x) \cap J^-(y)$ is compact.
\end{enumerate}
\end{definition}

We also define the \emph{Lorentz--Finsler distance} $l(x,y)$ for $x,y \in M$ (also called the \emph{time separation function}) by
\[ l(x,y) :=\sup_{\gamma} \mathfrak{L}(\gamma), \qquad
 \mathfrak{L}(\gamma) :=\int_0^1 F\big( \dot{\gamma}(t) \big) \,\d t, \]
where $\gamma\colon [0,1] \longrightarrow M$ runs over all causal curves from $x$ to $y$
(recall \eqref{eq:LtoF} for the definition of $F$).
A causal curve attaining the above supremum is said to be \emph{maximizing}.
Similarly to \cite{Mc}, we set $l(x,y):=-\infty$ if there is no causal curve from $x$ to $y$ (namely $x \not< y$).
(Note that we set $l(x,y)=0$ for $x \not< y$ in \cite{LMO1,LMO2,LMO3}.) In general, the distance function $l$ is only lower semi-continuous on $\{l \ge 0\}$ \cite[Proposition~6.7]{minguzzi2015c} and can be infinite.
In globally hyperbolic Finsler spacetimes,
$l$ is finite and continuous on $\{l \ge 0\}$ and upper semi-continuous on $M \times M$  \cite[Proposition~6.8]{minguzzi2015c}, 
and any pair of points $x,y \in M$ with $x<y$ admits a maximizing geodesic from $x$ to $y$
\cite[Propositions~6.9]{minguzzi2015c}. More details about the regularity of $l$ under global hyperbolicity are deferred to Section \ref{Sec:Behavior}.

To consider the dual structure to $L$ and the Legendre transform
(see \cite{Min-cone}, \cite[\S3.1]{Min-causality}, or \cite[\S4.4]{LMO2} for further discussions),
define the \emph{polar cone} to $\Omega_x$ by
\[ \Omega^*_x :=\big\{ \zeta \in T_x^*M \mid
 \zeta(v)<0\ \text{for all}\ v \in \overline{\Omega}_x \setminus \{0\} \big\}. \]
This is an open convex cone in $T_x^*M$.
For $\zeta \in \Omega^*_x$, we define
\[ L^*(\zeta) := -\frac{1}{2}\,\big( \sup \zeta(v) \big)^2
 =-\frac{1}{2}\,\inf  \big( \zeta(v) \big)^2, \]
where the supremum and the infimum run over $v \in F^{-1}(1) \cap \Omega_x$.
Then we have the \emph{reverse Cauchy--Schwarz inequality}
\begin{equation}\label{eq:rCS}
4\,L^*(\zeta)\, L(v) \le \big( \zeta(v) \big)^2
\end{equation}
for every $v \in \Omega_x$ and every $\zeta \in \Omega^*_x$,
as well as the following variational definition of the Legendre transform.

\begin{definition}[Legendre transform]\label{df:Leg}
The \emph{Legendre transform}
$\mathscr{L}^*\colon \Omega^*_x \longrightarrow \Omega_x$ is defined as the map sending $\zeta \in \Omega^*_x$
to the unique element $v \in \Omega_x$ satisfying $L(v)=L^*(\zeta)=\zeta(v)/2$.
We also define $\mathscr{L}^*(0):=0$.
\end{definition}

A coordinate expression of the Legendre transform is given by
\begin{equation*}
\mathscr{L}^*(\zeta)
 =\sum_{\alpha=1}^n \frac{\partial L^*}{\partial \zeta_{\alpha}}(\zeta)\, \frac{\partial}{\partial x^{\alpha}},
 \qquad \text{where}\,\ \zeta =\sum_{\alpha=1}^n \zeta_{\alpha} \,\d x^{\alpha}.
\end{equation*}
The inverse map $\mathscr{L}:=(\mathscr{L}^*)^{-1}$ is written in the same way as
\[
\mathscr{L}(v) =\sum_{\alpha=1}^n \frac{\partial L}{\partial v^{\alpha}}(v) \,\d x^{\alpha}.
\]
We refer to \cite{LMO2,Min-cone,Min-causality}
for some basic properties of the Legendre transforms.

\subsection{Variations of arclength}

Here we present variation formulas for arclength in the Lorentz--Finsler setting.
Although they are obtained by the standard arguments, we give outlines for completeness.
We refer to \cite{oneill1983} for the semi-Riemannian case (including Lorentzian manifolds) and to \cite[Chapter~5]{BCSbook} or \cite[Chapter~7]{Obook} for the Finsler case.

Let $\sigma\colon [a,b] \times (-\varepsilon,\varepsilon) \longrightarrow M$ be a $C^{\infty}$-variation such that $\sigma_s(t):=\sigma(t,s)$ is a timelike curve for each $s \in (-\varepsilon,\varepsilon)$.
We set
\[ \mathfrak{L}(s) :=\mathfrak{L}(\sigma_s) =\int_a^b F(\dot{\sigma}_s) \,\d t. \]
For brevity, we define the tangent and variation vector fields of $\sigma$ as $T:=\partial_t \sigma$ and $V:=\partial_s \sigma$, respectively.
We assume that $T$ vanishes nowhere.

\begin{lemma}[First variation]\label{lm:1st-var}
For $\sigma$ and $\mathfrak{L}$ as above, we have
\begin{align*}
\mathfrak{L}'(s) &=
-\bigg[ g_T\bigg( V,\frac{T}{F(T)} \bigg)(\cdot,s) \bigg]_a^b
+\int_a^b g_T\bigg( V,D^T_T\bigg[ \frac{T}{F(T)} \bigg] \bigg)(t,s) \,\d t \\
&= -\int_a^b g_T\bigg( D^T_T V,\frac{T}{F(T)} \bigg)(t,s) \,\d t.
\end{align*}
In particular, if $\sigma_s$ is a geodesic, then
\[ \mathfrak{L}'(s) =
-\frac{1}{F(\dot{\sigma}_s)}\, \Big[ g_T(V,T)(\cdot,s) \Big]_a^b. \]
\end{lemma}

\begin{proof}
This is shown in the same way as the Finsler case, cf.~e.g.~\cite[Proposition~7.1]{Obook}, by recalling $F^2(v)=-g_v(v,v)$ for every $v \in \Omega$ in \eqref{eq:LtoF}.
\end{proof}

To write down the second variation, we define the \emph{index form} for $C^1$-vector fields $V,W$ along a timelike geodesic $\gamma\colon [a,b] \longrightarrow M$ as
\[ I(V,W) := -\frac{1}{F(\dot{\gamma})} \int_a^b \Big( g_{\dot{\gamma}}(D^{\dot{\gamma}}_{\dot{\gamma}}V,D^{\dot{\gamma}}_{\dot{\gamma}}W) -g_{\dot{\gamma}}\big( R_{\dot{\gamma}}(V),W \big) \Big) \,\d t, \]
see \cite[Chapter~10]{beem1996} for the Lorentzian case.
Note that $I(V,W)=I(W,V)$ and, if $V$ is a Jacobi field, then we have
\begin{equation}\label{eq:index-Jacobi}
I(V,W)= -\frac{1}{F(\dot{\gamma})}\,
\Big[ g_{\dot{\gamma}}(D^{\dot{\gamma}}_{\dot{\gamma}}V,W) \Big]_a^b.
\end{equation}

\begin{proposition}[Second variation]\label{pr:2nd-var}
For $\sigma$ and $\mathfrak{L}$ as above, suppose that $\gamma:=\sigma_0$ is a geodesic.
Then we have
\begin{align*}
\mathfrak{L}''(0) &=
I(V_0,V_0) -\bigg[ g_{\dot{\gamma}} \bigg( D^T_V V(\cdot,0),\frac{\dot{\gamma}}{F(\dot{\gamma})} \bigg) \bigg]_a^b -\int_a^b \frac{(\partial_s[F(T)](t,0))^2}{F(\dot{\gamma}(t))} \,\d t \\
&= I(V_0^{\perp},V_0^{\perp}) -\bigg[ g_{\dot{\gamma}} \bigg( D^T_V V(\cdot,0),\frac{\dot{\gamma}}{F(\dot{\gamma})} \bigg) \bigg]_a^b,
\end{align*}
where we set $V_0:=V(\cdot,0)$ and
\[ V_0^{\perp}:=V_0+g_{\dot{\gamma}}\bigg( V_0,\frac{\dot{\gamma}}{F(\dot{\gamma})} \bigg) \,\frac{\dot{\gamma}}{F(\dot{\gamma})} \]
is the projection of $V_0$ to the $g_{\dot{\gamma}}$-orthogonal subspace to $\dot{\gamma}$.
\end{proposition}

\begin{proof}
We can again argue similarly to the Finsler case by taking care of the signature, let us explain along the lines of \cite[Theorem~7.6]{Obook} and \cite[\S5.2]{BCSbook}.
By differentiating $\mathfrak{L}'(s)$ in Lemma~\ref{lm:1st-var} and by $F^2(T)=-g_T(T,T)$, we find
\begin{align}
\begin{split}
\mathfrak{L}''(0) &=
-\frac{1}{F(\dot{\gamma})} \int_a^b \Big( g_{\dot{\gamma}}\big(D^T_V[D^T_T V],\dot{\gamma}\big) +g_{\dot{\gamma}}\big(D^{\dot{\gamma}}_{\dot{\gamma}}V_0,D^{\dot{\gamma}}_{\dot{\gamma}}V_0\big)\Big) \,\d t\\
&\qquad\qquad -\int_a^b \frac{(\partial_s[F(T)](t,0))^2}{F(\dot{\gamma}(t))} \,\d t.
\label{eq:L''}
\end{split}
\end{align}
Then, choosing local coordinates $(x^{\alpha})_{\alpha=1}^n$ such that $\dot{\gamma}=\partial/\partial x^1$, $g_{11}(\dot{\gamma})=-1$, $g_{ii}(\dot{\gamma})=1$ for $i=2,\dots, n$, and $g_{\alpha\beta}(\dot{\gamma})=0$ for distinct $\alpha,\beta=1,\dots,n$, we have
\[ g_{\dot{\gamma}}\big(D^T_V[D^T_T V]-D^T_T[D^T_V V],\dot{\gamma}\big)
=-\sum_{\alpha,\beta=1}^n R^1_{\alpha\beta 1}(\dot{\gamma}) \,V_0^{\alpha}\, V_0^{\beta}. \]
Now we apply the skew-symmetry \cite[(5.11)]{Obook} to see
\[ \sum_{\alpha,\beta=1}^n R^1_{\alpha\beta 1}(\dot{\gamma}) \,V_0^{\alpha}\, V_0^{\beta}
=\sum_{\alpha,\beta,\delta=1}^n g_{\alpha\delta}(\dot{\gamma})\, R^{\delta}_{\beta}(\dot{\gamma})\, V_0^{\alpha}\, V_0^{\beta}
=g_{\dot{\gamma}}\big( R_{\dot{\gamma}}(V_0),V_0 \big) \]
(thus \cite[(7.2)]{Obook} holds as it is).
Plugging these into \eqref{eq:L''} yields the first equation.

The second equation can be derived following \cite[Exercise~5.2.6]{BCSbook}.
Precisely, we have
\[ g_{\dot{\gamma}}(D^{\dot{\gamma}}_{\dot{\gamma}} V_0^{\perp},D^{\dot{\gamma}}_{\dot{\gamma}} V_0^{\perp})
= g_{\dot{\gamma}}(D^{\dot{\gamma}}_{\dot{\gamma}} V_0,D^{\dot{\gamma}}_{\dot{\gamma}} V_0) +\big( \partial_s[F(T)](\cdot,0) \big)^2 \]
by calculation and
\[ g_{\dot{\gamma}}\big( R_{\dot{\gamma}}(V_0^{\perp}),V_0^{\perp} \big)
= g_{\dot{\gamma}}\big( R_{\dot{\gamma}}(V_0),V_0 \big) \]
by \cite[Proposition~3.4]{LMO1}.
This completes the proof.
\end{proof}

\subsection{Invalidity of triangle comparisons}

Besides Ricci curvature bounds,
synthetic lower and upper sectional curvature bounds can be formulated in terms of triangle comparisons,
akin to celebrated theories of Alexandrov spaces and CAT$(k)$-spaces
arising from the Alexandrov--Toponogov triangle comparison theorems in Riemannian geometry, cf.~e.g.~\cite{BBI}.
However, similarly to Alexandrov spaces and CAT$(k)$-spaces, those curvature bounds by triangle comparisons rule out genuinely non-Lorentzian Lorentz--Finsler structures.
This means that, in particular, the splitting theorem in \cite{BORS} is not applicable to this framework.

A Lorentzian analogue to the triangle comparison property can be given as follows along \cite[Definition~4.7]{kunzinger2018}.
We denote by $\mathbb{M}^2_k$ the two-dimensional model space of constant sectional curvature $k \in \mathbb{R}$,
and by $\smash{\bar{l}}$ the distance function on it.
Then, for any $a,b,c>0$ with $a+b \le c$ and $\smash{c<\pi/\sqrt{k}}$ (or $c<\pi/\sqrt{-k}$, respectively) if $a+b=c$ and $k>0$ (or $a+b<c$ and $k<0$, respectively),
there exist $\smash{\bar{x},\bar{y},\bar{z} \in \mathbb{M}^2_k}$ such that $\smash{\bar{x} \ll \bar{y} \ll \bar{z}}$ and
\[ \bar{l}(\bar{x},\bar{y})=a, \qquad
\bar{l}(\bar{y},\bar{z})=b, \qquad
\bar{l}(\bar{x},\bar{z})=c. \]
(See the realizability lemmas in \cite[Lemma~2.1]{alexander2008} and \cite[Lemma~4.6]{kunzinger2018}).
Such a triangle $\triangle(\bar{x},\bar{y},\bar{z})$ is unique up to isometry and called a \emph{comparison triangle} of $\triangle(x,y,z)$ in $(M,L)$ when $l(x,y)=a$, $l(y,z)=b$, and $l(x,z)=c$.

We say that $(M,L)$ has \emph{timelike curvature bounded below  by $k \in \R$ in the sense of triangle comparison} if every point in $M$ admits a neighborhood $U$ satisfying the following:
For any $x,y,z \in U$ with $x \ll y \ll z$, a comparison triangle $\triangle(\bar{x},\bar{y},\bar{z}) \subset \mathbb{M}^2_k$ and for any $p,q$ on (the sides of) $\triangle(x,y,z)$ and the corresponding points $\bar{p},\bar{q}$ on (the sides of)  $\triangle(\bar{x},\bar{y},\bar{z})$, we have $l(p,q) \le \bar{l}(\bar{p},\bar{q})$. By reversing the latter inequality, we define $(M,L)$ to have \emph{timelike curvature bounded above by $k$ in the sense of triangle comparison}.

The investigation of triangle comparison theorems in the Lorentzian setting goes back to \cite{harris1982}. We refer to \cite{alexander2008} for the equivalence between triangle comparison properties and the corresponding sectional curvature bounds for general semi-Riemannian manifolds
(see also \cite[Example~4.9]{kunzinger2018}).

\begin{proposition}[Triangle comparison implies Lorentzianity]\label{pr:triangle}
Let $(M,L)$ be globally hyperbolic.
If $(M,L)$ has timelike curvature bounded below or above by some $k \in \R$ in the sense of triangle comparison, then $g_v=2L$ holds on $\Omega_x$ for every $v \in \Omega_x$ and every $x \in M$.
In particular, $L$ is necessarily quadratic in $\Omega_x$ for every $x \in M$.
\end{proposition}

\begin{proof}
We first consider the lower curvature bound.
Fix arbitrary $x \in M$ and $v,w \in \Omega_x$.
By taking a scaling limit of the triangle comparison inequality on $M$, we find that the curvature bound also holds for triangles in the tangent space $T_xM$ with $k=0$ (i.e.~$(T_xM,L)$ is nonnegatively curved).
Now, as a comparison triangle for $\triangle(0,v,v+w)$,
take $\triangle(0,\bar{v},\bar{v}+\bar{w})$ in the Minkowski space $(\mathbb{M}^2_0,\langle \cdot,\cdot \rangle)$.
Then for every $t \in (0,1)$ we obtain
\[
\langle \bar{v}+t\,\bar{w},\bar{v}+t\,\bar{w} \rangle
= (1-t)\,\langle \bar{v},\bar{v} \rangle +t\,\langle \bar{v}+\bar{w},\bar{v}+\bar{w} \rangle
 -(1-t)\,t\,\langle \bar{w},\bar{w} \rangle.
\]
Therefore, it follows from the hypothesis that
\begin{align*}
2\,L(v+t\,w)
&\ge (1-t)\,\langle \bar{v},\bar{v} \rangle +t\,\langle \bar{v}+\bar{w},\bar{v}+\bar{w} \rangle
 -(1-t)\,t\,\langle \bar{w},\bar{w} \rangle \\
&= 2\,(1-t)\,L(v) +2\,t\,L(v+w) -2\,(1-t)\,t\,L(w).
\end{align*}
On the other hand, considering the triangle $\triangle(0,t\,w,v+t\,w)$ and noticing
\[ t\,w + t\,\big( (v+t\,w) -t\,w \big) =t\,(v+w), \]
we observe
\[ L\big( t\,(v+w) \big) \ge (1-t)\,L(t\,w) +t\,L(v+t\,w) -(1-t)\,t\,L(v), \]
which is rewritten as
\[ L(v+t\,w) \le (1-t)\,L(v) +t\,L(v+w) -(1-t)\,t\,L(w). \]
This yields
\[ L(v+t\,w) =(1-t)\,L(v) +t\,L(v+w) -(1-t)\,t\,L(w), \]
and hence we have
\[ g_v(w,w) =\frac{\d^2}{\d t^2}\bigg|_0 L(v+t\,w) =2L(w). \]
This completes the proof for the lower curvature bound.

The case of upper curvature bound can be established in the same way by reversing the inequalities.
\end{proof}

\section[Behavior of the Lorentz--Finsler distance function]{Behavior of the Lorentz--Finsler distance\\ function}\label{Sec:Behavior}

From now on, we assume $(M,L)$ to be a globally hyperbolic Finsler spacetime. Under this standing assumption, in this section we study the behavior of the Lorentz--Finsler distance function $l$, 
and introduce the Lagrangian and the Hamiltonian induced from the $q$-th power of $F=\sqrt{-2\,L}$ for $q\in (0,1]$.
We follow the exposition of \cite[Sections~2, 3]{Mc}.

\subsection{Lagrangians and Hamiltonians}\label{Sec:Lorentz-Finsler distance}

We first introduce the \emph{$q$-dependent Lagrangian} $\mathcal{L}_q$ as in \cite{Mc} (see also \cite{suhr2018} for the case of $q=1$).
Let $q \in (0,1]$, and define $\mathcal{L}_q\colon TM\longrightarrow (-\infty,0]\cup\{+\infty\}$ as
\begin{equation*}
\mathcal{L}_q(v):=
\begin{cases}
 -q^{-1}\,\big(\!-\!2\,L(v)\big)^{q/2} 
 & \text{if } v \in \overline{\Omega}, \\
 +\infty
 & \text{otherwise}.
\end{cases}
\end{equation*}
Observe that $\smash{\mathcal{L}_q(v) =-q^{-1}\,F(v)^q}$ for $\smash{v\in\overline{\Omega}}$. 
Its convex dual is the \emph{$q$-dependent Hamiltonian} $\mathcal{H}_q\colon T^*M \longrightarrow [0,+\infty]$ given by
\begin{equation}\label{Eq:q-Hamiltonian}
    \mathcal{H}_q(\zeta) := \sup_{v \in TM} \big(\zeta(v) - \mathcal{L}_q(v)\big),
\end{equation}
where $\zeta(v)$ is the usual duality pairing.
When $q\in (0,1)$ and $p<0$ denotes its dual exponent, i.e.\ $q^{-1}+p^{-1}=1$, we have
\begin{equation}\label{eq:H_q}
    \mathcal{H}_q(\zeta) = \begin{cases} -p^{-1}\,\big(\!-\!2\,L^*(\zeta)\big)^{p/2} & \text{if } \zeta \in \Omega^*,\\
    +\infty & \text{otherwise},
    \end{cases}
\end{equation}
while for $q=1$ it is the $\{0,+\infty\}$-valued indicator function of a solid hyperboloid $\{L^* \le -1/2 \} \subset \Omega^*$; see Corollary~\ref{Cor:Classical Lagrangian}.

In view of \eqref{Eq:q-Hamiltonian}, we can show an analogue of $\mathscr{L}=(\mathscr{L}^*)^{-1}$.
Furthermore, choosing $q \in (0,1)$ makes the Lagrangian $\mathcal{L}_q$ strictly convex on $\Omega_x$.
Precisely, we have the following result.

\begin{lemma}[Duality and convexity of $\mathcal{L}_q$]\label{Le:Convex Lagrangian}
Fix $q \in (0,1)$ and $x \in M$.
\begin{enumerate}[label=\textnormal{(\roman*)}]
\item
The maps
\[
v \,\longmapsto\, \sum_{\alpha=1}^n \frac{\partial \mathcal{L}_q}{\partial v^{\alpha}}(v) \,\d x^{\alpha}, \qquad
\zeta \,\longmapsto\, \sum_{\alpha=1}^n \frac{\partial \mathcal{H}_q}{\partial \zeta_{\alpha}}(\zeta)\, \frac{\partial}{\partial x^{\alpha}}
\]
are the inverses of each other between $\Omega_x$ and $\Omega^*_x$.

\item
The Lagrangian $\mathcal{L}_q$ is convex on $T_xM$ and strictly convex on $\Omega_x$.
\end{enumerate}
\end{lemma}

\begin{proof}
(i)
Observe that
\[
\sum_{\alpha=1}^n \frac{\partial \mathcal{L}_q}{\partial v^{\alpha}}(v) \,\d x^{\alpha}
= F(v)^{q-2} \sum_{\alpha=1}^n \frac{\partial L}{\partial v^{\alpha}}(v) \,\d x^{\alpha}
=F(v)^{q-2}\, \mathscr{L}(v),
\]
and similarly
\begin{equation*}
\sum_{\alpha=1}^n \frac{\partial \mathcal{H}_q}{\partial \zeta_{\alpha}}(\zeta)\, \frac{\partial}{\partial x^{\alpha}}
=F^*(\zeta)^{p-2}\, \mathscr{L}^*(\zeta)
\end{equation*}
for $v \in \Omega_x$ and $\zeta \in \Omega^*_x$,
where $\smash{F^*(\zeta):=\sqrt{-2\,L^*(\zeta)}}$.
Since $\mathscr{L}^*=\mathscr{L}^{-1}$ and both involved maps are positively $1$-homogeneous, we obtain the claim by noting that
\[
\big( F(v)^{q-1} \big)^{p-2} \, F(v)^{q-2} =1.
\]

(ii)
Let $v \in \Omega_x$.
Recalling the definition of $g_{\alpha \beta}(v)$ in \eqref{eq:g_ij}, we observe
\[
\frac{\partial \mathcal{L}_q}{\partial v^{\alpha} \,\partial v^{\beta}}(v)
=F(v)^{q-4}\, \bigg( (2-q)\, \frac{\partial L}{\partial v^{\alpha}}(v)\, \frac{\partial L}{\partial v^{\beta}}(v)
+F(v)^2\, g_{\alpha \beta}(v) \bigg).
\]
Now note that, by writing $w\in T_x M$ as  $\smash{w=\sum_{\alpha=1}^n w^{\alpha} (\partial/\partial x^{\alpha})\vert_x}$,
\[
\sum_{\alpha=1}^n \frac{\partial L}{\partial v^{\alpha}}(v)\, w^{\alpha}
=g_v(v,w), \qquad
\sum_{\alpha,\beta=1}^n g_{\alpha \beta}(v)\, w^{\alpha}\, w^{\beta} =g_v(w,w).
\]
Thus, if $g_v(w,w) \ge 0$ we immediately find
\[
\sum_{\alpha,\beta=1}^n \frac{\partial \mathcal{L}_q}{\partial v^{\alpha} \partial v^{\beta}}(v)\, w^{\alpha}\, w^{\beta} \ge 0.
\]
If instead $g_v(w,w)<0$, i.e.~$w$ is $g_v$-timelike, then the reverse Cauchy--Schwarz inequality for $g_v$ yields
\[
\big( g_v(v,w) \big)^2
\ge g_v(v,v)\, g_v(w,w)
= -F(v)^2\, g_v(w,w).
\]
Since the identity $g_v(v,w)=0$ holds only when $w=0$ or $g_v(w,w)>0$, we get the  positive-definiteness.
Hence $\mathcal{L}_q$ is strictly convex on $\Omega_x$ and convex on $T_xM$ (by the convexity of the set $\Omega_x$). \end{proof}

In the case of $q=1$, note that $\mathcal{L}_1=-F$ is positively $1$-homogeneous and fails to be strictly convex in radial directions. 
In order to formulate the corresponding next result, recall that a function $u\colon T_xM \longrightarrow [-\infty,+\infty]$ is said to be \emph{subdifferentiable} at $v \in T_xM$ with a \emph{subgradient} $\zeta \in T_x^*M$ if $u(v) \in \R$ and
\[
u(v+w) \ge u(v) +\zeta(w) +o(|w|_h)
\]
holds for $w \in T_xM$ close to $0$, where $h$ is an auxiliary Riemannian metric.
We denote by $\partial_\bullet u(v)$ the set of all subgradients.
(See \eqref{eq:subgrad} below for the case of functions on $M$ and the concept of superdifferentiability.)

\begin{corollary}[The case of $q=1$]\label{Cor:Classical Lagrangian}
Fix $x \in M$.
\begin{enumerate}[label=\textnormal{(\roman*)}]
\item The Lagrangian $\mathcal{L}_1$ is convex on $T_xM$.

\item We have
\begin{equation*}
\mathcal{H}_1(\zeta) =\begin{cases}
0 & \text{if } \zeta \in \Omega_x^* \text{ and } L^*(\zeta) \le -1/2,\\
+\infty & \text{otherwise}.
\end{cases}
\end{equation*}

\item
$\smash{\mathcal{L}_1\vert_{T_xM}}$ is not subdifferentiable on $\smash{(\mathcal{L}_1\vert_{T_xM})^{-1}(0) \setminus \{0\}}$.
At $0 \in T_xM$, $\smash{\mathcal{L}_1\vert_{T_xM}}$ is subdifferentiable with
\[
\partial_\bullet (\mathcal{L}_1\vert_{T_xM})(0) =\mathcal{H}_1^{-1}(0) \cap T_x^*M.
\]
\end{enumerate}
\end{corollary}

\begin{proof}
(i)
This follows from the convexity of $\mathcal{L}_q$ by letting $q \to 1$.

(ii)
This is a standard fact following from \eqref{Eq:q-Hamiltonian} together with the reverse Cauchy--Schwarz inequality, or from \eqref{eq:H_q} by letting $p \to -\infty$.

(iii)
The failure of the subdifferentiability of $\mathcal{L}_1$ (or the superdifferentiability of $F$) on $\smash{(\mathcal{L}_1\vert_{T_xM})^{-1}(0) \setminus \{0\}}$ is a consequence of the fact that, for $s>0$,
\[
\frac{\partial}{\partial t}\sqrt{t^2 -s^2}
=\frac{t}{\sqrt{t^2 -s^2}}
\to +\infty \qquad \textnormal{as }t \downarrow s.
\]

At $0 \in T_xM$, on the one hand, $\zeta \in \Omega_x^*$ with $L^*(\zeta) \le -1/2$ satisfies, by the reverse Cauchy--Schwarz inequality \eqref{eq:rCS},
\[
\zeta(v) \le -\sqrt{4\,L(v)\,L^*(\zeta)}
=\mathcal{L}_1(v)\, \sqrt{-2\,L^*(\zeta)}
\le \mathcal{L}_1(v)
\]
for arbitrary $v \in T_xM$, which implies $\smash{\zeta \in \partial_\bullet (\mathcal{L}_1\vert_{T_xM})(0)}$.
On the other hand, if $\smash{\zeta \in \partial_\bullet (\mathcal{L}_1\vert_{T_xM})(0)}$, then
\[
\zeta(v) =\frac{\zeta(c\,v)}{c} \le \frac{\mathcal{L}_1(c\,v) +o(c\,|v|_h)}{c}
\to -1\qquad\textnormal{as }c \downarrow 0
\]
for any $v \in F^{-1}(1) \cap \Omega_x$. Thus  $L^*(\zeta) \le -1/2$ by the definition of $L^*(\zeta)$.
\end{proof}

\subsection{The $q$-Lorentz--Finsler distance functions}\label{Sub:LF distance}

For $q\in (0,1]$, the \emph{$q$-action} of a Lipschitz curve $\gamma \colon [0,1] \longrightarrow M$ induced by $\mathcal{L}_q$ is
\[ \mathcal{A}_q(\gamma)
 :=\int_0^1 \mathcal{L}_q(\dot{\gamma}) \,\d t, \]
and we define $l_q \colon M\times M \longrightarrow [0,
+\infty]\cup\{-\infty\}$ by
\begin{equation}\label{Eq:lq Lorentz-Finsler distance}
l_q(x,y) :=-\inf_{\gamma} \mathcal{A}_q(\gamma),
\end{equation}
where $\gamma\colon[0,1]\longrightarrow M$ runs over all Lipschitz curves with $\gamma(0)=x$ and $\gamma(1)=y$.
We call $l_q$ the \emph{$q$-Lorentz--Finsler distance function};
for $q=1$, $l_1$ coincides with the Lorentz--Finsler distance function $l$ introduced in Subsection \ref{ssc:Fspacetime}.

For notational simplicity, we will adopt the conventions:
\begin{align}\label{Eq:Conventions}
(-\infty) + \infty := -\infty, \qquad
(-\infty)^q = (-\infty)^{1/q} := -\infty.
\end{align}
Then $l$ satisfies the \emph{reverse triangle inequality}
\begin{align}\label{eq:reverse triangle}
    l(x,y) \geq l(x,z) + l(z,y),
\end{align}
and, for every $x,y\in M$ with $l(x,y)>-\infty$ and every $q\in (0,1)$, we find 
\begin{align}\label{Eq:Lorentz-Finsler distance}
l(x,y)=\big( q\,l_q(x,y) \big)^{1/q}
\end{align}
(by re-parametrizing an appropriate curve $\gamma$ so that $\mathcal{L}_1(\dot{\gamma})$ is constant almost everywhere).
In particular, the sets $\{l_q > 0\}$ and $\{l_q \geq 0\}$ do not depend on the exponent $q\in(0,1]$.

The hypothesized global hyperbolicity of $(M,L)$ entails finer properties of the Lorentz--Finsler distance function $l$. First of all, by the relation \eqref{Eq:Lorentz-Finsler distance}, we immediately observe the following.

\begin{lemma}[Continuity of $l_q$]\label{Le:l continuous}
For each $q\in (0,1]$, $l_q$ is finite and continuous on $\{l_q\geq 0\}$, and upper semi-continuous on $M \times M$.
\end{lemma}

For analyzing the behavior of $l$, we introduce the singular set.

\begin{definition}[Singular set of $l$]\label{Def:sing l}
We define $\sing (l)$ as the set of all pairs $(x,y)\in M\times M$ such that $l(x,y) \leq 0$ or no maximizing geodesic contains $x$ and $y$ in its relative interior.
\end{definition}

Then, thanks to Lemma~\ref{Le:l continuous},
it is not difficult to show that $\sing(l)$ is a closed set; see \cite[Theorem~3.6]{Mc} for instance.


The following two lemmas are shown in exactly the  same way as \cite[Lemmas~2.4, 2.5]{Mc}.
We remark that the limit curve theorems in \cite{beem1996} generalize to our Lorentz--Finsler framework \cite[Proposition~6.1]{minguzzi2015c}.

\begin{lemma}[Continuity of intermediate points]\label{Le:zs}
For every $t\in[0,1]$ and every $(x,y)\in (M\times M)\setminus \sing (l)$, there exists a unique point $z\in M$, henceforth denoted by $z_t(x,y)$, such that
\begin{align*}
l(x,z) = t\,l(x,y), \qquad
l(z,y) = (1-t)\,l(x,y).
\end{align*}
Moreover, $z_t(x,y)$ depends smoothly on $(t,x,y) \in [0,1]\times ((M\times M)\setminus \sing (l))$.
\end{lemma}

\begin{lemma}[Compactness of intermediate point sets]\label{Le:Zs}
For $x,y\in M$, $t\in [0,1]$, and $E\subset M\times M$, we set 
\begin{align*}
Z_t(x,y) := \big\lbrace z\in M \mid l(x,z) = t\,l(x,y),\,l(z,y) = (1-t)\,l(x,y)\big\rbrace
\end{align*}
if $l(x,y) \geq 0$ and $Z_t(x,y) := \emptyset$ otherwise, as well as
\begin{align*}
Z_t(E) := \bigcup_{(x,y)\in E} Z_t(x,y),\qquad
Z(E) := \bigcup_{t\in[0,1]} Z_t(E).
\end{align*}
Then $Z(E)$ is precompact if $E$ is. If $E$ is compact, so are $Z(E)$ and $Z_t(E)$.
\end{lemma}

\begin{remark}\label{Re:Zs s-intermediate points}
Take $x,y \in M$ with $l(x,y) > 0$.
For $t \in (0,1)$, by concatenating timelike geodesics, note that $z\in Z_t(x,y)$ if and only if $z$ is the $t$-interme\-di\-ate point of a timelike maximizing geodesic from $x$ to $y$.
Similarly, we have $Z_0(x,y) = \{x\}$ and $Z_1(x,y) = \{y\}$ since maximizing curves do not change the causal type.
\end{remark}




\subsection{Convexity, concavity, and smoothness}\label{Sec:Convex analytic}

\begin{definition}[Semi-convexity and semi-concavity]\label{Def:Semiconvexity}
Let $U\subset M$ be open.
A function $u\colon U \longrightarrow\R$ is said to be \emph{semi-convex} if for some smooth Rieman\-nian metric $h$ on $M$ there exists $c\in\R$ such that
\begin{align*}
    \liminf_{w\to 0} \frac{u(\exp^h_x(w)) + u(\exp^h_x (-w)) - 2\,u(x)}{2\,\vert w\vert_h^2} \geq c
\end{align*}
for every $x\in U$, where $w \in T_xM$.
The largest possible such constant $c$ is called the \emph{semi-convexity constant} of $u$ on $U$. Lastly, we call $u$ \emph{semi-concave} if $-u$ is semi-convex.
\end{definition}

We remark that, if $U$ is relatively compact, the semi-convexity is independent of the choice of $h$, whereas the semi-convexity constant depends on $h$.

\begin{proposition}[Semi-convexity of $l$]\label{Pr:Semiconvexity Lorentz-Finsler distance}
Let $h$ be a smooth Riemannian metric on $M$, and let $(x,y) \in\{l>0\}$. Then
\begin{align*}
    c(x,y) := \liminf_{w\to 0} \frac{l(\exp_x^h(w),y) + l(\exp_x^h(-w), y) - 2\,l(x,y)}{2\,\vert w\vert_h^2}
\end{align*}
is real-valued, where $w \in T_xM$.
Moreover, the quantity $c(x,y)$ depends continuously on $(x,y)\in \{l>0\}$.
\end{proposition}

\begin{proof}
Let $\gamma\colon [0,1] \longrightarrow M$ be a maximizing geodesic from $x$ to $y$.
Take $w \in T_xM$ with $|w|_h=1$ and extend it to a vector field $P$ along $\gamma$ which is parallel with respect to $g_{\dot{\gamma}}$, namely $\smash{D^{\dot{\gamma}}_{\dot{\gamma}}P = 0}$ and $P(0)=w$.
Consider the vector field
\[ W(t):=(1-t)P(t) \]
along $\gamma$ and define a smooth variation $\sigma\colon [0,1] \times (-\varepsilon,\varepsilon) \longrightarrow M$ of $\gamma$, where $\varepsilon>0$ is a fixed small parameter, by
\[ \sigma_s(t) =\sigma(t,s) :=\exp^h_{\gamma(t)}\big( sW(t) \big). \]
Note that $\sigma_0=\gamma$ and $\sigma_s(1)=\gamma(1)=y$ for all $s\in (-\varepsilon,\varepsilon)$.
Set
\[ \mathfrak{L}(s) :=\int_0^1 F(\dot{\sigma}_s) \,\d t
 =-\mathcal{A}_1(\sigma_s). \]
Then it follows from \eqref{Eq:lq Lorentz-Finsler distance} that
\[ \frac{l(\sigma_s(0),y)+l(\sigma_{-s}(0),y)-2\,l(x,y)}{2\,s^2}
\ge \frac{\mathfrak{L}(s)+\mathfrak{L}(-s) -2\,\mathfrak{L}(0)}{2\,s^2}, \]
and the right-hand side smoothly depends on $x$, $y$, $\dot{\gamma}(0)$, and $w$.
Precisely, it is written by the second variation formula from Proposition~\ref{pr:2nd-var} as
\[ I(W^{\perp},W^{\perp}) -\bigg[ g_{\dot{\gamma}}\bigg( D^T_V V(\cdot,0), \frac{\dot{\gamma}}{l(x,y)} \bigg) \bigg]_0^1, \]
where $T:=\partial_t \sigma$, $V:=\partial_s \sigma$, and
\[ W^{\perp} :=W +g_{\dot{\gamma}}\bigg( W,\frac{\dot{\gamma}}{l(x,y)} \bigg) \,\frac{\dot{\gamma}}{l(x,y)}. \]
Observing that
\[ c(x,y)=\inf_{w,\gamma} \bigg( I(W^{\perp},W^{\perp}) -\bigg[ g_{\dot{\gamma}}\bigg( D^T_V V(\cdot,0), \frac{\dot{\gamma}}{l(x,y)} \bigg) \bigg]_0^1 \bigg) \]
shows the claim, where the infimum runs over all  $w \in T_xM$ with $|w|_h=1$ and over all maximizing geodesics $\gamma\colon[0,1]\longrightarrow M$ from $x$ to $y$.
\end{proof}

The following characterization of timelike cut loci plays an important role in the sequel. 
We refer to \cite[Theorem~3.5]{Mc} for the Lorentzian case and to \cite[Lemma~3.1]{ohta2009} for the Finsler case.

\begin{theorem}[Semi-concavity fails on timelike cut loci]\label{th:semiconc}
If $(x,y)\in \{l>0\} \cap \sing(l)$, then we have
\begin{equation*}
\limsup_{w \to 0 \in T_xM} \frac{l(\exp_x^h(w), y) +  l(\exp_x^h(-w),y) -2\,l(x,y)}{2\,\vert w\vert_h^2} = +\infty.
\end{equation*}
\end{theorem}

\begin{proof}
Note that either there are two distinct maximizing geodesics from $x$ to $y$, or $y$ is the first conjugate point of $x$ along a unique maximizing geodesic that connects $x$ to $y$.

In the former case, let $\gamma,\eta\colon [0,l(x,y)] \longrightarrow M$ be distinct maximizing geodesics of unit speed from $x$ to $y$ (i.e.\ $F(\dot{\gamma})=1$).
Set $v:=\dot{\gamma}(0)$, $w:=\dot{\eta}(0)$, and $y_{\varepsilon}:=\gamma(l(x,y)-\varepsilon)$ for small $\varepsilon >0$.
Then, on the one hand, we deduce from the reverse triangle inequality and the first variation formula from Lemma~\ref{lm:1st-var} that
\begin{align*}
l\big( \eta(-t),y \big) -l(x,y)
&\ge l\big( \eta(-t),y_{\varepsilon} \big) +\varepsilon - (l(x,y_{\varepsilon}) +\varepsilon ) \\
&= g_v(v,-t\,w) +O(t^2).
\end{align*}
On the other hand,
\[ l\big( \eta(t),y \big) -l(x,y) = -t. \]
Thus we have
\begin{align*}
&\frac{l(\eta(t),y) +l(\eta(-t),y) -2\,l(x,y)}{t^2}\\
&\qquad\qquad \ge -\frac{g_v(v,w)+1}{t} +\frac{1}{t^2}\,O(t^2)
\to +\infty\qquad\textnormal{as }t\to 0
\end{align*}
since by the reverse Cauchy--Schwarz inequality \eqref{eq:rCS},
\[ g_v(v,w)=\mathscr{L}(v)(w)
< -\sqrt{4\,L(v)\,L(w)} =-1. \]

In the latter case, note first that, along the argument in the proof of \cite[Theorem~3.5]{Mc}, it is sufficient to show
\[ \limsup_{w \to 0 \in T_xM} \frac{l(\xi_w(1), y) +  l(\xi_w(-1),y) -2\,l(x,y)}{2\,\vert w\vert_h^2} = +\infty, \]
where $\xi_w\colon [-1,1]\longrightarrow M$ denotes the geodesic with $\smash{\dot{\xi}_w(0)=w}$ with respect to $L$ (where the values $\xi_w(\pm 1)$ exist provided  $|w|_h$ is sufficiently small).
Let $J$ be a Jacobi field along the unique maximizing geodesic $\gamma\colon [0,1] \longrightarrow M$ from $x$ to $y$ vanishing only at $0$ and $1$.
Then $\smash{w:=D^{\dot{\gamma}}_{\dot{\gamma}}J(0) \in T_xM \setminus \{0\}}$ and, in the same manner as in the proof of Proposition \ref{Pr:Semiconvexity Lorentz-Finsler distance}, we consider the vector field $P$ along $\gamma$ with $\smash{D^{\dot{\gamma}}_{\dot{\gamma}}P = 0}$ and $P(0)=w$, and set 
\[
W(t):=(1-t)P(t).
\]
We further define $J_{\varepsilon}:=J+\varepsilon W$ for small $\varepsilon>0$.
Observe that $J_{\varepsilon}(0)=\varepsilon w$ and $J_{\varepsilon}(1)=0$.
We remark that $J$ is $g_{\dot{\gamma}}$-orthogonal to $\dot{\gamma}$ since it is a Jacobi field vanishing at $t=0$ and $t=1$, and that
\[ g_{\dot{\gamma}}\big( P(t),\dot{\gamma}(t) \big)
= g_{\dot{\gamma}}\big( w,\dot{\gamma}(0) \big)
= \frac{\d}{\d t}\bigg\vert_0 g_{\dot{\gamma}}(J,\dot{\gamma}) =0 \]
for all $t\in[0,1]$.
Hence, $J_{\varepsilon}$ is $\smash{g_{\dot{\gamma}}}$-orthogonal to $\dot{\gamma}$.
Now, for a small parameter $\varepsilon>0$ we consider the variation $\sigma\colon [0,1] \times (-\varepsilon,\varepsilon) \longrightarrow M$ given by
\[ \sigma_s(t) =\sigma(t,s) :=\xi_{J_{\varepsilon}(t)}(s). \]
Note that $\sigma_s(0)=\xi_w(\varepsilon\, s)$.
Setting again
\begin{equation*}
\mathfrak{L}(s) :=\int_0^1 F(\dot{\sigma}_s) \,\d t,
\end{equation*}
we observe from the second variation formula from Proposition~\ref{pr:2nd-var} that
\begin{align*}
&\frac{l(\xi_w(\varepsilon\, s), y) +l(\xi_w(-\varepsilon\, s),y) -2\,l(x,y)}{2\,(\varepsilon\, s)^2} \\
&\qquad\qquad \ge \frac{\mathfrak{L}''(0)}{\varepsilon^2} +\frac{o(\varepsilon^2)}{\varepsilon^2}
= \frac{I(J_{\varepsilon},J_{\varepsilon})}{\varepsilon^2} +\frac{g_{\dot{\gamma}}(D^T_V V(0,0),\dot{\gamma}(0))}{\varepsilon^2 \,F(\dot{\gamma})} +\frac{o(\varepsilon^2)}{\varepsilon^2},
\end{align*}
where $T=\partial_t \sigma$ and $V:=\partial_s \sigma$.
As $J$ is a Jacobi field, the first summand of the latter term is rewritten with the help of \eqref{eq:index-Jacobi} as
\[
\frac{I(J,J)}{\varepsilon^2} +2\,\frac{I(J,W)}{\varepsilon} +I(W,W)
 = \frac{2\,g_{\dot{\gamma}}(w,w)}{\varepsilon\, l(x,y)} +I(W,W).
\]
Note also that, since $w \neq 0$ and $\smash{g_{\dot{\gamma}}\big(w,\dot{\gamma}(0)\big)=0}$, we have $\smash{g_{\dot{\gamma}}(w,w)>0}$.
As for the second summand above, we directly infer from the definition of $\sigma$ that $\smash{D^V_V V=0}$.
Then we have
\[ g_{\dot{\gamma}} \big( D^T_V V(0,0),\dot{\gamma}(0) \big)
= g_{\dot{\gamma}} \big( D^T_V V(0,0) -D^V_V V(0,0),\dot{\gamma}(0) \big)
=-\varepsilon^2\, \mathcal{T}_{\dot{\gamma}(0)}(w), \]
where
\[ \mathcal{T}_v(w):= \sum_{\alpha,\beta,\delta,\lambda=1}^n g_{\alpha \lambda}(v)\,\big(\Gamma^{\alpha}_{\beta \delta}(w) -\Gamma^{\alpha}_{\beta \delta}(v) \big)\, w^{\beta}\, w^{\delta}\, v^{\lambda} \]
is called the \emph{$\mathbf{T}$-curvature} (or the \emph{tangent curvature}) measuring the variation of tangent spaces; see \cite[\S6.3.2, \S8.3.3]{Obook}. Therefore, we deduce that
\begin{align*}
&\frac{l(\xi_w(\varepsilon\, s), y) +l(\xi_w(-\varepsilon\, s),y) -2\,l(x,y)}{2\,(\varepsilon\, s)^2}
\\
&\qquad\qquad \ge \frac{2\,g_{\dot{\gamma}}(w,w)}{\varepsilon\, l(x,y)} +I(W,W) -\frac{\mathcal{T}_{\dot{\gamma}(0)}(w)}{F(\dot{\gamma})} +\frac{o(\varepsilon^2)}{\varepsilon^2} \to +\infty
\end{align*}
as $\varepsilon\to 0$.
This completes the proof.
\end{proof}

Finally, we address smoothness properties of $l_q$.
We start with the usual case of $q=1$ in Theorem~\ref{Th:Lorentz distance}, based on which we establish the case of $q\in (0,1)$ in Corollary \ref{Cor:Twist}.
From Lemma~\ref{Le:l continuous}, we already know that $l_q$ is upper semi-continuous on $M\times M$ and continuous on $\{l\geq 0\}$.
In addition, recall from Subsection~\ref{Sub:LF distance} that the singular set $\sing(l)$ from Definition~\ref{Def:sing l} is closed.

For a function $u\colon M \longrightarrow [-\infty,+\infty]$,
we say that $u$ is \emph{subdifferentiable} at $x \in M$ with a \emph{subgradient} $\zeta \in T_x^*M$ if $u(x) \in \R$ and
\begin{equation}\label{eq:subgrad}
u\big( \exp_x(v) \big) \ge u(x) +\zeta(v) +o(|v|_h)
\end{equation}
for $v \in T_xM$ close to $0$, where $h$ is an auxiliary Riemannian metric.
The set of subgradients of $u$ at $x$ will be denoted by $\partial_\bullet u(x)$.
The \emph{superdifferentiability} and \emph{supergradients} $\partial^\bullet u(x)$ are defined in the same way with the reverse inequality.

\begin{theorem}[Smoothness of $l$]\label{Th:Lorentz distance}
The Lorentz--Finsler distance function $l$ satisfies the following properties.
\begin{enumerate}[label=\textnormal{(\roman*)}]
    \item It is smooth precisely on the complement of $\sing(l)$.
    
    \item It is locally Lipschitz and locally semi-convex on the open superlevel set $\{l>0\}$ of $M \times M$.
    Furthermore, for $(x,y)\in \{l>0\}$ and any maximizing geodesic $\gamma\colon [0,l(x,y)] \longrightarrow M$ of unit speed from $x$ to $y$, we have
    \[
    \mathscr{L}\big( \dot{\gamma}(0) \big) \in \partial_\bullet \big(l(\cdot,y)\big)(x), \qquad
    -\mathscr{L}\big( \dot{\gamma}\big( l(x,y) \big) \big) \in \partial_\bullet \big(l(x,\cdot)\big)(y).
    \]
    
    \item For $(x,y)\in \{l=0\}$, $l(\cdot,y)$ is not superdifferentiable at $x$ unless $x=y$, and $\smash{\partial^\bullet\big(l(\cdot,x)\big)(x)=\mathcal{H}_1^{-1}(0) \cap T_x^*M}$.
\end{enumerate}
\end{theorem}

\begin{proof}
(i) and (ii) are seen in the same way as \cite[Theorem~3.6]{Mc} with the help of Proposition~\ref{Pr:Semiconvexity Lorentz-Finsler distance} and Theorem~\ref{th:semiconc}.
(We remark that $-\mathscr{L}(\dot{\gamma}(l(x,y)))$ can be regarded as $\overline{\mathscr{L}}(\dot{\overline{\gamma}}(0))$ for the Legendre transform $\overline{\mathscr{L}}$ with respect to the \emph{reverse Lorentz--Finsler structure} $\overline{L}(v):=L(-v)$ and the reverse curve $\overline{\gamma}(t):=\gamma(l(x,y)-t)$ which is a maximizing geodesic with respect to $\overline{L}$.)
(iii) also follows from a similar argument to \cite[Theorem~3.6]{Mc} by virtue of Corollary~\ref{Cor:Classical Lagrangian}(iii).
\end{proof}

We recall from \eqref{Eq:Lorentz-Finsler distance} that $l_q(x,y)=l(x,y)^q/q$ provided $l(x,y)>-\infty$.

\begin{corollary}[Properties of $l_q$]\label{Cor:Twist}
For $q\in (0,1)$, $l_q$ satisfies the following.
\begin{enumerate}[label=\textnormal{(\roman*)}]
    \item $l_q$ is smooth precisely on $(M \times M) \setminus \sing(l)$ and is locally Lipschitz and locally semi-convex on $\{l>0\}$.
    
    \item Given $y\in M$, if $l_q(\cdot,y)$ has a supergradient $\zeta \in T_x^*M$,
    then $l(x,y)>0$ necessarily holds and we have
    \[
    y=\exp_x \Bigg( \sum_{\alpha=1}^n \frac{\partial \mathcal{H}_q}{\partial \zeta_{\alpha}}(\zeta) \frac{\partial}{\partial x^{\alpha}} \Bigg).
    \]
    
    \item The following matrix is non-degenerate for $(x,y) \in (M\times M)\setminus \sing(l)$:
    \[ \bigg( \frac{\partial^2 l_q}{\partial x^\alpha \partial y^\beta}(x,y) \bigg)_{\alpha,\beta=1}^n.
    \]
    
    \item For every $(x,y)\in \{l>0\} \cap \sing(l)$ and every smooth Riemannian metric $h$ on $M$, we have
    \begin{equation*}
        \sup_{0<\vert w\vert_h < 1} \frac{l_q(\exp_x^h(w),y) + l_q(\exp_x^h(-w),y) - 2\,l_q(x,y)}{2\,\vert w\vert_h^2} = +\infty.
    \end{equation*}
\end{enumerate}
\end{corollary}


\begin{proof}
(i)
This is shown along the same lines as \cite[Corollary~3.7]{Mc} thanks to Theorem~\ref{Th:Lorentz distance}(i), (ii).

(ii)
It follows from Theorem~\ref{Th:Lorentz distance}(iii) that $l(x,y)>0$.
Then, together with the subdifferentiability asserted in Theorem~\ref{Th:Lorentz distance}(ii), $l(\cdot,y)$ is differentiable at $x$ and we have the claim by noting that
\[
\zeta =\d \big(l_q(\cdot,y)\big)_x =l(x,y)^{q-1} \cdot \d\big(l(\cdot,y)\big)_x
\]
and, recalling Lemma~\ref{Le:Convex Lagrangian},
\[
\sum_{\alpha=1}^n \frac{\partial \mathcal{H}_q}{\partial \zeta_{\alpha}}(\zeta)\, \frac{\partial}{\partial x^{\alpha}}
=F^*(\zeta)^{p-2}\, \mathscr{L}^*(\zeta)
=l(x,y)\, \frac{\mathscr{L}^*(\zeta)}{F^*(\zeta)}
=l(x,y)\ \dot{\gamma}(0).
\]

(iii)
Note first that $\smash{(x^{\alpha})_{\alpha=1}^n}$ and $\smash{(y^{\beta})_{\beta=1}^n}$ are local coordinate systems around $x$ and $y$, respectively.
If the matrix in question is degenerate, then there exists $v \in T_xM \setminus \{0\}$ such that
\[
\sum_{\alpha=1}^n v^\alpha\, \frac{\partial^2 l_q}{\partial x^\alpha \partial y^\beta}(x,y)=0
\]
for all $\beta$.
By the first variation formula from Lemma~\ref{lm:1st-var}, this implies that the function
$\smash{f(z):=g_{w_z}(v,w_z)=\mathscr{L}(w_z)(v)}$ with $\smash{w_z:=\exp_x^{-1}(z)}$ satisfies $\d f_y=0$, which cannot happen.

(iv) 
This is an immediate consequence of Theorem~\ref{th:semiconc}.
\end{proof}

The non-degeneracy in (iii) corresponds to a  \emph{twist condition} for $l^q$, cf.~\cite{villani2009}.


\section{Geodesics of probability measures}\label{sc:geod}

Now we study the spacetime geometry of the space $\Prob(M)$ of Borel probability measures on $M$.
This geometry is described by the concept of $q$-geodesics taken from \cite[Definition 1.1]{Mc}, cf.~Definition \ref{Def:q-geodesics}.
We continue assuming the global hyperbolicity of $(M,L)$.

In Section \ref{Sec:Lorentz-Wasserstein} we recall basics of optimal transport in Lorentzian geometry (cf.~e.g.~\cite{cavalletti2020,eckstein2017,Mc}), which carry over to the Lorentz--Fins\-ler setting with no change.
Then we show $q$-geodesics to exist in large generality in Section \ref{sub:Existence q geodesics}.
Finally, based on the duality theory outlined in Section \ref{Sec:Kantorovich duality}, in Section \ref{Sec:Characterization} we address the uniqueness of $q$-geodesics (and simultaneously the one of couplings attaining the supremum in \eqref{Eq:lq distance}).
Here we follow \cite[Sections~2, 4, 5]{Mc}.

\subsection{The $q$-Lorentz--Wasserstein distance and $q$-geodesics}\label{Sec:Lorentz-Wasserstein}

Let $\proj_i\colon M\times M\longrightarrow M$ denote the projection map defined by $\smash{\proj_i(x) := x_i}$, where $i = 1,2$; similarly, for $i,j = 1,2,3$ with $i\neq j$ we will define $\smash{\proj_{ij}}\colon M\times M\times M\longrightarrow M\times M$ by $\smash{\proj_{ij}(x) := (x_i,x_j)}$. As said, $\Prob(M)$ is the space of Borel probability measures on $M$.
Let $\Prob_\comp(M)$ be the set of all $\mu\in\Prob(M)$ with compact support $\supp\mu\subset M$.
Given any Borel measure $\meas$ on $M$, let $\Prob^\ac(M,\meas)$ denote the space of all $\meas$-absolutely continuous elements of $\Prob(M)$, and set $\Prob_\comp^\ac(M,\meas) := \Prob^\ac(M,\meas) \cap \Prob_\comp(M)$. 

Given $\mu,\nu\in \Prob(M)$, let $\Pi(\mu,\nu)$ be  the set of all couplings $\pi$ of $\mu$ and $\nu$.
We call such a coupling \emph{causal} if $\pi[\{l\geq 0\}]=1$, and \emph{chronological} if $\pi[\{l>0\}]=1$.
By Lemma \ref{Le:l continuous}, $\pi$ is causal if and only if $\supp\pi \subset \{l\geq 0\}$. 

Throughout, we fix $q\in (0,1]$, and we state explicitly when $q<1$ is assumed in addition.
Recalling \eqref{Eq:Lorentz-Finsler distance} and \eqref{Eq:Conventions}, let $\smash{\ell_q}\colon \Prob(M)\times \Prob(M) \longrightarrow [0,+\infty] \cup \{-\infty\}$ denote the \emph{$q$-Lo\-rentz--Wasserstein distance} 
\begin{equation}\label{Eq:lq distance}
    \ell_q(\mu,\nu) := \sup_{\pi\in \Pi(\mu,\nu)} \bigg(\!\int_{M\times M} l(x,y)^q\,\d \pi(x,y)\bigg)^{1/q}.
\end{equation}
Our choice of the cost function $l^q$ which can drop to $-\infty$ ensures  $\smash{\ell_q(\mu,\nu)\geq 0}$ if and only if $\mu$ and $\nu$ admit a causal coupling. 

We term $\pi\in \Pi(\mu,\nu)$ \emph{$\ell_q$-optimal} if it is causal and it attains the supremum in \eqref{Eq:lq distance}.
We occasionally employ the subsequent standard result, which follows from Lemma \ref{Le:l continuous} and \cite[Theorem 4.1]{villani2009}.

\begin{lemma}[Existence of $\ell_q$-optimal couplings]\label{Le:Existence optimal couplings} 
Assume $\mu,\nu\in\mathcal{P}_\comp(M)$ to obey  $\ell_q(\mu,\nu)\geq 0$. Then there exists an $\smash{\ell_q}$-optimal coupling of $\mu$ and $\nu$.
\end{lemma}

\begin{definition}[Geodesics of probability measures]\label{Def:q-geodesics} Given any $q\in (0,1]$, a family $(\mu_t)_{t\in[0,1]}$ in $\Prob(M)$ is called a \emph{$q$-geo\-desic} if
\begin{enumerate}[label=\textnormal{(\arabic*)}]
    \item $\ell_q(\mu_0,\mu_1)$ is positive and finite, and
    \item for every $s,t\in [0,1]$ with $s<t$, 
\begin{equation*}
    \ell_q(\mu_s,\mu_t) = (t-s)\,\ell_q(\mu_0,\mu_1).
\end{equation*}
\end{enumerate}
\end{definition}

Akin to the corresponding property of Wasserstein distances, $\ell_q$ satisfies the \emph{reverse} triangle inequality.
This follows from a standard gluing argument, cf.~e.g.~\cite[Theorem 13]{eckstein2017}.
To analyze $q$-geodesics, we also need to study  cases of equality.
This is collected in  Proposition \ref{Pr:Reverse triangle} taken from \cite[Proposition 2.9]{Mc}.
One immediate consequence of it is that the property of having endpoints with compact support propagates through the interior of $q$-geodesics, cf.~Corollary \ref{Cor:Compact support}, a proof of which we give for convenience.

Recall Lemma~\ref{Le:Zs} for the definitions of $Z_s(E)$ and $Z(E)$.

\begin{proposition}[Reverse triangle inequality and cases of equality]\label{Pr:Reverse triangle} 
For every  $\mu_1,\mu_2,\mu_3\in\Prob(M)$, we have
\begin{equation}\label{Eq:Reverse lp}
\ell_q(\mu_1,\mu_3) \geq \ell_q(\mu_1,\mu_2) + \ell_q(\mu_2,\mu_3).
\end{equation}
Moreover, if $\ell_q(\mu_1,\mu_2) + \ell_q(\mu_2,\mu_3) \neq \pm\infty$ according to our conventions \eqref{Eq:Conventions} and there are Borel sets $X_1,X_3\subset M$ with $\mu_1[X_1]=\mu_3[X_3]=1$ yet $\mu_2[Z(X_1\times X_3)]<1$, then the inequality in \eqref{Eq:Reverse lp} is strict.

Conversely, assume that 
\begin{enumerate}[label=\textnormal{(\arabic*)}]
\item $q<1$, 
\item $\smash{\ell_q(\mu_1,\mu_3)}$ is positive and finite,
\item equality holds in \eqref{Eq:Reverse lp}, and 
\item $\mu_1$ and $\mu_2$ as well as $\mu_2$ and $\mu_3$ admit $\ell_q$-optimal couplings
\textnormal{(}e.g.~when $\mu_1$, $\mu_2$, and $\mu_3$ are compactly supported, cf.~Lemma \textnormal{\ref{Le:Existence optimal couplings})}.
\end{enumerate} 
Then there exists a measure $\omega\in \Prob(M\times M\times M)$ such that $\smash{\pi_{ij} := (\proj_{ij})_\sharp\pi}$ is an $\ell_q$-optimal coupling of its marginals $\mu_i$ and $\mu_j$ for every $i,j \in \{1,2,3\}$ with $i<j$, and every $(x,z,y)\in\supp \omega$ satisfies
\begin{equation*}
l(x,z) = t\,l(x,y),\qquad
l(z,y) = (1-t)\,l(x,y),
\end{equation*}
where $t := \ell_q(\mu_1,\mu_2)/\ell_q(\mu_1,\mu_3)$.
Furthermore, if $\pi_{13}$ is concentrated on a Borel set $E\subset M\times M$, then $\mu_2$ is concentrated on $Z_t(E)$.
In particular, if $Z_t(x,y)$ is a singleton $\{z_t(x,y)\}$ for $\pi_{13}$-a.e.~$(x,y)\in M\times M$, then 
\begin{equation*}
    \omega = (z_0, z_t, z_1)_\sharp\pi_{13},\qquad
    \mu_2 = (z_t)_\sharp\pi_{13}.
\end{equation*}
\end{proposition}

\begin{corollary}[Interpolants inherit compact support]\label{Cor:Compact support}
Let $(\mu_t)_{t\in[0,1]}$ be a $q$-geodesic in $\Prob(M)$ for some $q\in (0,1)$.
If the endpoints $\mu_0$ and $\mu_1$ are compactly supported, so is the intermediate point $\mu_t$ for every $t\in(0,1)$; more precisely, $\supp \mu_t \subset Z_t(\supp \mu_0\times\supp \mu_1)$.
\end{corollary}

\begin{proof}
Lemma \ref{Le:Zs} yields the compactness of $Z(E)$ and $Z_t(E)$, where we set $E:= \supp\mu_0\times\supp\mu_1$.
Since $\smash{\ell_q(\mu_0,\mu_1)\in (0,+\infty)}$, applying the first part of Pro\-position~\ref{Pr:Reverse triangle} yields $\mu_t[Z(E)]=1$.
In particular, $\mu_t$ has compact support.
Then, moreover, it follows from Lemma~\ref{Le:Existence optimal couplings} that $\mu_0$ and $\mu_t$ as well as $\mu_t$ and $\mu_1$ admit $\ell_q$-optimal couplings.
Hence, by the latter part of Proposition \ref{Pr:Reverse triangle}, we obtain in fact $\mu_t[Z_t(E)]=1$. 
\end{proof}

\subsection{Existence of $q$-geodesics}\label{sub:Existence q geodesics}

Employing a standard measurable selection argument \cite{ambrosio2013,villani2009}, we now construct $q$-geodesics by taking a chronological $\ell_q$-optimal coupling $\pi$ between appropriate endpoints $\mu_0$ and $\mu_1$ under consideration, connecting $\pi$-a.e.~$(x,y)\in M\times M$ by a maximizing timelike geodesic, and a lifting procedure.

We need the subsequent technical Lemma \ref{Le:Measurable extension}; its proof uses  Lemma \ref{Le:l continuous} and Remark \ref{Re:Zs s-intermediate points} and is analogous to its Lorentzian version \cite[Lemma 2.8]{Mc}.
Recall that a real-valued map defined on a topological space $X$ is termed \emph{universally measurable} if for every Borel probability measure $\mu$ on $X$, it is measurable with respect to the completion of the Borel $\sigma$-algebra of $X$ with respect to $\mu$.

\begin{lemma}[Selecting intermediate points on the timelike cut locus]\label{Le:Measurable extension}
For every $t\in[0,1]$, the map $z_t\colon (M\times M)\setminus \sing(l) \longrightarrow M$ from Lemma \textnormal{\ref{Le:zs}} has a non-relabeled universally measurable extension $z_t\colon \{l>0\}\longrightarrow M$ such that for every $(x,y)\in \{l>0\}$, $t\longmapsto z_t(x,y)$ is a maximizing geodesic from $x$ to $y$.
\end{lemma}


\begin{corollary}[Existence of $q$-geodesics] 
Let $\mu_0,\mu_1\in\Prob(M)$, and let $\pi \in \Pi(\mu_0,\mu_1)$ be a chronological $\smash{\ell_q}$-optimal coupling with finite total $\smash{\ell_q}$-cost. Then the following hold.
\begin{enumerate}[label=\textnormal{(\roman*)}]
\item The assignment $\mu_t := (z_t)_\sharp\pi$, where $z_t$ is the map provided by Lemma \textnormal{\ref{Le:Measurable extension}}, defines a $q$-geodesic $(\mu_t)_{t\in[0,1]}$ connecting $\mu_0$ to $\mu_1$. 
\item For every $s,t\in[0,1]$ with $s<t$, $(z_s, 
z_t)_\sharp\pi$ is a chronological $\smash{\ell_q}$-optimal coupling of the members $\mu_s$ and $\mu_t$ of that $q$-geodesic.
\end{enumerate}
\end{corollary}

\begin{proof}
The coupling $(z_s, z_t)_\sharp\pi$ of $\mu_s$ and $\mu_t$ is well-defined and chronological by Lemma \ref{Le:Measurable extension}.
In addition,
\begin{align*}
\ell_q(\mu_s,\mu_t)^q &\geq \int_{M\times M} l\big(z_s(x,y), z_t(x,y)\big)^q\,\d\pi(x,y)\\
&= (t-s)^q\int_{M\times M} l(x,y)^q\,\d\pi(x,y)\\
&= (t-s)^q\,\ell_q(\mu_0,\mu_1)^q.
\end{align*}
Since $s$ and $t$ are arbitrary, this estimate combined with the reverse triangle inequality from Proposition \ref{Pr:Reverse triangle} forces the inequality appearing in the above computation to become equality.
Thus, $(\mu_t)_{t\in[0,1]}$ is a $q$-geodesic, and $(z_s,z_t)_\sharp\pi$ is an $\smash{\ell_q}$-optimal coupling of $\mu_s$ and $\mu_t$. 
\end{proof}

\subsection{Kantorovich duality}\label{Sec:Kantorovich duality}

Now we revisit the duality theory developed in \cite[Section 4]{Mc} in the Lorentzian setting (see also \cite[\S2.4]{cavalletti2020}). 
Therein, the additional difficulty is that $l$ drops to $-\infty$ at points not in causal relation, hence the infimum in \eqref{Eq:Kantorovich} below might not be attained in general (see \cite{kell2020} for an example).
To ensure the existence of the latter, cf.~Theorem \ref{Th:Inf attained}, the $q$-separation property stated in Definition \ref{Def:q-separated} below has been introduced in \cite[Definition 4.1]{Mc}.
Lemma \ref{Le:Existence separation} will show that this notion is not void: $q$-separation holds for every pair of measures for which \emph{every} transport stays away from the light cone in a uniform way. This property is in practice easier to verify than $q$-separation, yet it does not suffice to consider only such measures. Indeed, the  intermediate points $\mu_s$ and $\mu_t$ of a $q$-geodesic $(\mu_t)_{t\in[0,1]}$ will not satisfy this condition in general as their supports typically overlap for $s,t\in[0,1]$ close to each other. On the other hand, $q$-separation is well-behaved along $q$-geodesics, cf.~Proposition \ref{Pr:Star-shaped}. 

Since $l$ is bounded from above on compact subsets of $M\times M$, the standard Kan\-toro\-vich duality \cite[Theorem 5.10]{villani2009} yields the following.




\begin{proposition}[Basic duality]\label{Pr:Duality}
For every $\mu,\nu\in\Prob_\comp(M)$ we have
\begin{equation}\label{Eq:Kantorovich}
q^{-1}\, \ell_q(\mu_0,\mu_1)^q = \inf_{(u,v)}\bigg(\int_M u\,\d\mu + \int_M v\,\d\nu \bigg),
\end{equation}
where $(u,v)\colon M \longrightarrow \R^2$ runs over all pairs of functions in $L^1(M,\mu)\times L^1(M,\nu)$ satisfying $\smash{l^q/q \leq u\oplus v}$ on $M\times M$, where 
\begin{equation*}
(u\oplus v)(x,y) := u(x) + v(y).    
\end{equation*}
Moreover, given any $X\supset \supp\mu$ and $Y\supset\supp\nu$, the infimum in \eqref{Eq:Kantorovich} is unchanged if we restrict it to those pairs of lower semi-continuous functions $u\in L^1(M,\mu)$ and $v\in L^1(M,\nu)$ such that $u={^{(l^q)}}v$ on $X$ and $v= u^{(l^q)}$ on $Y$, where
\begin{align}\label{Eq:Transform}
\begin{split}
^{(l^q)}v(x) &:= \sup_{y\in Y} \big(q^{-1}\,l(x,y)^q - v(y)\big),\\
u^{(l^q)}(y) &:= \sup_{x\in X} \big(q^{-1}\,l(x,y)^q -  u(x)\big).
\end{split}
\end{align}
\end{proposition}

Note that the transforms in \eqref{Eq:Transform} depend on $X$ and $Y$, respectively.

\begin{definition}[$q$-separation]\label{Def:q-separated}
We term  $(\mu,\nu)\in \Prob_\comp(M)\times\Prob_\comp(M)$  \emph{$q$-separated} \textnormal{(}\emph{by} $(\pi,u,v)$\textnormal{)} if there exist $\pi\in \Pi(\mu,\nu)$ and lower semi-continuous functions $u\colon \supp\mu\longrightarrow \R\cup\{+\infty\}$ and $v\colon \supp\nu\longrightarrow \R\cup\{+\infty\}$ with the following properties.
\begin{enumerate}[label=\textnormal{(\arabic*)}]
\item We have $l^q/q \le u \oplus v$ on $\supp\mu\times\supp\nu$.
\item The equality set
\begin{equation}\label{Eq:Equality set S}
E := \big\lbrace(x,y)\in\supp\mu\times\supp\nu \mid u(x) + v(y) = q^{-1}\,l(x,y)^q\big\rbrace
\end{equation}
satisfies $\supp\pi \subset E \subset \{l>0\}$.
\end{enumerate}
\end{definition}

Note that the lower semi-continuity of $u$ and $v$ plus the upper semi-continuity of $l$ imply the compactness of $E$, cf.~Theorem \ref{Th:Inf attained} below.

We omit the technical proof of the following Theorem \ref{Th:Inf attained}, which can be found in \cite[Theorem 4.3]{Mc}.
It transfers \emph{verbatim} to the Lorentz--Finsler case by employing Theorem \ref{Th:Lorentz distance}, Corollary \ref{Cor:Twist}, as well as an auxiliary smooth Riemannian metric on $M$.

To formulate it, we recall that a set $E\subset M\times M$ is \emph{$l^q$-cyclically monotone} (cf.~e.g.~\cite[Definition 2.6]{cavalletti2020} and \cite[Definition 5.1]{villani2009}) if for every $k\in\N$, every $(x_1,y_1),\dots,(x_k,y_k)\in E$, and every permutation $\sigma\in \mathfrak{S}_k$ we have
\begin{equation*}
\sum_{i=1}^k l(x_i,y_i)^q \geq \sum_{i=1}^k l(x_i,y_{\sigma(i)})^q.
\end{equation*}
Every $\smash{\ell_q}$-optimal coupling is concentrated on an $l^q$-cyclically monotone set, but unlike the case of real-valued cost functions \cite[Theorem 5.10]{villani2009}, the converse requires an additional hypothesis on the  coupling, cf.~\cite[Proposition 2.8]{cavalletti2020}.

\begin{theorem}[Duality by $q$-separation]\label{Th:Inf attained}
Assume $(\mu,\nu)\in\Prob_\comp(M) \times\Prob_\comp(M)$  to be $q$-separated by $(\pi,u,v)$.
Then the following hold.
\begin{enumerate}[label=\textnormal{(\roman*)}]
\item We have $\smash{u= {^{(l^q)}}v}$ on $\supp\mu$ and $\smash{v=u^{(l^q)}}$ on $\supp\nu$ according to \eqref{Eq:Transform}.
\item The equality set $E$ from \eqref{Eq:Equality set S} is compact and $l^q$-cyclically monotone.
\item The coupling $\pi$ is $\smash{\ell_q}$-optimal, while $(u,v)$ minimizes the right-hand side of \eqref{Eq:Kantorovich}.
\item The functions $u$ and $v$ extend to semi-convex Lipschitz functions on open neighborhoods of $\supp \mu$ and $\supp \nu$, respectively, with Lipschitz and semi-convexity constants estimated by those of $\smash{l^q\vert_E}$.
\end{enumerate}
\end{theorem}

\begin{lemma}[Existence of $q$-separation]\label{Le:Existence separation}
Let  $\mu,\nu\in\Prob_\comp(M)$ satisfy  $\supp\mu \times\supp \nu \subset\{l>0\}$. Then $(\mu,\nu)$ is $q$-se\-parated.
\end{lemma}

\begin{proof}
By Lemma \ref{Le:l continuous}, the function $l$ is continuous, bounded, and bounded away from zero on $\supp \mu \times\supp \nu$.
Thus the supremum in \eqref{Eq:lq distance} and, by  \cite[Theorem 5.10]{villani2009}, the infimum in \eqref{Eq:Kantorovich} are attained by some $\pi\in\Pi(\mu,\nu)$ and uniformly continuous functions $u\colon\supp \mu \longrightarrow \R$ and $v\colon \supp \nu \longrightarrow \R$ of the form \eqref{Eq:Transform}, respectively.
This duality also implies $\pi$ to be concentrated on the equality set $E$ as in \eqref{Eq:Equality set S}, and $E\subset \supp \mu \times\supp \nu \subset \{l>0\}$. 
\end{proof}

\subsection{Characterization of $q$-geodesics}\label{Sec:Characterization}

Now we turn to the description of $q$-geodesics between $q$-separated endpoints, ultimately stated in Corollary \ref{Cor:Char geos}. To this aim, a central part is the solvability of the $\ell_q$-Monge problem shown in Theorem \ref{Th:Char optimal maps}. This is the Lorentz--Finsler analogue of the Brenier--McCann theorem in Riemannian geometry \cite{brenier1991, Mc2} and of the Lorentzian result \cite{Mc}. Comparable results are obtained in \cite{suhr2018} under certain geometric conditions on the support of the target measure.

Since this subsection constitutes a central part of our work, and in order to emphasize the difference between the Lorentz--Finsler and the genuine Lorentzian settings, we provide full proofs even if they are essentially identical to those in \cite[Section~5]{Mc}.

We start by collecting various properties linked to intermediate points of geodesics in $M$ and $\Prob(M)$.
Proposition \ref{Pr:Lagrangian trajectories} and Lemma \ref{Le:Variational char} are based on the following trivial consequence of \eqref{eq:reverse triangle}: for every triple $(x,z,y)\in M\times M\times M$ and every $t\in (0,1)$,
\begin{equation}\label{Eq:lq inequ with s}
    l(x,y) \geq t\,\frac{l(x,z)}{t} + (1-t)\,\frac{l(z,y)}{1-t}.
\end{equation}

\begin{proposition}[Lagrangian trajectories do not cross]\label{Pr:Lagrangian trajectories} Let $q\in (0,1)$, $t\in(0,1)$, and $x,x',y,y'\in M$ with $l(x,y) >0$ and $l(x',y')>0$. If the set $Z_t(x,y)$ intersects $Z_t(x',y')$ yet
\begin{equation}\label{Eq:other inequality}
    l(x,y')^q + l(x',y)^q \leq l(x,y)^q + l(x',y')^q,
\end{equation}
then we have $x=x'$ and $y=y'$. 
\end{proposition}

\begin{proof} Given any $z\in Z_t(x,y)\cap Z_t(x',y')$, the inequality \eqref{Eq:lq inequ with s} for the triple $(x,z,y')$, our conventions \eqref{Eq:Conventions}, and the concavity of $r\longmapsto r^q$ on $[0,+\infty)$ yield
\begin{align}\label{Eq:other inequality 2}
\begin{split}
    l(x,y')^q &\geq \big(l(x,z)+l(z,y') \big)^q\\
    &=\big(t\,l(x,y) + (1-t)\,l(x',y')\big)^q \\
    &\geq t\,l(x,y)^q + (1-t)\,l(x',y')^q,
\end{split}
\end{align}
and analogously 
\begin{equation}\label{Eq:other inequality 3}
    l(x',y)^q \geq  t\,l(x',y')^q + (1-t)\,l(x,y)^q.
\end{equation}
This gives $l(x,y') > 0$ and $l(x',y)> 0$.
Furthermore, by \eqref{Eq:other inequality}, summing up the previous inequalities forces equalities all over \eqref{Eq:lq inequ with s} --- namely for the triples $(x,z,y')$ and $(x',z,y)$ ---, \eqref{Eq:other inequality 2}, and \eqref{Eq:other inequality 3}.
By Remark \ref{Re:Zs s-intermediate points}, equality in \eqref{Eq:lq inequ with s} implies $z$ to lie in the interior of a maximizing geodesic from $x$ to $y'$ and in the interior of a maximizing geodesic from $x'$ to $y$. Hence, all five points under consideration lie on the same maximizing geodesic $\gamma\colon[0,1]\longrightarrow M$.
Finally, identity in the second inequality of \eqref{Eq:other inequality 2}, together with \emph{strict} concavity of $r\longmapsto r^q$ on $[0,+\infty)$, implies $l(x,y) = l(x',y')$.
Since $z$ divides the segments of $\gamma$ from $x$ to $y$ and from $x'$ to $y'$ in the same ratio, we obtain $x=x'$ and $y=y'$. 
\end{proof}

\begin{theorem}[Continuous inverse maps]\label{Cor:Continuous inverse maps}
Let $q\in (0,1)$ and $t\in (0,1)$.
Suppose $(\mu_0,\mu_1)\in \mathcal{P}_\comp(M)\times\Prob_\comp(M)$  to be $q$-separated.
Then there exists a continuous map $W\colon Z_t(E) \longrightarrow E$ such that $z_t\circ W$ is the identity map on $z_t(E)$, and whenever $\mu_t\in\mathcal{P}(M)$ is a $t$-intermediate point of a $q$-geodesic from $\mu_0$ to $\mu_1$, $W_\sharp\mu_t$ belongs to $\Pi(\mu_0,\mu_1)$ and maximizes $\ell_q(\mu_0,\mu_1)$.
Here $E\subset\supp\mu_0 \times\supp\mu_1$ denotes the equality set from \eqref{Eq:Equality set S}, and $z_t$ is the  map given by Lemma \textnormal{\ref{Le:Measurable extension}}.
\end{theorem}

\begin{proof}
Recall that $E$ is compact and $l^q$-cyclically monotone by Theorem \ref{Th:Inf attained}.
Furthermore, $E\subset \{l>0\}$  by Definition \ref{Def:q-separated}.
Given any $z\in Z_t(E)$ we define $W(z) := (x,y)$, where $(x,y) \in E$ is chosen such that $z\in Z_t(x,y)$ ac\-cor\-ding to the definition of $Z_t(E)$ from Lemma \ref{Le:Zs}. 

We claim that $W$ is well-defined and continuous.
If $z\in Z_t(x,y) \cap Z_t(x',y')$ for  $(x,y),(x',y')\in E$, the $l^q$-cyclical monotonicity of $E$ yields \eqref{Eq:other inequality}, whence $x=x'$ and $y=y'$ by Proposition \ref{Pr:Lagrangian trajectories}.
Concerning continuity, let $(z_k)_{k\in \N}$ be a sequence in $Z_t(E)$ converging to $z\in Z_t(E)$, and consider a fixed  
non-relabeled subsequence of $(W(z_k))_{k\in\N}$.
Then the compactness of $E$ entails the existence of a limit $(x',y')\in E$ of a further non-relabeled subsequence of the latter. By the above well-definedness proof, $(x',y')$ necessarily coincide with the unique pair $(x,y)\in E$ with $z\in Z_t(x,y)$, i.e.~with $W(z)$.
Since the initial subsequence was arbitrary, the continuity of $W$ follows. 
The construction also ensures that $W$ is the right-inverse of $z_t$ on $z_t(E)$.

Finally, let $\mu_t$ be as hypothesized.
Since the pair $(\mu_0,\mu_1)$ is $q$-separated and $l$ is bounded from above on the compact set $\supp\mu_0\times\supp\mu_1$ by Lemma \ref{Le:l continuous}, we have $\ell_q(\mu_0,\mu_1)\in (0,+\infty)$; in particular $\ell_q(\mu_0,\mu_t) > 0$ and $\ell_q(\mu_t,\mu_1) >0$.
Hence, by applying Corollary \ref{Cor:Compact support}, Lemma \ref{Le:Zs}, and Lemma \ref{Le:Existence optimal couplings}, $\mu_0$ and $\mu_t$ as well as $\mu_t$ and $\mu_1$ admit $\smash{\ell_q}$-optimal couplings.
Proposition \ref{Pr:Reverse triangle} thus yields $\omega\in \mathcal{P}(M\times M\times M)$ with marginals $\mu_0$, $\mu_t$, and $\mu_1$ such that $\pi_{ij} := (\proj_{ij})_\sharp\omega$ is $\smash{\ell_q}$-optimal for $\mu_i$ and $\mu_j$ for every $i,j = 0,t,1$ with $i<j$.
By duality, cf.~Theorem \ref{Th:Inf attained} and Proposition \ref{Pr:Duality}, $\pi_{01}$ is concentrated on $E$.
By Proposition \ref{Pr:Reverse triangle}, $\mu_t = (\proj_t)_\sharp\omega$ is concentrated on $Z_t(E)$, which is compact by Lemma \ref{Le:Zs}.
Setting $\smash{W_i := \proj_i \circ W}$ for $i=0,1$, the above argument gives the concentration of $\omega$ on $W_0(Z_t(E))\times Z_t(E) \times W_1(Z_t(E))$.
This implies $\omega = (W_0, \Id, W_1)_\sharp\mu_t$ by \cite[Lemma 3.1]{ahmad2011}, and hence $\pi_{01}= W_\sharp \mu_t$. 
\end{proof}

\begin{lemma}[Variational characterization of geodesic endpoints]\label{Le:Variational char}
Let $q\in (0,1)$ and $\gamma\colon [0,1]\longrightarrow M$ be a 
 maximizing timelike geodesic.
Then for every $t\in (0,1)$ and $x\in M$, we have
\begin{equation}\label{Eq:identity lx1}
    l\big( x,\gamma(1)\big)^q \geq t^{1-q}\,l\big( x,\gamma(t)\big)^q + (1-t)^{1-q}\,l\big( \gamma(t),\gamma(1)\big)^q.
\end{equation}
Equality holds therein if and only if $x=\gamma(0)$.
\end{lemma}

\begin{proof}
The claimed bound \eqref{Eq:identity lx1} follows from \eqref{Eq:lq inequ with s} for $(x,\gamma(t),\gamma(1))$ and the concavity of the function $r \longmapsto r^q$.
Moreover, it is clear that $x=\gamma(0)$ achieves equality therein.
Now, assume that equality holds in \eqref{Eq:identity lx1} for some $x\in M$.
Then equality holds in \eqref{Eq:lq inequ with s} for $(x,\gamma(t),\gamma(1))$, which implies some timelike maximizing geodesic ending at $\gamma(1)$ to pass through $x$ and $\gamma(t)$.
As $\gamma(t)$ lies in the interior of $\gamma$, this geodesic necessarily coincides with $\gamma$,
and equality in \eqref{Eq:identity lx1} and the strict concavity of $r\longmapsto r^q$ on $[0,+\infty)$ force $x=\gamma(0)$. 
\end{proof}

\begin{proposition}[Star-shapedness of $q$-separation]\label{Pr:Star-shaped}
Let $q\in (0,1)$, $\mu_0,\mu_1$ $\in\mathcal{P}_\comp(M)$ and $(\mu_t)_{t\in[0,1]}$ be a $q$-geodesic from $\mu_0$ to $\mu_1$.
If $(\mu_0,\mu_1)$ is $q$-separated, then so is $(\mu_s,\mu_t)$ for every $s,t\in[0,1]$ with $s<t$.
\end{proposition}

\begin{proof}
It clearly suffices to prove the $q$-separation of $(\mu_0,\mu_t)$ and $(\mu_t,\mu_1)$ for every $t\in (0,1)$.
We concentrate on  $(\mu_0,\mu_t)$; $(\mu_t,\mu_1)$ is treated analogously by taking Remark \ref{Re:modified pis us vs} into account (or reduced to the reverse Lorentz--Finsler structure).

Assume the pair $(\mu_0,\mu_1)$ to be $q$-separated by $(\pi,u,v)$.
Theorem \ref{Th:Inf attained} implies the continuity of $u$ and $v$ on $\supp\mu_0$ and $\supp\mu_1$, respectively.
Denote by $E$ the equality set from \eqref{Eq:Equality set S} relative to $(\mu_0,\mu_1)$,
let $W\colon Z_t(E)\longrightarrow E$ be the continuous map provided by Theorem \ref{Cor:Continuous inverse maps}, and define $\smash{W_i := \proj_i\circ W}$ for $i=1,2$.
We claim the $q$-separation of $(\mu_0,\mu_t)$ by $(\pi_t, u_t,v_t)$, where
\begin{align}\label{Eq:pis us vs}
\begin{split}
    \pi_t &:= (W_1, \Id)_\sharp\mu_t,\\
    u_t &:= t^{q-1}\,u,\\
    v_t &:= t^{q-1}\, v \circ W_2 - t^{q-1}\,(1-t)^{1-q}\,q^{-1}\,l^q\circ (\Id, W_2).
    \end{split}
\end{align}

It follows from Theorem \ref{Cor:Continuous inverse maps} that $\pi_t \in \Pi(\mu_0,\mu_t)$.
By Proposition \ref{Pr:Reverse triangle} and the compactness of $Z_t(E)$ from Lemma \ref{Le:Zs}, we find $\supp\mu_t \subset Z_t(E)$.
Since $\smash{(\Id, W_2)\big(Z_t(E)\big)}\subset \{ l>0 \}$, Theorem \ref{Cor:Continuous inverse maps} and Lemma \ref{Le:l continuous} then imply the continuity of $u_t$ and $v_t$ on $\supp\mu_0$ and $\supp\mu_t$, respectively.
Now, given any $x\in\supp\mu_0$, $z \in \supp\mu_t$, and $y:=W_2(z)$, the hypothesized $q$-separation of $(\mu_0,\mu_1)$ entails
\begin{equation*}
    u(x) + v(y) \geq q^{-1}\,l(x,y)^q.
\end{equation*}
By Remark \ref{Re:Zs s-intermediate points} and Lemma \ref{Le:Variational char}, this yields
\begin{align*}
    t^{q-1}\,u(x) + t^{q-1}\,v(y)
    &\geq q^{-1}\,t^{q-1}\,l(x,y)^q\\
    &\geq q^{-1}\,l(x,z)^q + q^{-1}\,t^{q-1}\,(1-t)^{1-q}\,l(z,y)^q.
\end{align*}
Thus, we have
\begin{equation}\label{Eq:inequ us vs}
    u_t(x) + v_t(z) \geq q^{-1}\,l(x,z)^q.
\end{equation}
Moreover, we deduce from Lemma \ref{Le:Variational char} and $W(z) \in E$ that equality holds in \eqref{Eq:inequ us vs} if and only if $x = W_1(z)$.
This proves $\pi_t$ to be concentrated on
\begin{equation*}
    E_t := \big\lbrace (x,z) \in \supp \mu_0  \times \supp \mu_t  \mid u_t(x) + v_t(z) = q^{-1}\,l(x,z)^q\big\rbrace,
\end{equation*}
which is compact by the compactness of $\supp\mu_0 \times \supp\mu_t$ (cf.~Corollary \ref{Cor:Compact support}) 
and by Lemma \ref{Le:l continuous}.
Hence, we have $\supp\pi_t \subset E_t$.
It remains to prove $E_t \subset \{l>0\}$.
For any $(x,z) \in E_t$, recall that $x=W_1(z)$ necessarily holds, and $W(Z_t(E))\subset E\subset \{l>0\}$ yields $l(W_1(z),W_2(z)) > 0$.
Since $z\in Z_t(E)$, this implies $l(x,z) = t\,l(W_1(z),W_2(z))>0$. 
\end{proof}

\begin{remark}\label{Re:modified pis us vs} 
To prove the $q$-separation of $(\mu_t,\mu_1)$, the appropriate replacements of the quantities from \eqref{Eq:pis us vs} are
\begin{align*}
    \pi_t' &:= (\Id,W_2)_\sharp \mu_t,\\
    u_t' &:= (1-t)^{q-1}\,u\circ W_1 - (1-t)^{q-1}\,t^{1-q}\,q^{-1}\,l^q\circ(W_1, \Id),\\
    v_t' &:= (1-t)^{q-1}\,v.
\end{align*}
\end{remark}

Now we recall the definition of sub- and superdifferentiability of a function $u\colon M\longrightarrow [-\infty,+\infty]$ before Theorem \ref{Th:Lorentz distance}.
For $x\in M$ we write $x\in \Dom(\d u)$ if $u(x)\in\R$ and $u$ is differentiable at $x$.
The latter  is equivalent to $u$ being both sub- and superdifferentiable at $x$, in which case $\smash{\partial_\bullet u(x) = \partial^\bullet u(x) = \{\d u(x)\}}$. Finally, $\smash{\Dom(\d^2 u)}$ is the set of all $x\in \Dom(\d u)$ at which $u$ is  twice differentiable.

In the rest of this section, let $\meas$ be an arbitrary measure on $M$ comparable with the Lebesgue measure in each local coordinate system. It is only used to distinguish $\meas$-negligible sets in view of  differentiability properties: for example, by Alexandrov's theorem, semi-convex or semi-concave functions are twice differentiable $\meas$-a.e.\ (see \cite{bangert,ohta-uni} for the Riemannian and Finsler cases, as well as \cite[Chapter 14]{villani2009} for more details).

In the sequel, we adopt the following conventions without further notice.
Given any $X\subset M$, every semi-convex Lipschitz function $u\colon X \longrightarrow \R$ (may and) will be treated as the restriction of an appropriate semi-convex Lipschitz function defined on a neighborhood of $X$, cf.~Theorem \ref{Th:Inf attained}.
Moreover, any map defined on $\Dom(\d u)$ will be extended to $M$ by an arbitrarily chosen constant.
Our results will not depend on the respective choices of extensions.

We also make use of the reverse Lorentz--Finsler strucure $\overline{L}(v):=L(-v)$ (briefly mentioned in the proof of Theorem \ref{Th:Lorentz distance}), equipped with the time orientation that $v \in TM$ is future-directed timelike for $\overline{L}$ if and only if $-v$ is future-directed timelike for $L$.
Note that, given a timelike geodesic $\gamma\colon[0,1]\longrightarrow M$ for $L$, its reverse curve defined by $\overline{\gamma}(t):=\gamma(-t)$ is a timelike geodesic for $\overline{L}$.

\begin{lemma}[Transport maps and Jacobi fields]\label{Le:Maps Jacobian}
Let $q\in (0,1)$, $X,Y \subset M$ be compact, $U\subset M$ be a given relatively compact neighborhood of $X$,
and $u \colon U \longrightarrow \R$ be semi-convex and Lipschitz continuous with  
 \begin{equation}\label{Eq:u ulq assumption}
     u(x) \geq {}^{(l^q)}(u^{(l^q)})(x)
 \end{equation}
  for every $x\in U$ according to \eqref{Eq:Transform} defined in terms of $X$ and $Y$.
  For $t \in [0,1]$, define $\FF_t \colon M \longrightarrow M$ and $\GG_t \colon M \longrightarrow M$ by
  \begin{align}\label{Eq:Fs definition}
  \begin{split}
      \FF_t(x) &:= \exp_x\!\big(t\,F^*(\d u)^{p-2}\,\mathscr{L}^*(\d u)\big),\\
      \GG_t(y) &:= \overline{\exp}_y\!\big(\!-\!t\,F^*(-\d u^{(l^q)})^{p-2}\,\mathscr{L}^*(-\d u^{(l^q)})\big),
      \end{split}
  \end{align}
  provided $x\in \Dom(\d u)$ and $y\in \Dom(\d u^{(l^q)})$, respectively, where $\smash{\overline{\exp}_y}$ denotes the exponential map for the reverse structure $\overline{L}$.
 \begin{enumerate}[label=\textnormal{(\roman*)}]
     \item For every $(x,y)\in U\times Y$, we have
     \begin{equation}\label{eq:u+v}
         u(x) + u^{(l^q)}(y) \geq q^{-1}\,l(x,y)^q.
     \end{equation}
     Also, assuming that equality holds at $(x,y)\in (X\times Y)\cap \{l>0\}$,
     we have $y= \FF_1(x)$ if $x\in\Dom(\d u)$, and $(x,y)\notin \sing(l)$ if $\smash{x\in \Dom(\d^2 u)}$.
     Analogously, $\smash{y\in \Dom(\d  u^{(l^q)})}$ implies $x = \GG_1(y)$, and $\smash{y\in \Dom(\d^2 u^{(l^q)})}$ gives $(x,y)\notin \sing(l)$.
     \item For $\meas$-a.e.~$x\in X \cap \{\d u \neq 0\}$, the 
     derivative $\d\FF_t(x) \colon T_xM \longrightarrow T_{\FF_t(x)}M$ 
     exists, depends smoothly on $t\in[0,1]$, and $t \longmapsto \d\FF_t(x)(w)$ is a Jacobi field along the geodesic $\gamma_x(t):=\FF_t(x)$ for every $w \in T_xM$.
     \item 
     For $\meas$-a.e.~$x \in X \cap \{\d u \neq 0\}$ and every $w \in T_xM$, we have
\[
D^{\dot{\gamma}_x}_{\dot{\gamma}_x} \big[ \d\FF_t(x)(w) \big]_{0}
= F^*\big( \d u(x) \big)^{p-2} \, \nabla^2 u(w) +\d[F^*(\d u)^{p-2}](w) \,\nabla u(x),
\]
where $\nabla u:=\mathscr{L}^*(\d u)$ is the \emph{gradient vector} and $\nabla^2 u \colon T_xM \longrightarrow T_xM$ is the \emph{Hessian} defined by
\begin{equation}\label{Eq:HESSIAN}
\nabla^2 u(w) :=D_w^{\nabla u}(\nabla u).
\end{equation}
 \end{enumerate}
\end{lemma}

\begin{proof}
(i)
The first claim \eqref{eq:u+v} is a consequence of the definition \eqref{Eq:Transform} and \eqref{Eq:u ulq assumption}.
We turn to the equality case.
When $x\in \Dom(\d u)$, the first claim and our assumption on $x$ and $y$ entail that, for every sufficiently small $w \in T_xM$, 
\begin{align*}
    q^{-1}\,l\big(\exp_x^h(w), y \big)^q &\leq u\big(\exp_x^h(w) \big) + u^{(l^q)}(y)\\
    &= u(x) + \d u(w) + u^{(l^q)}(y) + o(\vert w \vert_h)\\
    &= q^{-1}\,l(x,y)^q + \d u(w) + o(\vert w \vert_h).
    \end{align*}
In other words, $\smash{q^{-1}\,l(\cdot,y)^q}$ is super\-dif\-feren\-tiable at $x$ with a supergradient $\d u\in T_x^*M$.
Hence, Corollary \ref{Cor:Twist}(ii) implies $y=\FF_1(x)$ as claimed.

Next, let $\smash{x\in \Dom(\d^2 u)}$.
Arguing similarly to the previous step and employing Taylor's theorem around $x$ we obtain
\begin{align*}
    q^{-1}\,l\big(\exp_x^h(w),y \big)^q &\leq u\big(\exp_x^h(w) \big) + u^{(l^q)}(y)\\
    &= u(x) + \d u(w) + \d^2 u(w \otimes w) + u^{(l^q)}(y) + o(\vert w\vert_h^2)\\
    &= q^{-1}\,l(x,y)^q + \d u(w) + \d^2 u(w \otimes w) + o(\vert w\vert_h^2)
\end{align*}
for every sufficiently small $w \in T_xM$.
By replacing $w$ with $-w$ and adding up the resulting inequalities, the terms involving $\d u$ cancel out.
Thus, we get an upper bound for the numerator in Corollary \ref{Cor:Twist}(iv), and we find $(x,y) \notin \sing(l)$.

The analogous statements replacing the role of $x$ and $y$ can be seen by the same argument for $\overline{L}$ (by noticing $\overline{F}{}^*(\zeta)=F^*(-\zeta)$ and $\overline{\mathscr{L}}{}^*(\zeta)=-\mathscr{L}^*(-\zeta)$ for $-\zeta \in \Omega^*_y$, where $\overline{F}{}^*$ and $\overline{\mathscr{L}}{}^*$ are associated with $\overline{L}$).


(ii)
On the one hand, since $u$ is semi-convex, it is twice differentiable $\meas$-a.e.
In particular, for $\meas$-a.e.\ $x \in X$, $\d u(x)$ exists and, when $\d u(x) \neq 0$, the Hessian $\nabla^2 u$ as in \eqref{Eq:HESSIAN} is well-defined.
On the other hand, by definition, $\gamma_{z}(t)=\FF_t(z)$ is a geodesic for every $z \in \Dom(\d u)$.
Hence, $\d \FF_t(x)$ exists and $t \longmapsto \d \FF_t(x)(w)$ is a Jacobi field (and hence smooth).

(iii)
Observe that
\[
\dot{\gamma}_x(0) =F^*(\d u)^{p-2} \mathscr{L}^*(\d u) =F^*(\d u)^{p-2} \, \nabla u(x).
\]
This implies, for every $w \in T_xM$,
\begin{align*}
D^{\dot{\gamma}_x}_{\dot{\gamma}_x} \big[ \d\FF_t(x)(w) \big]_{0}
&= D^{\dot{\gamma}_x}_w \big[ F^*(\d u)^{p-2} \, \nabla u \big] \\
&= F^*(\d u)^{p-2} \, \nabla^2 u(w) +\d[F^*(\d u)^{p-2}](w) \, \nabla u(x).
\end{align*}
This terminates the proof.
\end{proof}

We remark that, when equality holds in \eqref{eq:u+v} as in (i), then $\smash{\d u(x) \in \Omega^*_x}$.
Therefore, $F^*(\d u(x))$ is well-defined and $\nabla u(x) =\smash{\mathscr{L}^*(\d u(x)) \in \Omega_x}$.
Analogously, we find $\smash{-\d u^{(l^q)}(y) \in \Omega^*_y}$ and $\smash{\mathscr{L}^*(-\d u^{(l^q)}(y)) \in \Omega_y}$.
Thus, the vector $\smash{-\mathscr{L}^*(-\d u^{(l^q)}(y))}$ is future-directed with respect to the reverse structure $\smash{\overline{L}}$, and the map $t\longmapsto \mathcal{G}_t(y)$ constitutes  a future-directed timelike geodesic for $\smash{\overline{L}}$.

\begin{theorem}[Characterizing optimal maps]\label{Th:Char optimal maps}
Let $q\in (0,1)$ and assume $(\mu,\nu)\in \mathcal{P}_\comp(M)\times\Prob_\comp(M)$ to be $q$-separated by $(\pi,u,v)$. Then the following hold.
\begin{enumerate}[label=\textnormal{(\roman*)}]
    \item Suppose $\smash{\mu \in \mathcal{P}^\ac(M,\meas)}$ and
    define $\FF_1 \colon M\longrightarrow M$ as in \eqref{Eq:Fs definition}. 
    Then we have $\pi = (\Id,\FF_1)_\sharp\mu$, which is a unique maximizer of $\ell_q(\mu,\nu)$, and $(x,\FF_1(x))\notin \sing(l)$ holds for $\mu$-a.e.~$x\in M$.
    \item The map $\FF_1$ in \textnormal{(i)} is unique in the following sense.
    Assume that $\tilde{u}\colon U\longrightarrow \R$ is a Lipschitz function defined on a neighborhood $U$ of $\supp\mu$ with
    \begin{equation}\label{Eq:utilde max}
        \tilde{u}(x) = \sup_{y\in \supp\nu}\big(q^{-1}\,l(x,y)^q - \tilde{u}^{(l^q)}(y)\big),
    \end{equation}
    and suppose that the map $\smash{\tilde{\FF}_1\colon M \longrightarrow M}$ given by
    \begin{equation*}
        \tilde{\FF}_1(x) = \exp_x\!\big(F^*(\d\tilde{u})^{p-2}\,\mathscr{L}^*(\d\tilde{u})\big)
    \end{equation*}
    provided $x\in \Dom(\d\tilde{u})$, obeys $\nu = (\Tilde{\FF}_1)_\sharp\mu$.
    Then $\smash{\d\Tilde{u} = \d u}$ $\mu$-a.e.
    \item In addition to the hypothesis of \textnormal{(i)}, assume $\nu\in \Prob^\ac(M,\meas)$ and define $\GG_1\colon M \longrightarrow M$ as in \eqref{Eq:Fs definition}.
    Then we have $\FF_1\circ \GG_1(y) = y$ for $\nu$-a.e.~$y\in M$ as well as $\GG_1\circ \FF_1(x) = x$ for $\mu$-a.e.~$x\in M$.
\end{enumerate}
\end{theorem}

\begin{proof}
(i)
Theorem \ref{Th:Inf attained} yields that $\smash{u= {}^{(l^q)}v}$ on $\supp\mu$, $\smash{v = u^{(l^q)}}$ on $\supp\nu$, and that $u$ and $v$ extend to semi-convex Lipschitz functions on neighborhoods of $\supp\mu$ and $\supp\nu$, respectively.
Rademacher's theorem ensures that $u$ is differentiable $\meas$-a.e.~and thus $\mu$-a.e.
In particular, the measure $(\Id, \FF_1)_\sharp\mu$ makes sense.

Again by Theorem \ref{Th:Inf attained}, $\pi$ and $(u,v)$ attain the respective supremum and infimum in \eqref{Eq:Kantorovich}.
Let $E$ be the equality set from \eqref{Eq:Equality set S}.
By Lemma \ref{Le:Maps Jacobian}(i), for $\pi$-a.e.~$(x,y) \in E$ with $\smash{x\in \Dom(\d^2 u)}$ we have $y=\FF_1(x)$ and $(x,\FF_1(x)) \notin \sing(l)$.
Since $M \setminus \smash{\Dom(\d^2 u)}$ is $\meas$-negligible by Alexandrov's theorem and hence $\mu$-negligible, we obtain $\smash{\pi = (\Id, \FF_1)_\sharp\mu}$, therefore, $\smash{(\Id, \FF_1)_\sharp \mu}$ is a maximizer of $\smash{\ell_q(\mu,\nu)}$.
Given any such maximizer $\pi'$, Proposition \ref{Pr:Duality} gives  $\pi'[E]=1$, and arguing as above shows $\pi' = \pi$ and hence the uniqueness of $\pi$.

(ii)
First, we claim that $\smash{\Tilde{\pi} := (\Id, \Tilde{\FF}_1)_\sharp\mu}$ maximizes $\smash{\ell_q(\mu,\nu)}$.
By Lemma \ref{Le:Maps Jacobian}(i), for every $x\in\Dom(\d\Tilde{u})$ the supremum of \eqref{Eq:utilde max} is necessarily attained at $y = \tilde{\FF}_1(x)$.
As in (i), we see that $M\setminus \Dom(\d\Tilde{u})$ is $\mu$-negligible. Therefore,
\begin{equation*}
    q^{-1}\,l\big( x,\Tilde{\FF}_1(x) \big) = \Tilde{u}(x) + \Tilde{u}^{(l^q)}\big( \Tilde{\FF}_1(x) \big)
\end{equation*}
for $\mu$-a.e.~$x\in M$. By integrating this inequality against $\mu$ and using $\nu = (\Tilde{\FF}_1)_\sharp\mu$ as well as Proposition \ref{Pr:Duality}, we deduce the claim.

The uniqueness part of (i) thus gives $(\Id, \FF_1)_\sharp\mu = (\Id, \Tilde{\FF}_1)_\sharp\mu$, whence $\FF_1 = \Tilde{\FF}_1$ $\mu$-a.e.
We claim that in fact $\d u = \d\Tilde{u}$ on $\Dom(\d u) \cap \Dom(\d\Tilde{u})$ (a set which has $\mu$-full measure, as seen above).
If that fails  at all points $x\in \Dom(\d u)\cap\Dom(\d\Tilde{u})$ with $\smash{\FF_1(x) = \Tilde{\FF}_1(x)}$ in a set of positive $\mu$-measure, there are two different maximizing geodesics connecting $x$ and $\FF_1(x)$.
Hence $(x,\FF_1(x))\in\sing(l)$, contradicting (i).

(iii) An analogous argument as in (i), taking into account that $v = \smash{u^{(l^q)}}$ on $\supp\nu$, gives $\pi = (\GG_1,\Id)_\sharp\nu$.
Therefore, $\pi$ is concentrated on $\Dom(\d u) \times \Dom(\d v)$.
Hence $y=\FF_1(x)$ and $x=\GG_1(y)$ for $\pi$-a.e.~$(x,y)\in M\times M$ by Lemma \ref{Le:Maps Jacobian}(i), which is the claim.
\end{proof}

\begin{corollary}[Lagrangian characterization of $q$-geodesics]\label{Cor:Char geos}
Given any $q\in (0,1)$, assume $(\mu_0,\mu_1)\in\Prob_\comp^\ac(M,\meas)\times\Prob_\comp(M)$ to be $q$-separated by $(\pi,u,v)$.
Define $\FF_t \colon M \longrightarrow M$ as in \eqref{Eq:Fs definition}, where $t\in[0,1]$.
Then $((\FF_t)_\sharp\mu_0)_{t\in[0,1]}$ is a unique $q$-geodesic from $\mu_0$ to $\mu_1$.
\end{corollary}

\begin{proof}
By Theorem \ref{Th:Char optimal maps}, $\pi$ is the unique $\smash{\ell_q}$-optimal coupling of $\mu_0$ and $\mu_1$, and has the form  $\smash{\pi=(\Id,\FF_1)_\sharp\mu_0}$. 
In particular, since 
\begin{equation*}
    l\big( \FF_s(x),\FF_t(x) \big)^q = (t-s)^q\,l\big( x,\FF_1(x) \big)^q
\end{equation*}
for $\mu_0$-a.e.~$x\in M$ and every $s,t\in[0,1]$ with $s<t$, $((\FF_t)_\sharp\mu_0)_{t\in[0,1]}$ is a $q$-geodesic from $\mu_0$ to $\mu_1$; Corollary \ref{Cor:Compact support} thus implies the compactness of $\supp[ (\FF_t)_\sharp\mu_0]$.

To prove the uniqueness, let $(\mu_t)_{t\in[0,1]}$ be a $q$-geodesic from $\mu_0$ to $\mu_1$, and fix $t\in (0,1)$.
Again $\mu_t$ is compactly supported  by Corollary \ref{Cor:Compact support}.
In particular, by Lemma \ref{Le:Existence optimal couplings} $\mu_0$ and $\mu_t$ as well as $\mu_t$ and $\mu_1$ admit $\smash{\ell_q}$-optimal couplings.
Then we fix $\omega\in \Prob(M\times M\times M)$ as given by Proposition \ref{Pr:Reverse triangle}.
The $\smash{\ell_q}$-optimality of its marginals and uniqueness of $\pi$ ensures that $\pi_{13} = \pi$, where $\pi_{13}:=(\proj_{13})_\sharp\omega$.
Hence $\pi_{13}[\sing(l)]=0$ by Theorem \ref{Th:Char optimal maps}(i).
Combining Lemma \ref{Le:zs} with the last part of Proposition \ref{Pr:Reverse triangle} and  $z_t(x,\FF_1(x)) = \FF_t(x)$ for every $x\in \Dom(\d u)$, cf.~Remark \ref{Re:Zs s-intermediate points}, finally implies
\begin{equation*}
\mu_t = (z_t)_\sharp\pi_{13} = (z_t)_\sharp\big((\Id, \FF_1)_\sharp\mu_0\big) = (\FF_t)_\sharp\mu_0
\end{equation*}
and hence the desired conclusion.
\end{proof}

In fact, under the hypotheses of Corollary \ref{Cor:Char geos} the interpolants $(\FF_t)_\sharp\mu_0$ are absolutely continuous with respect to $\meas$ provided $t<1$.
We defer the proof of the following corresponding proposition to the next section since we need the measure contraction property stated in Lemma \ref{Le:mcp}. 

\begin{proposition}[Absolute continuity of interpolants]\label{Pr:Absolute continuity}
Given any $q\in (0,1)$, let $(\mu_t)_{t\in[0,1]}$ denote a unique $q$-geodesic joining a pair of $q$-separated endpoints $(\mu_0,\mu_1)\in \Prob_\comp^\ac(M,\meas) \times \Prob_\comp(M)$. Then $\mu_t\in \Prob^\ac_\comp(M,\meas)$ for all $t\in (0,1)$.
\end{proposition}

A direct consequence is the following \emph{Monge--Amp\`ere type equation} (also called the \emph{Jacobian equation}; cf.\ \cite[Theorem~11.1]{villani2009}).
It is shown in exactly the same way as \cite[Corollary 5.11]{Mc} (whose main ingredient \cite[Theorem 5.10]{Mc} is a local statement that can be established chart-wise in our setting) and thus the proof is omitted. 

\begin{corollary}[Monge--Ampère type equation]\label{Cor:Monge Ampère}
Retain the assumptions and the notation of Proposition \textnormal{\ref{Pr:Absolute continuity}}, and set $\smash{\mu_t :=(\FF_t)_{\sharp}\mu_0=\rho_t\, \meas}$ for $t \in [0,1)$.
Then, for every $t \in (0,1)$ and $\mu_0$-a.e.\ $x \in M$, $\d\FF_t(x)$ exists and we have
\[
    \rho_0(x) = \rho_t \big( \FF_t(x) \big)\, \det_{\meas}[\d\FF_t(x)].
\]
This conclusion also holds for $t=1$ provided $\smash{\mu_1\in \Prob_\comp^\ac(M,\meas)}$.
\end{corollary}

Here, $\det_{\meas}[\d\FF_t(x)]$ is the Jacobian (determinant) of  $\d\FF_t(x)$ with respect to $\meas$, namely
\[
\det_{\meas}[\d\FF_t(x)] :=\frac{\meas_{\FF_t(x)}[\d\FF_t(x)(A)]}{\meas_x[A]},
\]
where $\meas_x$ is the Lebesgue measure on $T_xM$ induced from $\meas$ via coordinates and $A \subset T_xM$ is an arbitrary nonempty, bounded open set.
We remark that $\det_{\meas}[\d\FF_0(x)]=1$ and $\det_{\meas}[\d\FF_t(x)]>0$ for $\mu_0$-a.e.\ $x\in M$.

Finally, we remark that the statements from Theorem \ref{Th:Char optimal maps}, Corollary \ref{Cor:Char geos}, and Proposition \ref{Pr:Absolute continuity} can be extended beyond $q$-separated measures.
Compare this with Remark \ref{Re:About q-sep} below. 

\begin{corollary}[Optimal maps and $q$-geodesics without duality]\label{Cor:Optimal maps without duality} Let $q\in (0,1)$, and let $\smash{\mu_0 \in\Prob_\comp^\ac(M,\meas)}$ as well as $\mu_1\in\Prob_\comp(M)$. 
\begin{enumerate}[label=\textnormal{(\roman*)}]
    \item There exists at most one \emph{chronological} $\smash{\ell_q}$-optimal coupling $\smash{\pi\in\Pi(\mu_0,\mu_1)}$.
    \end{enumerate}
Moreover, if $\pi$ is such a coupling, then we have the following.
\begin{enumerate}[label=\textnormal{(\roman*)}]\setcounter{enumi}{1}
    \item It can be written as $\pi = (\Id, \FF_1)_\sharp\mu_0$, where $\FF_1\colon M\longrightarrow M$ is a certain Borel map, and satisfies $\pi[\sing(l)]=0$.
    \item It decomposes as $\pi = \sum_{i\in\N}\pi_i$, where $(\pi_i)_{i\in\N}$ is a sequence of mutually singular sub-probability measures on $M\times M$ with the following properties for every $i\in\N$\textnormal{:} The set $\smash{\supp\mu^i \times\supp\nu^i}$, where $\smash{\mu^i}$ and $\smash{\nu^i}$ are the marginals of $\smash{\hat\pi_i := \pi_i[M\times M]^{-1}\,\pi_i}$, is compact and disjoint from $\{l\geq 0\}$. In particular, the pair $(\mu^i,\nu^i)$ is $q$-separated. $\FF_1$ agrees $\smash{\mu^i}$-a.e.~with the transport map $\smash{\FF_1^i}$ pushing forward $\mu^i$ to $\nu^i$ from Theorem \textnormal{\ref{Th:Char optimal maps}}. Lastly, the above sum is finite if and only if $\supp\pi$ is compact and disjoint from $\{l\leq 0\}$.
    \item The assignment $\mu_t := (z_t)_\sharp\pi$, where $z_t$ designates the map from Lemma \textnormal{\ref{Le:Measurable extension}},  defines a $q$-geodesic $(\mu_t)_{t\in[0,1]}$ from $\mu_0$ to $\mu_1$ such that $\mu_t\in \smash{\Prob_\comp^\ac(M,\meas)}$ provided $t<1$. Subject to the decomposition of $\pi$ from \textnormal{(ii)},  $((z_t)_\sharp\pi_i)_{i\in[0,1]}$ are mutually singular sub-probability measures for $t<1$.
\end{enumerate}
\end{corollary}

\begin{proof} This is proven in the same way as \cite[Theorem 7.1]{Mc}. A decomposition of $\pi$ as in (ii) is obtained by dividing $\{l>0\} \cap \supp\pi$ into countably many open and relatively compact rectangles whose closures are contained in the open set $\{l>0\}$, and restriction of $\pi$ to these. We skip the details.
\end{proof}

In view of the last claim of (iii), we stress that the condition $\pi[\{l>0\}]=1$ does not imply the disjointness of $\supp\pi$ from $\{l\leq 0\}$. The latter is thus a strictly stronger property.

\section[Curvature-dimension condition from Ricci  curvature bounds]{Curvature-dimension condition from Ricci\\ curvature bounds}\label{sc:Ric-TCD}

In the previous section we saw that optimal transports are done along geodesics, and hence Jacobi fields play a key role to control their behavior.
Then, it is natural to expect that the Ricci curvature comes into play, and we show in this section that a lower bound of the weighted Ricci curvature provides a convexity estimate for a certain entropy functional, called the curvature-dimension condition.
We refer to \cite{cordero2001,lott2009,renesse2005,sturmI,sturmII} for the Riemannian case and \cite{ohta2009} for the Finsler case.

\subsection{Weighted Ricci curvature}\label{Sec:weight}

Here and henceforth, we fix a smooth positive measure $\meas$ on a Finsler spacetime $(M,L)$.
Then the associated \emph{weight function} $\smash{\psi_{\meas} \colon \overline{\Omega} \setminus \{0\} \longrightarrow \R}$ is defined by
\begin{equation*}
\d\meas =\e^{-\psi_{\meas}(\dot{\gamma}(t))} \sqrt{-\det\big[ g_{\alpha\beta}\big(\dot{\gamma}(t) \big) \big]}
 \,\d x^1\, \d x^2 \cdots \d x^n
\end{equation*}
along timelike geodesics $\gamma$.
We remark that $\psi_{\meas}$ is positively $0$-homogeneous in the sense that $\psi_{\meas}(cv)=\psi_{\meas}(v)$ for every $\smash{v \in \overline{\Omega}} \setminus \{0\}$ and $c>0$.

\begin{remark}\label{rm:psi}
In \cite{LMO1,LMO2,LMO3}, we deal with a more general framework by directly considering a positively $0$-homogeneous weight function $\smash{\psi \colon \overline{\Omega} \setminus \{0\} \longrightarrow \R}$,
which is not necessarily associated with a measure as above.
In this paper, however, we begin with a measure $\meas$ since we need it as a reference measure to define entropy functionals.
In addition, we will not use the \emph{$\epsilon$-range} introduced in \cite{LMO1} (in other words, we put $\epsilon=1$).
\end{remark}

For later convenience, along a timelike geodesic $\gamma \colon [0,l) \longrightarrow M$, we denote
\begin{equation}\label{eq:psi_g}
\psi_{\gamma}(t) :=\psi_{\meas}\big( \dot{\gamma}(t) \big).
\end{equation}

\begin{definition}[Weighted Ricci curvature]\label{df:LF-wRic}
Given a causal vector $v \in \overline{\Omega} \setminus \{0\}$,
let $\gamma \colon (-\varepsilon,\varepsilon) \longrightarrow M$ be the causal geodesic with $\dot{\gamma}(0)=v$.
Then, for $N \in \R \setminus \{n\}$, we define the \emph{weighted Ricci curvature} at $v$ by
\begin{equation}\label{eq:LF-wRic}
\Ric_N(v) :=\Ric(v) +\psi''_{\gamma}(0) -\frac{\psi'_{\gamma}(0)^2}{N-n},
\end{equation}
where $\Ric(v)$ is the usual Ricci curvature defined in Definition~\ref{df:curv}.
We also set
\[
\Ric_{\infty}(v) :=\Ric(v) +\psi''_{\gamma}(0),
 \qquad \Ric_n(v) :=\lim_{N \downarrow n} \Ric_N(v),
\]
and $\Ric_N(0):=0$.
\end{definition}

\begin{remark}\label{rm:Ric_N}
\begin{enumerate}[label=\textnormal{(\alph*)}]
\item
By definition, the quantity $\Ric_N(v)$ is non-decreasing in $N$ in the ranges $[n,+\infty]$ and $(-\infty,n)$.
Moreover, for $N \in (n,+\infty)$ and $N' \in (-\infty,0)$ we have
\[
\Ric_n(v) \le \Ric_N(v) \le \Ric_{\infty}(v) \le \Ric_{N'}(v) \le \Ric_0(v).
\]

\item
In our previous papers \cite{LMO1,LMO2,LMO3}, we chose the notation $\dim M=n+1$ but adopted the same definition \eqref{eq:LF-wRic}.
As a result, $\Ric_{N+1}$ in this paper (and \cite{Mc}) corresponds to $\Ric_N$ in \cite{LMO1,LMO2,LMO3}.
\end{enumerate}
\end{remark}

One can obtain several singularity theorems \cite{LMO1} and comparison theorems \cite{LMO2} under the curvature bound $\Ric_N \ge K$ for $K \in \R$, by which we mean $\smash{\Ric_N(v) \ge KF^2(v) =-2KL(v)}$ for all $v \in \Omega$ (and hence for all $\smash{v \in \overline{\Omega}}$ by continuity). 
For giving a deferred proof of Proposition~\ref{Pr:Absolute continuity}, here we explain the \emph{timelike measure contraction property} as a corollary to (the proof of) the Bishop--Gromov volume comparison theorem in \cite[Theorem~5.10]{LMO2} (see \cite{ohta-mcp,sturmII} for the Riemannian measure contraction property).
To this end, we define
\[
\mathfrak{s}_{\kappa}(r)
:=\begin{cases} \dfrac{1}{\sqrt{\kappa}}\sin(\sqrt{\kappa}r) & \kappa>0, \\
r & \kappa=0, \\
\dfrac{1}{\sqrt{-\kappa}}\sinh(\sqrt{-\kappa}r) & \kappa<0 \end{cases}
\]
for $r \in \R$ when $\kappa \leq 0$, and for $r \in [0,\pi/\sqrt{\kappa}]$ when $\kappa > 0$, as well as
\begin{equation}\label{eq:tau}
\tau_{K,N}^{(t)}(r) :=t^{1/N} \bigg( \frac{\mathfrak{s}_{K/(N-1)}(tr)}{\mathfrak{s}_{K/(N-1)}(r)} \bigg)^{(N-1)/N}
\end{equation}
for $K \in \R$, $N \in (-\infty,0] \cup [n,+\infty)$, and $t \in (0,1)$, which is defined for $r>0$ when $K/(N-1) \le 0$, and for $\smash{r\in (0,\pi\sqrt{(N-1)/K})}$ when $K/(N-1)>0$. We also define $\smash{\tau_{K,N}^{(t)}(0):=t}$.

\begin{lemma}[Timelike measure contraction property]\label{Le:mcp}
Let $(M,L,\meas)$ satisfy $\Ric_N \ge K$ for some $K \in \R$ and $N \in [n,+\infty)$.
Then we have
\[ \meas[Z_t(x,B)] \ge \inf_{y \in B} \tau_{K,N}^{(t)}\big( l(x,y) \big)^N \, \meas[B] \]
for every $x \in M$, every compact set $B \subset I^+(x)$, and every $t \in (0,1)$. Moreover,
\[ \meas[Z_t(A,y)] \ge \inf_{x \in A} \tau_{K,N}^{(1-t)}\big( l(x,y) \big)^N \, \meas[A] \]
for every $y \in M$, every compact set $A \subset I^-(y)$, and every $t \in (0,1)$.
\end{lemma}

\begin{proof}
First of all, if $K>0$, then the Lorentzian counterpart to the Bonnet--Myers theorem \cite[Theorem~5.17]{LMO1} implies $\smash{l(x,y) \le \pi\sqrt{(N-1)/K}}$, and hence $l(x,y)<\smash{\pi\sqrt{(N-1)/K}}$ for $\meas$-a.e.\ $y \in I^+(x)$.

We first consider the set $Z_t(x,B)$.
Note that $x$ and $\meas$-a.e.\ $\smash{y \in I^+(x)}$ are connected by a unique maximizing timelike geodesic $\gamma_y \colon [0,1] \longrightarrow M$, and
\[
\meas[Z_t(x,B)] =\int_B \det_{\meas}[\d\Psi_t] \,\d\meas,
\]
where $\det_{\meas}[\d\Psi_t]$ is the Jacobian of the map $\Psi_t(y):=\gamma_y(t)$ with respect to the measure $\meas$.
Precisely, setting $\smash{w_y:=\dot{\gamma}_y(0)} \in T_xM$ and using an orthonormal basis $(e_\alpha)_{\alpha=1}^n$ of $(T_xM,g_v)$ with $e_1=w_y/l(x,y)$ as well as the vector fields $\smash{J_i(t):=\d[\exp_{tw_y}](te_i)}$ along $\gamma_y$, where $i=2,\ldots,n$, we have
\[
\det_{\meas}[\d\Psi_t(y)] =\frac{t\,\phi_y(t)}{\phi_y(1)},\qquad
\phi_y(t):=\e^{-\psi_{\gamma_y}(t)} \sqrt{ \det \big[ g_{\dot{\gamma}_y(t)}\big( J_i(t),J_j(t) \big) \big]}.
\]
The function $s \longmapsto \phi_y(s/l(x,y))^{1/(N-1)}$ coincides with $h=h_1$ in the proof of \cite[Theorem 5.10]{LMO2}  with $\epsilon=1$ and $c=1/(N-1)$, and it follows from the hypothesis $\Ric_N \ge K$ that $h/\mathfrak{s}_{K/(N-1)}$ is non-increasing.
Hence,
\begin{align*}
\meas[Z_t(x,B)]
&=\int_B \frac{t\,\phi_y(t)}{\phi_y(1)} \,\d\meas(y) \\
&\ge \int_B t\, \frac{\mathfrak{s}_{K/(N-1)}(t\,l(x,y))^{N-1}}{\mathfrak{s}_{K/(N-1)}(l(x,y))^{N-1}} \,\d\meas(y) \\
&\ge \inf_{y \in B} \tau^{(t)}_{K,N}\big( l(x,y) \big)^N \, \meas[B].
\end{align*}

The assertion on $Z_t(A,y)$ can be reduced to the reverse structure $\overline{L}$, for which $\Ric_N \ge K$ (with respect to $\meas$) also holds since
\[ \Ric_N^{\overline{L}}(v) =\Ric_N^L(-v) \ge -2KL(-v) =-2K\overline{L}(v) \]
for every tangent vector $v$ which is causal with  respect  to $\smash{\overline{L}}$.
\end{proof}

\subsection{Absolute continuity of intermediate measures}

We are ready to provide a proof of Proposition \ref{Pr:Absolute continuity}, which we restate below for convenience. The Lorentzian version of it has been shown in \cite[Theorem~A.1]{Mc} by establishing a modified Monge--Mather shortening estimate, see also \cite{buffoni, villani2009}. 
We propose an alternative proof based on ideas of Figalli and Juillet \cite{figalli2008}; a byproduct of the proof is a qualitative density bound we do not state explicitly, but which should be compared to \cite{braun2023}.

\begin{proposition} Given any $q\in (0,1)$, let $(\mu_t)_{t\in[0,1]}$ denote a unique $q$-geodesic joining a pair of $q$-separated endpoints $(\mu_0,\mu_1)\in \Prob_\comp^\ac(M,\meas) \times \Prob_\comp(M)$ according to Corollary \textnormal{\ref{Cor:Char geos}}. Then $\mu_t\in \Prob^\ac_\comp(M,\meas)$ for every $t\in (0,1)$.
\end{proposition}

\begin{proof} 
We first check the case $\supp\mu_0\times\supp\mu_1\subset \{l>0\}$, i.e.~$l$ is bounded and bounded away from zero on the product $\supp\mu_0\times\supp\mu_1$. Let $K\in\R$ and $N\in (n,+\infty)$ be fixed numbers such that $\Ric_N\geq K$ in all timelike directions on the compact causal diamond $\smash{J^+(\mu_0)\cup J^-(\mu_1)}$.
Then we claim that 
\begin{equation}\label{Eq:Claim}
    \int_{\{\rho_0 >0\}} f\circ \FF_t\,\d\meas \leq 
    \sup_{x\in\supp\mu_0, y\in\supp\mu_1} \tau_{K,N}^{(1-t)}\big(l(x,y)\big)^{-N}
    \int_M f\,\d\meas
\end{equation}
for every nonnegative, continuous function $f$ on $M$ with compact support. 

To this aim, we fix a sequence $(\mu_1^k)_{k\in\N}$ of compactly supported  probability measures of the form $\smash{\mu_1^k = \lambda_1^k\,\delta_{x_1^k} + \dots + \lambda_k^k\,\delta_{x_k^k}}$  which weakly converges to $\mu_1$; here, given any   $k\in\N$,  $\smash{x_1^k,\dots,x_k^k}$ are points in $\supp\mu_1$, and $\smash{\lambda_1^k,\dots,\lambda_k^k}$ are nonnegative real numbers which sum up to $1$.  Since $\supp\mu_0\times\supp\mu_1^k \subset\{l>0\}$, $\smash{(\mu_0,\mu_1^k)}$ is $q$-separated by Lemma \ref{Le:Existence separation}. 
Theorem \ref{Th:Char optimal maps} and Corollary \ref{Cor:Char geos} imply the existence of appropriate Borel maps $\smash{\FF_t^k\colon M \longrightarrow M}$, where $t\in[0,1]$, with the properties
\begin{itemize}
    \item $\smash{\FF_1^k}$ induces the unique $\ell_q$-optimal coupling  of $\smash{(\mu_0,\mu_1^k)}$, and
    \item $((\FF_t^k)_\sharp\mu_0)_{t\in[0,1]}$ constitutes the unique $q$-geodesic from $\mu_0$ to $\smash{\mu_1^k}$.
\end{itemize}
Let $\smash{A_i^k\subset \{\rho_0 > 0\}}$ be the set of all points $x\in \{\rho_0>0\}$ with $\smash{\FF_1^k(x) = x_i^k}$, where $i=1,\dots,k$. Then $\smash{A_1^k,\dots,A_k^k}$ are mutually disjoint, and their union has full $\mu_0$-measure.
Given any $i=1,\dots,k$, Theorem \ref{Th:Char optimal maps} and Remark \ref{Re:Zs s-intermediate points} imply the curve $\smash{(\FF_t^k(x))_{t\in[0,1]}}$ to form the unique maximizing timelike geodesic from $x$ to $\smash{x_i^k}$ for $\mu_0$-a.e.~$\smash{x\in A_i^k}$.
Since $\mu_0$ and $\meas$ are equivalent on $\{\rho_0>0\}$, the same statement holds for every $\smash{x\in B_i^k}$, where $\smash{B_i^k\subset A_i^k}$ is an appropriate Borel set with $\smash{\meas[A_i^k \setminus B_i^k]=0}$. In particular $\smash{Z_t(E', x_i^k) = \FF_t^k(E')}$ for every Borel set $E'\subset B_i^k$, for which Lemma \ref{Le:mcp} (applied to the Finsler spacetime built by the causal diamond $\smash{J^+(\mu_0)\cup J^-(\mu_1)}$) implies
\begin{equation*}
    \meas\big[\FF_t^k(E')\big] \geq \inf_{x\in\supp\mu_0,y\in\supp\mu_1}\tau_{K,N}^{(1-t)}\big(l(x,y)\big)^N\,\meas[E'].
\end{equation*}
In particular, by appropriately decomposing an arbitrary Borel set $E\subset \{\rho_0>0\}$ we get
\begin{equation*}
    \meas\big[\FF_t^k(E)\big] \geq \inf_{x\in\supp\mu_0,y\in\supp\mu_1}\tau_{K,N}^{(1-t)}\big(l(x,y)\big)^N\, \meas[E].
\end{equation*}
This inequality can be rewritten as
\begin{equation*}
    \meas[F] \geq \inf_{x\in\supp\mu_0,y\in\supp\mu_1}\tau_{K,N}^{(1-t)}\big(l(x,y)\big)^N\,\meas\big[(\FF_t^k)^{-1}(F)\cap \{\rho_0>0\}\big]
\end{equation*}
for every Borel set $\smash{F\subset \FF_t^k(\{\rho_0>0\})}$, or equivalently 
\begin{equation}\label{Eq:Pre-inequ}
    \int_{\{\rho_0>0\}} f\circ \FF_t^k\,\d\meas \leq \sup_{x\in\supp\mu_0,y\in\supp\mu_1}\tau_{K,N}^{(1-t)}\big(l(x,y)\big)^{-N}\int_M f\,\d\meas
\end{equation}
for every function $f$ as above.
Now by the $q$-separation of $(\mu_0,\mu_1)$ and the stability of optimal maps \cite[Corollary~5.23]{villani2009}, $\smash{(\FF_t^k)_{k\in\N}}$ converges to $\FF_t$ in $\mu_0$-measure, and thus $\FF_t^k \to \FF_t$ $\meas$-a.e.~on $\{\rho_0>0\}$ as $k \to +\infty$ up to a non-relabeled subsequence.
Passing to the limit in \eqref{Eq:Pre-inequ} then shows  \eqref{Eq:Claim}.

An approximation argument now allows us to extend \eqref{Eq:Claim} to nonnegative Borel functions $f$ on $M$.
For $f := 1_{\FF_t(\{\rho_0>0\})}\,h\,\rho_0\circ \FF_t^{-1}$, where $h$ is a given nonnegative and bounded Borel function on $M$, this translates into
\begin{align*}
     \int_{\FF_t(\{\rho_0>0\})} h\,\d\mu_t &= \int_{\{\rho_0>0\}} h\circ \FF_t\,\d\mu_0\\
     &= \int_{\{\rho_0>0\}} (h\circ \FF_t)\,\rho_0\,\d\meas\\
     &\leq \sup_{x\in\supp\mu_0,y\in\supp\mu_1}\tau_{K,N}^{(1-t)}\big(l(x,y)\big)^{-N}\int_M h\,\rho_0\circ \FF_t^{-1}\,\d\meas.
\end{align*}
This implies the claim under our initial assumptions on $\mu_0$ and $\mu_1$.

The general case now follows by sub-partition as in \cite[Proposition 3.38]{braun2022}.
The proof is similar to the one of Corollary \ref{Cor:Optimal maps without duality}.
We skip the details.
\end{proof}


\subsection{Timelike curvature-dimension condition}

To introduce the timelike curvature-dimension conditions we will study in a moment, we need to define several entropy functionals on $\mathcal{P}(M)$. Though we give a definition for general probability measures, in fact we only deal with $\meas$-absolutely continuous measures. 

\begin{definition}[Entropies]\label{df:entropy}
Let $\mu=\rho\,\meas +\mu_{\perp}$ denote the Lebesgue decomposition of a given probability measure $\mu\in\Prob(M)$ into its $\meas$-absolutely continuous and $\meas$-singular parts.
\begin{enumerate}[label=\textnormal{(\arabic*)}]
    \item The \emph{Boltzmann--Shannon entropy} of $\mu$ is defined by
    \[ \Ent_{\meas}(\mu) :=\int_M \rho \log \rho \,\d\meas \]
if $\mu \in \mathcal{P}^{\ac}(M,\meas)$ and $\int_{\{\rho<1\}} \rho\log\rho \,\d\meas>-\infty$, otherwise $\Ent_{\meas}(\mu):=+\infty$.
\item For $n \leq N < +\infty$ we define the \emph{$N$-Rényi entropy} of $\mu$ by
\begin{align*}
    S_\meas^N(\mu) := -\int_M\rho^{(N-1)/N}\,\d\meas.
\end{align*}

\item For $N<0$, the corresponding \emph{$N$-Rényi entropy} of $\mu$ is
\begin{align*}
    S_\meas^N(\mu) := \int_M \rho^{(N-1)/N}\,\d\meas
\end{align*}
if $\mu\in\Prob^\ac(M,\meas)$, otherwise $S_\meas^N(\mu) := +\infty$.

\item The \emph{$0$-Rényi entropy} of $\mu$ is given by
\begin{align*}
    S_\meas^0(\mu) := \esssup_{x\in M} \rho(x) = \Vert \rho\Vert_{L^\infty(M,\meas)}
\end{align*}
if $\mu\in \Prob^\ac(M,\meas)$, otherwise $S_\meas^0(\mu) := +\infty$.
\end{enumerate}
\end{definition}

We remark that the generation functions $r\longmapsto r\log r$, $\smash{r\longmapsto -r^{(N-1)/N}}$ (for $n \leq N < +\infty$), and $\smash{r\longmapsto r^{(N-1)/N}}$ (for $N<0$) are all convex on $[0,\infty)$.
(This explains the change of sign for $S_\meas^N$ in passing from positive to negative $N$.)
In addition, we have $\smash{\Ent_{\meas}(\mu) \ge -\log\meas[\supp\mu]}$ for $\mu \in \mathcal{P}_{\comp}(M)$ by Jensen's inequality, and similarly we obtain a lower bound for $\smash{S_\meas^N(\mu)}$ for $N\in (-\infty,0] \cup [n,+\infty)$, see the proof of Theorem \ref{Th:Brunn Minkowski}.

Now we introduce the \emph{timelike curvature-dimension condition} $\smash{\TCD_q(K,N)}$ partly using the function $\smash{\tau_{K,N}^{(t)}}$ in \eqref{eq:tau}.

\begin{definition}[Timelike curvature-dimension condition]\label{df:TCD}
Let $q \in (0,1)$, $K \in \R$, and $N \in (-\infty,0] \cup [n,+\infty]$.
We say that the weighted Finsler spacetime $(M,L,\meas)$ satisfies $\TCD_q(K,N)$ if, for every $q$-separated $(\mu_0,\mu_1) \in \smash{\mathcal{P}^{\ac}_{\comp}(M,\meas)}\times\smash{\mathcal{P}^{\ac}_{\comp}(M,\meas)}$  there exist a $q$-geodesic $(\mu_t)_{t\in[0,1]}$ from $\mu_0$ and $\mu_1$ and an $\smash{\ell_q}$-optimal coupling $\pi\in\Pi(\mu_0,\mu_1)$ such that the following hold.
\begin{enumerate}[label=\textnormal{(\arabic*)}]
\item
When $N=+\infty$, for every $t\in[0,1]$ we have
\begin{align}\label{eq:TCDinf}
\begin{split}
\Ent_{\meas}(\mu_t) &\le (1-t)\,\Ent_{\meas}(\mu_0) +t\,\Ent_{\meas}(\mu_1)\\
&\qquad\qquad -\frac{K}{2}\,t\,(1-t) \int_{M \times M} l(x,y)^2 \,\d\pi(x, y).
\end{split}
\end{align}

\item
When $n \leq N <+\infty$, for every $N' \geq N$ and every $t\in[0,1]$ we have
\begin{align}\label{eq:TCDpos}
\begin{split}
S_{\meas}^{N'}(\mu_t) &\le
-\int_{M \times M} \tau_{K,N'}^{(1-t)}\big( l(x,y) \big)\, \rho_0(x)^{-1/N'}\,\d\pi(x,y)\\
&\qquad\qquad - \int_{M\times M} \tau_{K,N'}^{(t)}\big( l(x,y) \big) \,\rho_1(y)^{-1/N'} \,\d\pi(x, y).
\end{split}
\end{align}

\item
When $N<0$, for every $N'\in [N,0)$ and every $t\in[0,1]$ we have
\begin{align}\label{eq:TCDneg}
\begin{split}
S_{\meas}^{N'}(\mu_t) &\le
\int_{M \times M}  \tau_{K,N'}^{(1-t)}\big( l(x,y) \big)\, \rho_0(x)^{-1/N'} \,\d\pi(x,y)\\
&\qquad\qquad +\int_{M\times M}\tau_{K,N'}^{(t)}\big( l(x,y) \big) \,\rho_1(y)^{-1/N'} \d\pi(x, y),
\end{split}
\end{align}
where we set $\smash{\tau_{K,N'}^{(t)}(r):=+\infty}$ if $K<0$ and $\smash{r \ge \pi\sqrt{(N'-1)/K}}$.

\item
When $N=0$, for every $t\in[0,1]$ we have
\begin{align}\label{eq:TCD0}
\begin{split}
S_{\meas}^0(\mu_t) &\le \esssup_{(x,y)\in \supp\pi} \max\bigg\{ \frac{\mathfrak{s}_{-K}((1-t)\,l(x,y))}{(1-t)\,\mathfrak{s}_{-K}(l(x,y))}\, \rho_0(x),\\
&\qquad\qquad \frac{\mathfrak{s}_{-K}(t\,l(x,y))}{t\,\mathfrak{s}_{-K}(l(x,y))} \,\rho_1(y) \bigg\},
\end{split}
\end{align}
where we set $\smash{\mathfrak{s}_{-K}(tl(x,y))/(t\mathfrak{s}_{-K}(l(x,y))):=1}$ if $l(x,y)=0$.
\end{enumerate}
\end{definition}

We refer to \cite{ohta-negative} and \cite{ohta-needle} for the curvature-dimension condition for Riemannian or Finsler manifolds with $N<0$ and $N=0$, respectively.
Note that, in the case of $N<0$ and $K<0$, $\TCD_q(K,N)$ provides only a local control since the defining inequality \eqref{eq:TCDneg} is trivial if
\[
\pi\big[\big\{ (x,y) \in M \times M \,\big\vert\, l(x,y) \ge \pi\sqrt{(N-1)/K} \big\} \big] >0.
\]

We also remark that $\pi$ is an $\ell_q$-optimal coupling of $\mu_0$ and $\mu_1$, thereby the last term  $\int_{M \times M}l(x,y)^2 \,\d\pi(x,y)$ in \eqref{eq:TCDinf} is not directly written as a transport cost. In fact, $\pi$ is a unique $\smash{\ell_q}$-optimal coupling of $\mu_0$ and $\mu_1$ by Lemma \ref{Le:Existence separation} and Theorem \ref{Th:Char optimal maps}, and $(\mu_t)_{t\in[0,1]}$ is a unique $q$-geodesic connecting its endpoints by Corollary \ref{Cor:Char geos}. Here we have implicitly used that all couplings of $\mu_0$ and $\mu_1$ are chronological, which leads us to a further remark on the chronology assumption on $\mu_0$ and $\mu_1$ in Definition \ref{df:TCD}.

\begin{remark}[Timelike $q$-dualizability]\label{Re:About q-sep} 
The hypothesis of $q$-separation in Definition \ref{df:TCD}  lies in between strong timelike $p$-dualizability and timelike $p$-dualizability of $(\mu_0,\mu_1)$ \cite[Definitions 2.18, 2.27]{cavalletti2020}. These are the properties in whose terms strong and weak  timelike curvature-dimension conditions have been defined in \cite{braun2022, cavalletti2020}. In the young literature about TCD spaces this distinction is still necessary for stability questions, since chronology is not a closed condition.

The results in Theorem \ref{th:Ric-TCD} and Theorem \ref{Th:TCD-Ric} still hold if Definition \ref{Def:q-separated} was instead given for timelike $p$-dualizable or strongly timelike $p$-dualizable pairs $(\mu_0,\mu_1)$.
For $N\in [n,+\infty)$ the equivalence of these various TCD conditions is due to \cite{braun2022}.  The proof employs a construction similar to Corollary \ref{Cor:Optimal maps without duality}, convexity, and a timelike non-branching property (which is clear in our case). With some work, this can be extended to the range $N\in (-\infty,0]\cup \{+\infty\}$. One could even merely assume $(\mu_0,\mu_1)$ to obey  $\smash{\supp\mu_0\times\supp\mu_1 \subset\{l>0\}}$, cf.~Lemma \ref{Le:Existence separation}. This property should be easier to verify in practice, but as observed in Section \ref{Sec:Kantorovich duality}, unlike the other conditions (see Proposition \ref{Pr:Star-shaped} and \cite[Lemma 3.1]{braun2023}) it does not propagate along $q$-geodesics.
\end{remark}

Integrating the behavior of optimal transports in the previous section and the control of $\meas$ by the weighted Ricci curvature, we arrive at the following main result.

\begin{theorem}[$\smash{\Ric_N \ge K}$ implies $\smash{\TCD_q(K,N)}$]\label{th:Ric-TCD}
Let $(M,L,\meas)$ be a globally hyperbolic weighted Finsler spacetime satisfying $\Ric_N \ge K$ in timelike directions for some $K \in \R$ and $N \in (-\infty,0] \cup [n,+\infty]$.
Then it satisfies $\TCD_q(K,N)$ for any $q \in (0,1)$.
\end{theorem}

\begin{proof}
We divide the proof into five steps.

\begin{step}
We start with some preparations.  Throughout this proof, we fix a pair $\smash{(\mu_0,\mu_1) \in \mathcal{P}^{\ac}_{\comp}(M,\meas)\times \Prob_\comp^\ac(M,\meas)}$ which is $q$-separated, say by $(\pi,u,v)$. Let $(\mu_t)_{t \in [0,1]}$ be the unique $q$-geodesic which connects  $\mu_0$ to $\mu_1$.
Recall from Theorem~\ref{Th:Char optimal maps} and Corollary~\ref{Cor:Char geos} that $\mu_t=(\FF_t)_{\sharp}\mu_0$ for every $t\in(0,1)$, where
\[
\FF_t(x) :=\exp_x\! \big( t\,F^*(\d u)^{p-2} \,\mathscr{L}^*(\d u) \big).
\]
Moreover, $\mu_t = \rho_t\,\meas$ is $\meas$-absolutely continuous, and Corollary~\ref{Cor:Monge Ampère} yields
\[ \rho_0(x) =\rho_t \big( \FF_t(x) \big) \,\det_{\meas} [\d\FF_t(x)] \]
for $\mu_0$-a.e.~$x\in M$. 
To be more precise, if $\d u(x) \neq 0$ (equivalently, $x \neq \FF_1(x)$), one can write down
\[ \det_{\meas} [\d\FF_t(x)] =\e^{\psi_{\gamma}(0)-\psi_{\gamma}(t)} \,\det[\d\FF_t(x)] \]
for the maximizing geodesic $\gamma \colon [0,1] \longrightarrow M$ from $x$ to $\FF_1(x)$ and $\psi_{\gamma}$ given by \eqref{eq:psi_g}, where $\det[\d\FF_t(x)]$ denotes the Jacobian associated with the measure $\smash{\vol_{g_{\dot{\gamma}}}}$ along $\gamma$ induced from the Lorentzian metric $g_{\dot{\gamma}}$ as in \eqref{eq:g_v}.
\end{step}

\begin{step}
Now, we can follow essentially the same lines as the Lorentzian case in \cite{Mc},
however, we need more detailed calculations to improve $\TCD^*(K,N)$ to $\TCD(K,N)$.
Here we give the detailed proof for completeness.
Fix $x \in \supp\mu_0$ with $\d u(x) \neq 0$ as above and take an $g_{\dot{\gamma}}$-orthonormal basis
$(e_{\alpha})_{\alpha=1}^n$ of $T_xM$ with $e_1=\dot{\gamma}(0)/F(\dot{\gamma}(0))$.
That is, we have $g_{\dot{\gamma}}(e_1,e_1)=-1$, $g_{\dot{\gamma}}(e_i,e_i)=1$ for $i \ge 2$, 
and $g_{\dot{\gamma}}(e_{\alpha},e_{\beta})=0$ for $1 \le \alpha<\beta \le n$.
We consider the $g_{\dot{\gamma}}$-parallel vector fields $E_{\alpha}$ along $\gamma$
with $E_{\alpha}(0)=e_{\alpha}$ (i.e.\ $\smash{D^{\dot{\gamma}}_{\dot{\gamma}} E_{\alpha} = 0}$) 
as well as the Jacobi fields $J_{\alpha}(t):=\d \FF_t(x)(e_{\alpha})$ along $\gamma$. 
We remark that $J_1(t)$ is not necessarily tangent to $\dot{\gamma}(t)$.
For simplicity, we denote the covariant derivative along $\gamma$ by $'$,
namely $\smash{J'_{\alpha} :=D^{\dot{\gamma}}_{\dot{\gamma}} J_{\alpha}}$.
Then we define the $n \times n$ matrix $\smash{J(t)=(J_{\alpha \beta}(t))_{\alpha,\beta=1}^n}$ by
\[
J_{\alpha}(t) :=\sum_{\beta=1}^n J_{\alpha \beta}(t)\, E_{\beta}(t).
\]
We have $J''_{\alpha} =\sum_{\beta =1}^n J''_{\alpha \beta}\, E_\beta$ and observe from the Jacobi equation that
\begin{align*}
J''_{\alpha} &=-R_{\dot{\gamma}}(J_{\alpha})
 =-\sum_{\delta=1}^n J_{\alpha\delta}\, R_{\dot{\gamma}}(E_{\delta}) \\
&= \sum_{\delta=1}^n J_{\alpha\delta}\, g_{\dot{\gamma}}\big( R_{\dot{\gamma}}(E_{\delta}),E_1 \big) E_1
 -\sum_{\delta=1}^n \sum_{j=2}^n J_{\alpha\delta} \, g_{\dot{\gamma}}\big( R_{\dot{\gamma}}(E_{\delta}),E_j \big) E_j \\
&= -\sum_{\beta,\delta=1}^n J_{\alpha\delta}\, g_{\dot{\gamma}}\big( R_{\dot{\gamma}}(E_{\delta}),E_{\beta} \big) E_{\beta};
\end{align*}
in the last line, we have used the identity 
 $g_{\dot{\gamma}}(R_{\dot{\gamma}}(E_{\delta}),E_1) =g_{\dot{\gamma}}(R_{\dot{\gamma}}(E_1),E_{\delta}) =0$
provided by $E_1=\dot{\gamma}/F(\dot{\gamma})$ and \eqref{eq:R_v}.
Hence, by letting
\[
B:=J'\, J^{-1}, \qquad
R(t) :=\Big( g_{\dot\gamma}\big( R_{\dot{\gamma}}(E_{\alpha}),E_{\beta}\big)(t) \Big)_{\alpha,\beta=1}^n,
\]
we have the \emph{Riccati equation}
\begin{equation}\label{eq:matrix-Riccati}
B' =J''\, J^{-1} -(J'\, J^{-1})^2 =-J\, R\, J^{-1} -B^2.
\end{equation}
This can be rewritten as 
\[
D' =-J^{-1}\, J'\, J^{-1}\, B\,J -R -J^{-1}\, B^2\, J +J^{-1}\, B\,J' =-R -D^2,
\]
where $D:=J^{-1}\,B\,J$. Therefore, taking the trace yields
\begin{equation}\label{eq:trace-Riccati}
\trace[B'] +\trace[B^2] =\trace[D'] +\trace[D^2] =-\Ric(\dot{\gamma}).
\end{equation}
\end{step}

\begin{step}
Next, we have a closer look on $\trace[B^2]$ at $t=0$.
This is the step where we need a more delicate analysis to achieve $\TCD(K,N)$.
First of all, since $J_{\alpha}(0)=E_{\alpha}(0)$, $J(0)$ is the identity matrix and we have $D(0)=B(0)$.
By the choice of $(e_{\alpha})_{\alpha=1}^n$,
\[
B_{\alpha 1}(0) =-g_{\dot{\gamma}} \big( J'_{\alpha}(0),e_1 \big), \qquad
B_{\alpha j}(0) =g_{\dot{\gamma}} \big( J'_{\alpha}(0),e_j \big)
\]
for every $j \ge 2$.
We recall from Lemma~\ref{Le:Maps Jacobian}(iii) that
\[
J'_{\alpha}(0)
 = F^*(\d u)^{p-2} \, \nabla^2 u(e_{\alpha}) +\d[ F^*(\d u)^{p-2}] (e_{\alpha}) \, \nabla u(x),
\]
and observe from the identities $\smash{F^*(\d u)^2 =-g_{\nabla u}(\nabla u,\nabla u)}$, compare with \cite[(3.2)]{LMO1}, and $\smash{e_1=\nabla u(x)/F^*(\d u)}$ that
\begin{align*}
\d[ F^*(\d u)^{p-2}] (e_{\alpha}) \, \nabla u(x)
&= -(p-2)\, F^*(\d u)^{p-4}\, g_{\nabla u} \big( \nabla^2 u(e_{\alpha}),\nabla u \big) \, \nabla u(x) \\
&= -(p-2)\, F^*(\d u)^{p-2}\, g_{\nabla u} \big( \nabla^2 u(e_{\alpha}),e_1 \big) \, e_1.
\end{align*}
Therefore,
\[
B_{\alpha 1}(0) =(1-p)\, F^*(\d u)^{p-2}\, g_{\nabla u} \big( \nabla^2 u(e_{\alpha}),e_1 \big),
\]
while for $j \ge 2$ we obtain
\[
B_{\alpha j}(0) =F^*(\d u)^{p-2}\, g_{\nabla u} \big( \nabla^2 u(e_{\alpha}),e_j \big).
\]
Thus we can decompose $B(0) =F^*(\d u)^{p-2}\, H^u\, Q^2$, where
\[
H^u :=\Big( g_{\nabla u}\big( \nabla^2 u(e_{\alpha}),e_{\beta} \big) \Big)_{\alpha,\beta=1}^n, \qquad
Q:=\textnormal{diag}\big(\sqrt{1-p}, 1,\dots,1\big).
\]
Note that $H^u$ is symmetric  \cite[Lemma~4.13]{LMO2}.

Since $D(0)=B(0)$, we deduce from \eqref{eq:matrix-Riccati}, \eqref{eq:R_v}, and the symmetry of $H^u$ that
\begin{align} \label{eq:D11}
\begin{split}
D'_{11}(0) &=-\sum_{\alpha=1}^n D_{1 \alpha}(0) \, D_{\alpha 1}(0)\\
&= -F^*(\d u)^{2p-4}\, \big( H^u\, Q^2\, H^u\, Q^2 \big)_{11}\\
&= -F^*(\d u)^{2p-4}\, \Bigg( (1-p)^2\, (H^u_{11})^2 +(1-p)\,\sum_{j=2}^n (H^u_{1j})^2 \Bigg)\\
&\le -F^*(\d u)^{2p-4}\, (1-p)^2\, (H^u_{11})^2\\
&= -D_{11}(0)^2.
\end{split}
\end{align}
We set $\phi(t) :=\log \det[\d\FF_t(x)] =\log \det[J(t)]$ and decompose it as
\[
\phi_1(t) := \int_0^t D_{11}(s) \,\d s, \qquad
\phi_2(t) := \phi(t) -\phi_1(t).
\]
Note that $\phi(0)=\phi_1(0)=\phi_2(0)=0$ and
\begin{align*}
\phi'(t) &= \trace\!\big[J'(t)\, J(t)^{-1} \big] =\trace[B(t)] =\trace[D(t)], \\
\phi'_2(0) &=\sum_{i=2}^n D_{ii}(0) =F^*(\d u)^{p-2}\, \sum_{i=2}^n H^u_{ii},
\end{align*}
and \eqref{eq:D11} yields
\begin{equation}\label{eq:phi_1}
\big[\e^{\phi_1(t)}\big]''_{0} =\big( D'_{11}(0) +D_{11}(0)^2 \big)\,\e^{\phi_1(0)} \le 0.
\end{equation}
Moreover, \eqref{eq:trace-Riccati} implies
\[
\phi''_2(t) =\phi''(t) -\phi''_1(t)
 =-\Ric \big( \dot{\gamma}(t) \big) -\trace[D(t)^2] -D'_{11}(t).
\]
If follows from the symmetry of $H^u$ and the Cauchy--Schwarz inequality that
\begin{align*}
\trace[D(0)^2] +D'_{11}(0)
&= \sum_{\alpha,\beta=1}^n D_{\alpha\beta}(0) \, D_{\beta\alpha}(0)
 -\sum_{\beta=1}^n D_{1 \beta}(0)\, D_{\beta 1}(0) \\
&= F^*(\d u)^{2p-4}\, \Bigg( (1-p)\, \sum_{j=2}^n (H^u_{1j})^2 +\sum_{i,j=2}^n (H^u_{ij})^2 \Bigg) \\
&\ge F^*(\d u)^{2p-4}\, \sum_{i=2}^n (H^u_{ii})^2 \\
&\ge F^*(\d u)^{2p-4}\, \frac{1}{n-1} \Bigg( \sum_{i=2}^n H^u_{ii} \Bigg)^2 \\
&= \frac{1}{n-1}\, \phi'_2(0)^2.
\end{align*}
Combining these we obtain
\begin{equation}\label{eq:phi_2}
\frac{\big[\e^{\phi_2(t)/(n-1)}\big]''_0}{\e^{\phi_2(0)/(n-1)}}
 =\frac{\phi''_2(0)}{n-1} +\frac{\phi'_2(0)^2}{(n-1)^2}
 \le -\frac{\Ric(\dot{\gamma}(0))}{n-1}.
\end{equation}

We further consider the functions
\begin{alignat*}{3}
c(t) &:=\det_{\meas}[\d \FF_t(x)]^{1/N}, && \qquad
& c_1(t) & :=\exp\!\bigg( \frac{\psi_{\gamma}(0)-\psi_{\gamma}(t)}{N-n} \bigg), \\
c_2(t) &:=\e^{\phi(t)/n}, &&&
c_3(t) &:=c_1(t)^{(N-n)/(N-1)}\, \e^{\phi_2(t)/(N-1)},
\end{alignat*}
where we assume $N \neq 0$ for $c(t)$ and $c_2(t)$ will be used in the proof of Theorem~\ref{Th:TCD-Ric}.
Note that, since $N \in (-\infty,0] \cup (n,\infty)$,
\begin{align*}
(N-1)\,\frac{c''_3}{c_3} &= (N-n)\,\frac{c_1''}{c_1} + (n-1)\,\frac{[\e^{\phi_2/(n-1)}]''}{\e^{\phi_2/(n-1)}}\\
&\qquad\qquad -\frac{(n-1)(N-n)}{N-1}\,\bigg(\frac{c'_1}{c_1} -\frac{[\e^{\phi_2/(n-1)}]'}{\e^{\phi_2/(n-1)}} \bigg)^2 \\
&\le -\psi''_{\gamma} +\frac{(\psi'_{\gamma})^2}{N-n} +(n-1)\,\frac{[\e^{\phi_2/(n-1)}]''}{\e^{\phi_2/(n-1)}}.
\end{align*}
Therefore, at $t=0$, we obtain from \eqref{eq:phi_2} and the hypothesis $\Ric_N \ge K$ that
\begin{equation}\label{eq:c_3}
(N-1)\,\frac{c''_3(0)}{c_3(0)} \le -\Ric_N \big( \dot{\gamma}(0) \big) \le -K\,F^2 \big( \dot{\gamma}(0) \big).
\end{equation}
\end{step}

\begin{step}
We shall generalize \eqref{eq:phi_1} and \eqref{eq:c_3} to $t>0$.
When we perform the same calculation at $t_0 \in (0,1)$
with the orthonormal basis $(E_{\alpha}(t_0))_{\alpha=1}^n$,
then $J(t)$, $B(t)$, and $D(t)$ are respectively replaced by
\[
J^{t_0}(t) =J(t_0)^{-1}\, J(t), \qquad B^{t_0}(t) =J(t_0)^{-1}\, B(t)\, J(t_0), \qquad
D^{t_0}(t) =D(t)
\]
for $t \ge t_0$.
This immediately implies \eqref{eq:phi_1} as well as \eqref{eq:phi_2} at $t_0$,
and then \eqref{eq:c_3} also follows.

On the one hand, we deduce from the concavity inequality \eqref{eq:phi_1} that
\begin{equation}\label{eq:phi_1+}
\e^{\phi_1(t)} \ge (1-t)\,\e^{\phi_1(0)} +t\,\e^{\phi_1(1)} =(1-t) +t\,\e^{\phi_1(1)}.
\end{equation}
On the other hand, \eqref{eq:c_3} implies that $-(N-1)\log c_3$ is $(KF^2(\dot{\gamma}),N-1)$-convex 
in the sense of \cite[Lemma~2.2]{EKS} for $N \in (n,+\infty)$ or \cite[Lemma~2.1]{ohta-negative} for $N \le 0$, respectively. 
Thus, we obtain
\begin{equation}\label{eq:c_3+}
c_3(t) \ge \frac{\mathfrak{s}_{K/(N-1)}((1-t)\,d(x))}{\mathfrak{s}_{K/(N-1)}(d(x))}
 +\frac{\mathfrak{s}_{K/(N-1)}(t\,d(x))}{\mathfrak{s}_{K/(N-1)}(d(x))}\, c_3(1)
\end{equation}
for $N \in (n,+\infty)$, and
\begin{equation}\label{eq:c_3-}
c_3(t) \le \frac{\mathfrak{s}_{K/(N-1)}((1-t)\,d(x))}{\mathfrak{s}_{K/(N-1)}(d(x))}
 +\frac{\mathfrak{s}_{K/(N-1)}(t\,d(x))}{\mathfrak{s}_{K/(N-1)}(d(x))}\, c_3(1)
\end{equation}
for $N \le 0$, where we set $d(x):=l(x,\FF_1(x))=F(\dot{\gamma})$ for simplicity.
\end{step}

\begin{step}
In this final step, we turn to the proofs of the respective entropic displacement semi-convexity.
We first consider the case of $N \in (n,+\infty)$.
We remark that, thanks to the monotonicity as in Remark~\ref{rm:Ric_N},
it suffices to show \eqref{eq:TCDpos} for $N'=N$.
Combining \eqref{eq:phi_1+}, \eqref{eq:c_3+}, and  Hölder's inequality, we obtain
\begin{align*}
c(t) &= c_3(t)^{(N-1)/N}\, \e^{\phi_1(t)/N} \\
&\ge \bigg(\frac{\mathfrak{s}_{K/(N-1)}((1-t)\,d(x))}{\mathfrak{s}_{K/(N-1)}(d(x))}
 +\frac{\mathfrak{s}_{K/(N-1)}(t\,d(x))}{\mathfrak{s}_{K/(N-1)}(d(x))}\, c_3(1) \bigg)^{(N-1)/N} \\
&\qquad\qquad \times \big( (1-t) +t\,\e^{\phi_1(1)} \big)^{1/N} \\
&\ge \bigg(\frac{\mathfrak{s}_{K/(N-1)}((1-t)\,d(x))}{\mathfrak{s}_{K/(N-1)}(d(x))} \bigg)^{(N-1)/N}
 (1-t)^{1/N} \\
&\qquad\qquad +\bigg(\frac{\mathfrak{s}_{K/(N-1)}(t\,d(x))}{\mathfrak{s}_{K/(N-1)}(d(x))}\, c_3(1) \bigg)^{(N-1)/N}
 \big(t\,\e^{\phi_1(1)} \big)^{1/N} \\
&= \tau^{(1-t)}_{K,N}\big( d(x) \big) +\tau^{(t)}_{K,N}\big( d(x) \big) \, c(1).
\end{align*}
We remark that this inequality holds as equality for $\mu_0$-a.e.\ $x \in \{\d u=0\}$ since $d(x)=0$.
Therefore, with the help of Corollary \ref{Cor:Monge Ampère}, we obtain the desired convexity \eqref{eq:TCDpos} of $S_{\meas}^N$ as
\begin{align*}
S^N_{\meas}(\mu_t)
 &= -\int_M \rho_t\big( \FF_t(x) \big)^{(N-1)/N}\, \det_{\meas}[\d\FF_t(x)] \,\d\meas(x) \nonumber\\
&= -\int_M \rho_0(x)^{(N-1)/N}\, \det_{\meas}[\d\FF_t(x)]^{1/N} \,\d\meas(x) \\
&\le -\int_M \rho_0(x)^{(N-1)/N} \tau_{K,N}^{(1-t)}\big( d(x) \big)\,\d\meas(x) \nonumber\\
&\qquad\qquad -\int_M \rho_0(x)^{(N-1)/N}\,\tau_{K,N}^{(t)}\big( d(x) \big) \,\det_{\meas}[\d\FF_1(x)]^{1/N}\, \d\meas(x) \nonumber\\
&= -\int_M  \tau_{K,N}^{(1-t)}\big( d(x) \big)\, \rho_0(x)^{(N-1)/N}\,\d\meas(x) \nonumber\\
&\qquad\qquad -\int_M \tau_{K,N}^{(t)}\big( d(x) \big)\, \rho_1\big( \FF_1(x) \big)^{(N-1)/N}\, \det_{\meas} [\d\FF_1(x)] \,\d\meas(x) \nonumber\\
&= -\int_{M \times M} \tau_{K,N}^{(1-t)}\big( l(x,y) \big)\, \rho_0(x)^{-1/N}\,\d\pi(x,y)\nonumber\\
&\qquad\qquad -\int_{M\times M}\tau_{K,N}^{(t)}\big( l(x,y) \big) \,\rho_1(y)^{-1/N} \,\d\pi(x, y).\nonumber
\end{align*}
This completes the proof in the case of $N \in (n,+\infty)$.
Then the case of $N=n$ is obtained by a limit argument, since $\Ric_n \ge K$ implies $\Ric_N \ge K$ for $N>n$.

For $N<0$, we make use of \eqref{eq:c_3-} instead of \eqref{eq:c_3+} to see
\begin{align*}
c(t) &= c_3(t)^{(N-1)/N}\, \e^{\phi_1(t)/N} \\
&\le \bigg(\frac{\mathfrak{s}_{K/(N-1)}((1-t)\,d(x))}{\mathfrak{s}_{K/(N-1)}(d(x))}
 +\frac{\mathfrak{s}_{K/(N-1)}(t\,d(x))}{\mathfrak{s}_{K/(N-1)}(d(x))}\, c_3(1) \bigg)^{(N-1)/N} \\
&\qquad\qquad \times \big((1-t) +t\,\e^{\phi_1(1)} \big)^{1/N} \\
&\le \bigg(\frac{\mathfrak{s}_{K/(N-1)}((1-t)\,d(x))}{\mathfrak{s}_{K/(N-1)}(d(x))} \bigg)^{(N-1)/N}\, (1-t)^{1/N} \\
&\qquad\qquad +\bigg(\frac{\mathfrak{s}_{K/(N-1)}(t\,d(x))}{\mathfrak{s}_{K/(N-1)}(d(x))}\, c_3(1) \bigg)^{(N-1)/N}\,
 \big( t\,\e^{\phi_1(1)} \big)^{1/N} \\
&= \tau^{(1-t)}_{K,N}\big( d(x) \big) +\tau^{(t)}_{K,N}\big( d(x) \big) \, c(1),
\end{align*}
where the latter inequality follows from the H\"older inequality of the form
\begin{align*}
(a_1 +b_1)^{(N-1)/N}
&=\bigg( \frac{a_1 \cdot a_2^{1/(N-1)}}{a_2^{1/(N-1)}} +\frac{b_1 \cdot b_2^{1/(N-1)}}{b_2^{1/(N-1)}} \bigg)^{(N-1)/N} \\
&\le (a_1^{(N-1)/N} a_2^{1/N} +b_1^{(N-1)/N} b_2^{1/N})\, (a_2 +b_2)^{-1/N}
\end{align*}
for $a_1,a_2,b_1,b_2 >0$. Then \eqref{eq:TCDneg} is deduced along the same lines as above.

Next, in the case of $N=+\infty$, we more directly see that
$\log[\det_{\meas}[\d\FF_t(x)]] =\psi_{\gamma}(0) -\psi_{\gamma}(t) +\phi(t)$ satisfies
\begin{align*}
[\psi_{\gamma}(0) -\psi_{\gamma}(t) +\phi(t)]''
&= -\psi''_{\gamma}(t) -\trace[D(t)^2] -\Ric\big( \dot{\gamma}(t) \big) \\
&\le -\Ric_{\infty}\big( \dot{\gamma}(t) \big)
\le -K\,d(x)^2.
\end{align*}
This implies
\[
\log \det_{\meas}[\d \FF_t(x)]  \ge t\log \det_{\meas}[\d \FF_1(x)]  +\frac{K}{2}\,t\,(1-t)\,d(x)^2
\]
(which again holds true also for $\mu_0$-a.e.\ $x \in \{\d u=0\}$ since $d(x)=0$),
and we deduce the inequalities
\begin{align*}
\Ent_{\meas}(\mu_t)
 &= \int_M \rho_t \big( \FF_t(x) \big) \log \rho_t \big( \FF_t(x) \big) \, \det_{\meas}[\d\FF_t(x)] \,\d\meas(x) \\
&= \int_M \rho_0(x) \log\frac{\rho_0(x)}{\det_{\meas}[\d\FF_t(x)]} \,\d\meas(x) \\
&\le (1-t)\,\Ent_{\meas}(\mu_0) +t\int_M \rho_0(x) \log\frac{\rho_0(x)}{\det_{\meas}[\d\FF_1(x)]} \,\d\meas(x) \\
&\qquad\qquad -\frac{K}{2}\,t\,(1-t)\, \int_M d(x)^2 \,\d\mu_0(x) \\
&= (1-t)\,\Ent_{\meas}(\mu_0) +t\,\Ent_{\meas}(\mu_1) -\frac{K}{2}\,t\,(1-t) \int_{M \times M} l(x,y)^2 \,\d\pi(x, y).
\end{align*}

Finally, when $N=0$, \eqref{eq:c_3-} is still available, namely
\[ c_3(t) \le \frac{\mathfrak{s}_{-K}((1-t)\,d(x))}{\mathfrak{s}_{-K}(d(x))}
 +\frac{\mathfrak{s}_{-K}(t\,d(x))}{\mathfrak{s}_{-K}(d(x))}\, c_3(1). \]
In this case, we have
\begin{align*}
c_3(t) &=c_1(t)^n\, \e^{-\phi_2(t)}\\
&=\e^{\psi_{\gamma}(t)-\psi_{\gamma}(0)} \,\e^{\phi_1(t)}\, \det[\d \FF_t(x)]^{-1}\\
&=\e^{\phi_1(t)}\, \det_{\meas}[\d \FF_t(x)]^{-1}.
\end{align*}
Combining this with \eqref{eq:phi_1+} yields
\begin{align*}
\det_{\meas}[\d \FF_t(x)]
 &= \e^{\phi_1(t)}\, c_3(t)^{-1} \\
&\ge \big( (1-t)\,+t\, \e^{\phi_1(1)} \big)\, \bigg( \frac{\mathfrak{s}_{-K}((1-t)\,d(x))}{\mathfrak{s}_{-K}(d(x))}
 + \frac{\mathfrak{s}_{-K}(t\,d(x))}{\mathfrak{s}_{-K}(d(x))}\, c_3(1) \bigg)^{-1} \\
&\ge \min\!\bigg\{ \frac{(1-t)\,\mathfrak{s}_{-K}(d(x))}{\mathfrak{s}_{-K}((1-t)\,d(x))},
 \frac{t\,\mathfrak{s}_{-K}(d(x))}{\mathfrak{s}_{-K}(t\,d(x))}\,  \det_{\meas}[\d \FF_1(x)] \bigg\},
\end{align*}
where the latter inequality follows from
\[ \frac{a_1 +a_2}{b_1 +b_2} =\frac{b_1}{b_1 +b_2} \frac{a_1}{b_1} +\frac{b_2}{b_1 +b_2} \frac{a_2}{b_2} \ge \min\bigg\{ \frac{a_1}{b_1},\frac{a_2}{b_2} \bigg\} \]
for $a_1,a_2,b_1,b_2 >0$.
This implies
\begin{align*}
\rho_t\big( \FF_t(x) \big)
&=\frac{\rho_0(x)}{\det_{\meas}[\d \FF_t(x)]}\\
&\le  \max\! \bigg\{ \frac{\mathfrak{s}_{-K}((1-t)\,d)}{(1-t)\,\mathfrak{s}_{-K}(d)} \,\rho_0(x),
 \frac{\mathfrak{s}_{-K}(t\,d)}{t\,\mathfrak{s}_{-K}(d)}\, \frac{\rho_0(x)}{\det_{\meas}[\d \FF_1(x)]} \bigg\} \\
&= \max\! \bigg\{ \frac{\mathfrak{s}_{-K}((1-t)\,d)}{(1-t)\,\mathfrak{s}_{-K}(d)} \,\rho_0(x),
 \frac{\mathfrak{s}_{-K}(t\,d)}{t\,\mathfrak{s}_{-K}(d)}\, \rho_1\big( \FF_1(x) \big) \bigg\},
\end{align*}
while $\rho_t(\FF_t(x))=\rho_0(x)=\rho_1(\FF_1(x))$ for $\mu_0$-a.e.\ $x \in \{\d u=0\}$.
Therefore, we obtain \eqref{eq:TCD0}, and the proof is finished.\qedhere
\end{step}

\end{proof}

We remark that, in the Riccati equation in \cite{ohta2009} concerning the Finsler case, as $R_{\alpha\beta}$ we employed $g_{\dot{\gamma}}(R_{\dot{\gamma}}(J_{\alpha}),J_{\beta})$ instead of $g_{\dot{\gamma}}(R_{\dot{\gamma}}(E_{\alpha}),E_{\beta})$.
They both work well in the positive definite case, however, we found that the latter is more convenient in the current setting (due to the fact that $g_{\dot{\gamma}}(E_{\alpha},E_{\beta})$ is not the identity matrix).

\begin{remark}\label{Re:N<0}
For negative $N$, one can define the entropic timelike curvature-dimension condition $\smash{\TCD_q^e(K,N)}$ akin to the approach \cite{cavalletti2020}. However, in this dimensional range it seems that $\smash{\TCD_q^e(K,N)}$ is stronger than $\Ric_N \ge K$ \cite[Remark~4.16]{ohta-negative}. On the other hand, if $n \leq N <+\infty$ these properties imply their entropic counterpart by \cite{braun2022}.
\end{remark}

\subsection{Timelike Brunn--Minkowski inequality}

An important consequence of Theorem \ref{th:Ric-TCD} is the timelike Brunn--Minkowski inequality. In either dimensional case, it follows quite directly from  $\smash{\TCD_q(K,N)}$ by Jensen's inequality. 
In the range $N\in [1,\infty)$, it has been derived first in \cite{cavalletti2020} from the \emph{entropic} timelike curvature-dimension condition after \cite{EKS}, and then in \cite{braun2022} in sharp form. See e.g.~\cite{ohta2009, ohta-negative, sturmII,villani2009} for the positive signature case.

\begin{theorem}[Timelike Brunn--Minkowski inequality]\label{Th:Brunn Minkowski}
Assume that $(M,L,\meas)$ is a globally hyperbolic weighted Finsler spacetime obeying  $\smash{\Ric_N\geq K}$ for some $K\in\R$ and $N\in (-\infty,0]\cup[n,+\infty]$.
Let $A_0,A_1\subset M$ be two relatively compact Borel sets with positive $\meas$-measure with $\smash{\overline{A}_0\times \overline{A}_1 \subset \{l>0\}}$. 
Then the following hold for every $t\in (0,1)$.
\begin{enumerate}[label=\textnormal{(\roman*)}]
    \item When $N = +\infty$, we have
    \begin{align*}
    \log\meas\big[Z_t(A_0,A_1)\big] &\geq (1-t)\log\meas[A_0] + t\log\meas[A_1]\\
    &\qquad\qquad + \inf_{x\in A_0,y\in A_1}\frac{K}{2}\,t\,(1-t)\,l(x,y)^2.
    \end{align*}
    \item When $n \leq N <+\infty$, we have
    \begin{align*}
        \meas\big[Z_t(A_0,A_1)\big]^{1/N} &\geq \inf_{x \in A_0, y\in A_1}\tau_{K,N}^{(1-t)}\big(l(x,y)\big)\cdot \meas[A_0]^{1/N}\\ 
        &\qquad\qquad + \inf_{x\in A_0, y\in A_1} \tau_{K,N}^{(t)}\big(l(x,y)\big)\cdot \meas[A_1]^{1/N}.
    \end{align*}
    \item When $N  <0$, we have 
    \begin{align*}
        \meas\big[Z_t(A_0,A_1)\big]^{1/N} &\leq \sup_{x\in A_0,y\in A_1} \tau_{K,N}^{(1-t)}\big(l(x,y)\big)\cdot \meas[A_0]^{1/N}\\
        &\qquad\qquad + \sup_{x\in A_0,y\in A_1} \tau_{K,N}^{(t)}\big(l(x,y)\big)\cdot \meas[A_1]^{1/N}.
    \end{align*}
    \item When $N = 0$, we have
    \begin{align*}
    \meas\big[Z_t(A_0,A_1)\big] &\geq \min\!\bigg\lbrace\inf_{x\in A_0,y\in A_1} \frac{(1-t)\,\mathfrak{s}_{-K}(l(x,y))}{\mathfrak{s}_{-K}((1-t)\,l(x,y))}\,\meas[A_0],\\
    &\qquad\qquad \inf_{x\in A_0,y\in A_1} \frac{t\,\mathfrak{s}_{-K}(l(x,y))}{\mathfrak{s}_{-K}(t\,l(x,y))}\,\meas[A_1]\bigg\rbrace.
    \end{align*}
\end{enumerate}
\end{theorem}

\begin{proof} Let $\mu_0,\mu_1\in \Prob_\comp^\ac(M,\meas)$ be the uniform distributions of $A_0$ and $A_1$ with respect to $\meas$, respectively, i.e.~$\smash{\mu_0 := \meas[A_0]^{-1}\,\meas\vert_{A_0}}$ and $\smash{\mu_1 := \meas[A_1]^{-1}\,\meas\vert_{A_1}}$.
By our assumption on $A_0$ and $A_1$, Lemma \ref{Le:Existence separation}, and Theorem \ref{th:Ric-TCD} there exists a $q$-geodesic $(\mu_t)_{t\in [0,1]}$ from $\mu_0$ to $\mu_1$ and an $\smash{\ell_q}$-optimal coupling $\pi\in\Pi(\mu_0,\mu_1)$ witnessing the entropic displacement semi-convexity given by $\smash{\TCD_q(K,N)}$. Depending on the range of $N$, from the precise form of the latter we will derive the claims in each respective case. 

Some small preparations are in order. By Lemma \ref{Le:Existence separation}, the pair $(\mu_0,\mu_1)$ is $q$-separated. Therefore $(\mu_t)_{t\in[0,1]}$ is in fact the unique $q$-geodesic connecting $\mu_0$ to $\mu_1$ by Corollary \ref{Cor:Char geos}. By Proposition \ref{Pr:Absolute continuity}, it only consists of $\meas$-absolutely continuous measures, and we write $\mu_t = \rho_t\,\meas$. Moreover, recall that $\supp\mu_t \subset Z_t(A_0,A_1)$ by Corollary \ref{Cor:Compact support} for given $t\in (0,1)$.

(i) In the case $N=+\infty$, the hypothesis implies
\begin{align}\label{Eq:TCD for N = infty}
\begin{split}
    \Ent_\meas(\mu_t) &\leq (1-t)\,\Ent_\meas(\mu_0) + t\,\Ent_\meas(\mu_1)\\
    &\qquad\qquad - \frac{K}{2}\,t\,(1-t)\int_{M\times M} l(x,y)^2\,\d\pi(x,y).
    \end{split}
\end{align}
It is easy to see that
\begin{align*}
    \Ent_\meas(\mu_0) = -\log\meas[A_0],\qquad\Ent_\meas(\mu_1) = -\log\meas[A_1],
\end{align*}
and thus the right-hand side of \eqref{Eq:TCD for N = infty} is bounded from above by
\begin{align*}
    -(1-t)\log\meas[A_0] - t\log\meas[A_1] - \inf_{x\in A_0,y\in A_1} \frac{K}{2}\,t\,(1-t)\,l(x,y)^2.
\end{align*}
To bound the left-hand side of \eqref{Eq:TCD for N = infty} from below, by Jensen's inequality we find
\begin{align*}
    \Ent_\meas(\mu_t) &\geq -\log\meas[\supp\mu_t] \geq -\log\meas\big[Z_t(A_0,A_1)\big].
\end{align*}

(ii) In the case $n \leq N<+\infty$, the hypothesis implies
\begin{align*}
    S_\meas^N(\mu_t) &\leq -\int_{M\times M} \tau_{K,N}^{(1-t)}\big(l(x,y)\big)\,\rho_0(x)^{-1/N}\,\d\pi(x,y)\\
    &\qquad\qquad -\int_{M\times M} \tau_{K,N}^{(t)}\big(l(x,y)\big)\,\rho_1(y)^{-1/N}\,\d\pi(x,y).
\end{align*}
The right-hand side of this inequality is bounded from above by
\begin{align*}
    &-\inf_{x\in A_0,y\in A_1} \tau_{K,N}^{(1-t)}\big(l(x,y)\big)\,\meas[A_0]^{1/N} - \inf_{x\in A_0,y\in A_1} \tau_{K,N}^{(t)}\big(l(x,y)\big)\,\meas[A_1]^{1/N}.
\end{align*}
The claim follows again by applying Jensen's inequality to the convex function $\smash{r\longmapsto -r^{(N-1)/N}}$ on $[0,\infty)$, thereby obtaining
\begin{align*}
    S_\meas^N(\mu_t) \geq -\meas[\supp\mu_t]^{1/N} \geq -\meas\big[Z_t(A_0,A_1)\big]^{1/N}.
\end{align*}

(iii) This is argued analogously to the previous item and thus omitted.

(iv) In the case $N=0$, the hypothesis implies
\begin{align*}
    S_\meas^0(\mu_t) &\leq \esssup_{(x,y)\in \supp\pi} \max\!\bigg\lbrace  \frac{\mathfrak{s}_{-K}((1-t)\,l(x,y))}{(1-t)\,\mathfrak{s}_{-K}(l(x,y))}\,\rho_0(x), \frac{\mathfrak{s}_{-K}(t\,l(x,y))}{t\,\mathfrak{s}_{-K}(x,y)}\,\rho_1(y)\bigg\rbrace.
\end{align*}
The right-hand side of this inequality is bounded from above by
\begin{align*}
    &\max\!\bigg\lbrace\sup_{x\in A_0,y\in A_1}\frac{\mathfrak{s}_{-K}((1-t)\,l(x,y))}{(1-t)\,\mathfrak{s}_{-K}(l(x,y))}\,\meas[A_0]^{-1}, \\
    &\qquad\qquad\sup_{x\in A_0,y\in A_1} \frac{\mathfrak{s}_{-K}(t\,l(x,y))}{t\,\mathfrak{s}_{-K}(l(x,y))}\,\meas[A_1]^{-1}\bigg\rbrace.
\end{align*}
The claim follows by estimating
\begin{align*}
    \meas\big[Z_t(A_0,A_1)\big]^{-1} \leq \meas[\supp\mu_t]^{-1} \leq \esssup_{z\in\supp\mu_t} \rho_t(z) = S_\meas^0(\mu_t).
\end{align*}
This terminates the proof.
\end{proof}

\section[Ricci curvature bounds from curvature-dimension condition]{Ricci curvature bounds from curvature-\\dimension condition}\label{sc:TCD-Ric}

Now we turn to the converse of Theorem \ref{th:Ric-TCD}, i.e.~the derivation of the lower bound on the weighted Ricci curvature given the timelike curvature-dimension condition.
Our strategy follows the proofs of \cite[Theorem 7.3]{lott2009} and \cite[Theorem 4.3]{mondino-suhr}.

\begin{theorem}[$\smash{\TCD_q(K,N)}$ implies $\smash{\Ric_N\geq K}$]\label{Th:TCD-Ric} Let $(M,L,\meas)$ be a globally hyperbolic weighted Finsler spacetime satisfying $\smash{\TCD_q(K,N)}$ for some $q\in (0,1)$, $K\in \R$, and $N\in (-\infty,0]\cup [n,+\infty]$. Then we have $\smash{\Ric_N\geq K}$ in timelike directions.
\end{theorem}

\begin{proof} We first consider the case $n < N < +\infty$.
Let $p<0$ denote the dual exponent of $q$, i.e.~$p^{-1} + q^{-1}=1$.
Suppose to the contrary that for some $\varepsilon > 0$,  $z \in M$, and a timelike tangent vector $w \in \Omega_z$,
\begin{align*}
\Ric_N(w) < (K-2\,\varepsilon)\,F^2(w).
\end{align*}
For every sufficiently small $s>0$, we find $\delta > 0$ as well as a smooth function $\smash{u\colon B_\delta(z) \longrightarrow \R}$ defined on a Riemannian ball $\smash{B_\delta(z)}$ around $z$ such that its first and second derivatives at $z$ are prescribed as
\begin{align}\label{Eq:psi trace}
\begin{split}
F^*\big( \d u(z) \big)^{p-2}\,\mathscr{L}^*\big( \d u(z) \big) &= s\,w,\\
F^*\big( \d u (z) \big)^{p-2}\, Q\,H^u\,Q &=-\frac{\psi'_{\gamma}(0)}{N-n}\,I_n,
\end{split}
\end{align}
where we define $\gamma(t) := \exp_z(t\,s\,w)$ for $t \in [0,1]$, $Q$ and $H^u$ are the matrices defined in the proof of Theorem \ref{th:Ric-TCD}, and $I_n$ is the identity matrix.
We choose $s$ and $\delta$ such that $\FF_t\colon B_{\delta}(z) \longrightarrow M$, where 
\begin{align*}
    \FF_t(x) := \exp_x\!\big(t\,F^*(\d u)^{p-2}\,\mathscr{L}^*(\d u)\big),
\end{align*}
becomes a diffeomorphism onto its image for every $t \in [0,1]$.
By \eqref{Eq:psi trace} and \cite[Lemma 8.3]{Mc}, up to reducing $s$ and $\delta$ we may and will assume $u$ to be $(\smash{l^q/q})$-convex, so that the above map will eventually give an optimal transport. Lastly, by the continuity of $u$ and possibly reducing $\delta$ further, we may and will also assume that the time derivative of the curve $t \longmapsto \FF_t(x)$ is timelike and lies uniformly away from the light cone in $x\in B_\delta(z)$ and $t \in [0,1]$.

Define the measures $\smash{\mu_0 := \meas[B_\delta
(z)]^{-1}\,\meas\vert_{B_\delta(z)} = \rho_0\,\meas}$ and $\mu_t := (\FF_t)_\sharp\mu_0$ for $t\in[0,1]$.
By further reducing $\delta$ if necessary, we achieve $\supp\mu_0 \times\supp\mu_1 \subset \{l>0\}$. In particular,  $(\mu_0,\mu_1)$ is $q$-separated by Lemma \ref{Le:Existence separation} and the hypothesized $(l^q/q)$-convexity of $u$, and $(\mu_t)_{t\in[0,1]}$ constitutes the unique $q$-geodesic connecting $\mu_0$ to $\mu_1$ by Corollary \ref{Cor:Char geos}.
After rescaling,  $\mu_t$ is $\meas$-absolutely continuous for every $t\in[0,1]$ by Proposition \ref{Pr:Absolute continuity}.

Along $\gamma(t)=\exp_z(t\, s\, w)$ as above, we consider the decomposition
\begin{align*}
     c(t) := \det_\meas[\d\FF_t(z)]^{1/N} = c_1(t)^{(N-n)/N}\,c_2(t)^{n/N}
\end{align*}
as in the proof of Theorem \ref{th:Ric-TCD}, where
\[
c_1(t) := \exp \bigg(\frac{\psi_\gamma(0)- \psi_\gamma(t)}{N-n} \bigg),\qquad
c_2(t) := \e^{\phi(t)/n} =\det[\d\FF_t(z)]^{1/n}.
\]
Then we have, by $\phi'=\trace[B]$ and the Riccati equation \eqref{eq:trace-Riccati},
\begin{align*}
    N\,\frac{c''}{c} &= (N-n)\,\frac{c_1''}{c_1} + n\,\frac{c_2''}{c_2} - \frac{n(N-n)}{N}\,\bigg(\frac{c_1'}{c_1} - \frac{c_2'}{c_2}\bigg)^2\\
    &= -\Ric_N(\dot{\gamma}) -\trace[B^2] +\frac{(\trace[B])^2}{n}  - \frac{n(N-n)}{N}\,\bigg(\frac{c_1'}{c_1} - \frac{c_2'}{c_2}\bigg)^2.
\end{align*}
At $t=0$, recalling $B(0)=F^*(\d u(z))^{p-2}\,H^u\, Q^2$,
we deduce from \eqref{Eq:psi trace} that
\[
    \frac{c_1'(0)}{c_1(0)} = -\frac{\psi'_{\gamma}(0)}{N-n} = \frac{1}{n}\trace[B(0)] = \frac{c_2'(0)}{c_2(0)}
\]
and, since $H^u$ is symmetric,
\[
\trace[B(0)^2] =\frac{(\trace[B(0)])^2}{n}
=\frac{n\, \psi'_{\gamma}(0)^2}{(N-n)^2}.
\]
Hence, we obtain
\begin{align*}
    N\,\frac{c''(0)}{c(0)} = -\Ric_N(s\, w) > -(K-2\,\varepsilon)\,F^2(s\, w).
\end{align*}
Thus, by continuity and a further rescaling, we may and will assume
\begin{align*}
    N\,\frac{c''(t)}{c(t)}
    > -(K-\varepsilon ) \, F^2\big(\dot\FF_t(x)\big)
\end{align*}
for every $x\in B_\delta(z)$ and every $t\in[0,1]$.
This inequality implies, by putting $d(x):=F(\dot{\mathcal{F}}_t(x))$,
 \begin{equation}\label{Eq:Renyi comp}
 c(t) \le \frac{\mathfrak{s}_{(K-\varepsilon)/N}((1-t)\,d(x))}{\mathfrak{s}_{(K-\varepsilon)/N}(d(x))} \,c(0) +\frac{\mathfrak{s}_{(K-\varepsilon)/N}(t\,d(x))}{\mathfrak{s}_{(K-\varepsilon)/N}(d(x))}\, c(1)
 \end{equation}
thanks to (an analogue of) \cite[Lemma 2.2]{EKS}.
Then it follows from the Monge--Amp\`ere equation (Corollary \ref{Cor:Monge Ampère}) that
\begin{align*}
 \rho_t \big( \mathcal{F}_t(x) \big)^{-1/N} &\leq  \frac{\mathfrak{s}_{(K-\varepsilon)/N}((1-t)\,d(x))}{\mathfrak{s}_{(K-\varepsilon)/N}(d(x))}\, \rho_0(x)^{-1/N}\\
 &\qquad\qquad   +\frac{\mathfrak{s}_{(K-\varepsilon)/N}(t\,d(x))}{\mathfrak{s}_{(K-\varepsilon)/N}(d(x))}\, \rho_1 \big( \mathcal{F}_1(x) \big)^{-1/N}.
\end{align*}
Therefore, we find from \cite[Theorem~1.5(v)]{braun2022} the failure of the \emph{reduced} timelike curvature-dimension condition $\TCD_q^*(K,N)$, which is weaker than $\TCD_q(K,N)$ (see \cite[Proposition 3.6]{braun2022}).

The case $N=n$ can be reduced to the above argument by employing the smoothness of $(M,L,\meas)$.
Precisely, if $\Ric_n(w) <(K-2\,\varepsilon)F^2(w)$, then we have $\Ric_N(w) <(K-\varepsilon)F^2(w)$ for sufficiently small $N>n$.
This implies that $\TCD_q^*(K,N)$ fails, and hence the stronger condition $\TCD_q^*(K,n)$ does not hold as well (see \cite[Proposition~3.7]{braun2022}).

The case $N<0$ is also argued as in \eqref{Eq:Renyi comp}.
We have
\[
 c(t) \ge \frac{\mathfrak{s}_{(K-\varepsilon)/N}((1-t)\,d(x))}{\mathfrak{s}_{(K-\varepsilon)/N}(d(x))} \,c(0) +\frac{\mathfrak{s}_{(K-\varepsilon)/N}(t\,d(x))}{\mathfrak{s}_{(K-\varepsilon)/N}(d(x))}\, c(1),
\]
and hence
\begin{align*}
 \rho_t \big( \mathcal{F}_t(x) \big)^{-1/N} &\leq \frac{\mathfrak{s}_{(K-\varepsilon)/N}((1-t)\,d(x))}{\mathfrak{s}_{(K-\varepsilon)/N}(d(x))} \,\rho_0(x)^{-1/N}\\
 &\qquad\qquad   +\frac{\mathfrak{s}_{(K-\varepsilon)/N}(t\,d(x))}{\mathfrak{s}_{(K-\varepsilon)/N}(d(x))}\, \rho_1 \big( \mathcal{F}_1(x) \big)^{-1/N}.
\end{align*}
Recalling that the $N$-Rényi entropy has no minus sign for $N<0$, we deduce that $\TCD_q^*(K,N)$ fails, which is again weaker than $\TCD_q(K,N)$ (which is argued as for \cite[Proposition 4.7]{ohta-negative}).


In the case $N=+\infty$, we replace the second condition of \eqref{Eq:psi trace} with $H^u=0$.
Then we have, as in the proof of Theorem \ref{th:Ric-TCD},
\begin{align*}
\frac{\d^2}{\d t^2}\bigg\vert_0 \log\det_{\meas}[\d\FF_t(z)] 
&= \frac{\d^2}{\d t^2}\bigg\vert_0 \big[ \psi_{\gamma}(0) -\psi_{\gamma}(t) +\phi(t) \big] \\
&= -\Ric_{\infty}(s\, w) >-(K-2\,\varepsilon) F^2(s\, w),
\end{align*}
with which one can show the failure of $\TCD_q(K,\infty)$ in the same way.

The case $N=0$ uses the hierarchy from Remark \ref{rm:Ric_N}.
Upper bounds on $\Ric_0$ yield the same upper bounds on $\Ric_N$ for every $N<0$, and following the above arguments we have
\begin{align*}
 \rho_t \big( \mathcal{F}_t(x) \big)&\geq  \bigg( \frac{\mathfrak{s}_{(K-\varepsilon)/N}((1-t)\,d(x))}{\mathfrak{s}_{(K-\varepsilon)/N}(d(x))} \,\rho_0(x)^{-1/N}\\
 &\qquad\qquad   +\frac{\mathfrak{s}_{(K-\varepsilon)/N}(t\,d(x))}{\mathfrak{s}_{(K-\varepsilon)/N}(d(x))}\, \rho_1 \big( \mathcal{F}_1(x) \big)^{-1/N} \bigg)^{-N}.
\end{align*}
By virtue of \cite[Proposition 4.7]{ohta-negative}, this implies
\begin{align*}
 \rho_t \big( \mathcal{F}_t(x) \big) &\ge \Big( \tau^{(1-t)}_{K-\varepsilon,N}\big( d(x) \big)\, \rho_0(x)^{-1/N}
  +\tau^{(t)}_{K-\varepsilon,N} \big( d(x) \big)\, \rho_1 \big( \mathcal{F}_1(x) \big)^{-1/N} \Big)^{-N}.
\end{align*}
Therefore, letting $N \to 0$ (and therefore $-1/N \to +\infty$), we obtain
\[
\rho_t \big( \mathcal{F}_t(x) \big)
\ge \max\!\bigg\{ \frac{\mathfrak{s}_{-K+\varepsilon}((1-t)\,d(x))}{(1-t)\,\mathfrak{s}_{-K+\varepsilon}(d(x))}\, \rho_0(x),
\frac{\mathfrak{s}_{-K+\varepsilon}(t\,d(x))}{t\,\mathfrak{s}_{-K+\varepsilon}(d(x))} \,\rho_1 \big( \mathcal{F}_1(x) \big) \bigg\},
\]
which shows the failure of  $\smash{\TCD_q(K,0)}$.
\end{proof}

\end{document}